\documentclass{amsart}
\usepackage{amsmath,amssymb,mathrsfs,yhmath,longtable,mathtools,hyperref,enumitem}
\usepackage{amsthm,rotating,color,epsfig,mdframed,url,bbm}

\usepackage[utf8]{inputenc}

\newtheorem{prop}{Proposition}[section]
\newtheorem{coro}[prop]{Corollary}

\newtheorem{lemm}[prop]{Lemma}

\newtheorem{clai}[prop]{Claim}
\newtheorem{theo}[prop]{Theorem}

\theoremstyle{definition}
\newtheorem{defi}[prop]{Definition}
\newtheorem{exam}[prop]{Example}

\newtheorem{conv}[prop]{Convention}
\newtheorem{rema}[prop]{Remark}

\usepackage{graphicx} 
\graphicspath{ {pic/} }

\thanks{The work is supported by the Russian Science Foundation under grant~22-11-00299.}
\title{Legendrian Lavrentiev links}
\keywords{Legendrian links, smoothing, bi-Lipschitz maps}
\subjclass{57K33}
\author{Maxim Prasolov}
\address{
Steklov Mathematical Institute of Russian Academy of Sciences, 8 Gubkina Str., Moscow 119991, Russia\\
\indent Moscow Center of Fundamental and Applied Mathematics, Moscow, Russia}
\email{0x00002a@gmail.com}
\date{}

\begin{document}

\begin{abstract}
Lavrentiev curves form a special class of rectifiable curves which includes cusp-free piecewise smooth curves. We call a Lavrentiev curve Legendrian if the integral of the contact form equals zero on any its subarc. We define Legendrian isotopies of such curves and prove that the equivalence classes of Legendrian Lavrentiev links with respect to Legendrian isotopies coincide with smooth classes.
\end{abstract}

\maketitle

\tableofcontents

\section{Introduction}
In contact topology sometimes one needs to deal with piecewise smooth objects. It can be a triangulation (see \cite{CGH}) or something obtained from it. For example, contact cell decompositions (see \cite{Gir02}) or Legendrian links associated with a rectangular diagram (see \cite{dp17}). Another examples are Legendrian graphs (see \cite{BI}) and bypasses for surfaces (see \cite{honda}) or for Legendrian links (see \cite{DyPr}). To deform such objects one needs more general category than the smooth one. In the non-contact setting one usually uses the topological one. Let us discuss what category can be used in the contact setting.

Some work is done in the topological category. Let us call a homeomorphism contact if it is the $C^0$-limit of contact diffeomorphisms. In~\cite{MuSpa14} and~\cite{Mu19} it is proved that for diffeomorphisms this definition agrees with the classical one, i.e. that a diffeomorphism is the $C^0$-limit of contact diffeomorphisms if and only if it is contact. Let us call a (topological) submanifold Legendrian if it is the image under a contact homeomorphism of a smooth Legendrian submanifold. In the recent preprint~\cite{DimSull} it is proved that for smooth submanifolds this definition is equivalent to the classical one. This is very interesting approach, but we will not touch it in this paper.

There is no such thing as the piecewise smooth category because even the composition of two piecewise smooth functions on a segment can be not piecewise smooth.

There is one more common category which supports piecewise smooth objects, this is the Lipschitz one. We are going to develop this approach in dimension 3. Isomorphisms in Lipschitz category are bi-Lipschitz homeomorphisms, while curves equivalent to smooth ones are Lavrentiev (also called chord-arc curves in some sources). Since Lavrentiev curves are rectifiable it is natural to call such curve Legendrian if the integral of the contact form on any its subarc is zero. We call a bi-Lipschitz homeomorphism contact if it preserves the class of Legendrian Lavrentiev curves.

We define a Legendrian isotopy (see Definition~\ref{def-projection}) as a generalization of the Legendrian isotopy for piecewise smooth links given in~\cite{dp2?}. The idea behind this definition is the following smoothing procedure. Let $L$ be a piecewise smooth Legendrian knot in $(\mathbb R^3, dz + xdy)$ such that its orthogonal projection to the plane $z=0$ has a finite number of self-intersections. Consider any smooth Legendrian knot sufficiently $C^0$-close to $L$ such that its projection has the same number of self-intersections as the projection of $L$. Then the Legendrian type of this smooth knot depends only on $L$. A similar procedure can be found in~\cite[1.4.3 Elliptic pivot lemma]{EF1}, when smoothing the union of two leaves of the contact foliation on a surface, and in~\cite{OP}, when smoothing paths in Legendrian graphs.

Our main result is the following.

\begin{theo}\label{main-theorem}
Any Legendrian Lavrentiev link is Legendrian isotopic to a smooth Legendrian link. If two smooth Legendrian links are Legendrian isotopic as Lavrentiev links, they are smoothly Legendrian isotopic.
\end{theo}

Using this theorem we can associate canonically an equivalence class of smooth Legendrian links with a given Legendrian Lavrentiev link. 

The proof of Theorem~\ref{main-theorem} occupies Section~\ref{Legendrian-smoothing-section}. Let us discuss the plan of the proof of the first part of the theorem, i. e. that a Legendrian curve admits a smoothing. For simplicity we now assume that the contact structure is coorientable. In Section~\ref{regular-neighborhood-section} we define and prove the existence of a regular neighborhood of a Legendrian curve. This allows us to reduce the problem to the case when the ambient manifold is the solid torus $\{(x,y,z): x\in\mathbb S^1, y\in (-1;1), z\in\mathbb R\}$ with a contact structure $\ker\left(dz + a(x,y)dx + b(x,y)dy\right)$ where $a$ and $b$ are some functions, and the orthogonal projection to the plane $z=0$ of the Legendrian curve is injective. Then using results of Subsection~\ref{regular-projection-section} on the regular projection we reformulate the problem in terms of Lavrentiev curves on the plane. In Section~\ref{Lavrentiev-smoothing-section} we make the final step by using Tukia's theorem to smooth Lavrentiev curves on the plane. We prove the relative variant of this part of the theorem, so that some portion of the link is being smoothed while the other part stays fixed. We also prove that the isotopy can be made Lipschitz.

We use several definitions of Legendrian curves (see definitions~\ref{def-angle} and~\ref{def-projection}) and we prove in propositions~\ref{leg-integral} and~\ref{leg-criteria1} that they are equivalent.

In Section~\ref{Continuous-Legendrian-curves-section} using Definition~\ref{def-projection} we try to define continuous Legendrian curves and we conclude that Theorem~\ref{main-theorem} is wrong in this case.

In Section~\ref{contactomorphisms-section} we define bi-Lipschitz contactomorphisms, give some examples and discuss some basic facts about them. We will continue this work in the future.

Some results on the symplectic side of the Lipschitz category can be found in~\cite{humi} and~\cite{joks}.

\section*{Acknowledgments}
Ivan Dynnikov introduced the problem to me and proposed Definition~\ref{def-projection} -- the key definition of the paper. I am grateful to Ivan Dynnikov and Vladimir Shastin for valuable discussions and continuous support.

\section{Lavrentiev curves}
In this section we remind some basic facts about Lavrentiev curves and bi-Lipschitz maps and prove some technical lemmas needed for the main result. For simplicity we consider embedded curves only.

\begin{defi}\label{Lavrentiev-definition}
A subset of a topological space is called {\it an embedded curve} if it is homeomorphic to a connected subset of a circle.
{\it A length of a curve} $L$, embedded in a metric space, is a finite or infinite number
$$\ell(L) = \sup\limits_{L = \bigcup\limits_i \gamma_i}\sum\limits_i \mathrm{diam} (\partial \gamma_i),$$
where supremum is taken over all unions  $L = \bigcup\limits_i \gamma_i$ of compact subarcs $\gamma_i$ which have no interior points in common and $\mathrm{diam} (\partial \gamma_i)$ is the distance between the endpoints of $\gamma_i.$ An embedded curve is called {\it rectifiable} if the length of any its compact subarc is finite.

We write $L\big|_P^Q$ for any shortest subarc of the curve $L$ which joins two points $P$ and $Q$. We orient this subarc from $P$ to $Q.$

Let $C$ be a number. A rectifiable curve $L$ is called {\it$C$-Lavrentiev} if for any two points $P,Q\in L$
$$
\ell(L\big|_P^Q)\le C\cdot \mathrm{dist}(P,Q),
$$
where $\mathrm{dist}$ denotes the distance.

A curve is called {\it Lavrentiev} if it is $C$-Lavrentiev for some number $C$. An embedded curve $L$ is called {\it locally Lavrentiev} if for any point $p\in L$ there exists a neighborhood $U\ni p$ such that $L\cap U$ is a Lavrentiev curve.
\end{defi}

The notion of a Lavrentiev curve can be found in \cite[Theorem 8 2)]{Lavr}. Definition 2.29 in~\cite{LuukkVai} of a quasiconvex metric space generalizes the concept of a Lavrentiev curve.

Smooth curves and cusp-free piecewise smooth curves are Lavrentiev by Lemma~\ref{Lavrentiev-example}. Logarithmic spirals are also allowed as a partial case of the following example.

\begin{exam}\label{Lavrentiev-example-2}
Let $\alpha$ and $\beta$ be two non-intersecting embedded smooth subarcs of the annulus $1 < |z| < 2$ on the complex plane which are orthogonal to the boundary of the annulus at their endpoints and $\partial \alpha = \{1, 2\},$ $\partial \beta = \{-1, -2\}$. Let $\phi: z\mapsto 2z$ be a homothety. Then $\overline{\bigcup\limits_{k\in\mathbb Z} \phi^{k}\left(\alpha\cup\beta\right)}$ is a Lavrentiev curve.
\end{exam}

A similar example in dimension 3 can be constructed to obtain a wild Lavrentiev curve.

\begin{defi}
Let $X,Y$ be metric spaces and $C$ be a number. A map $f:X\to Y$ is called {\it $C$-bi-Lipschitz} if for any $p,q\in X$ 
$$\frac1C\cdot\mathrm{dist}(p,q)\le\mathrm{dist}(f(p),f(q))\le C\cdot\mathrm{dist}(p,q).$$

A map $f:X\to Y$ is called {\it $C$-Lipschitz}, if for any $p,q\in X$  we have
$\mathrm{dist}(f(p),f(q))\le C\cdot\mathrm{dist}(p,q).$ 

A map is called {\it (bi-)Lipschitz} if it is $C$-(bi-)Lipschitz for some $C$. A map $f:X\to Y$ is called {\it locally (bi-)Lipschitz} if $X$ admits an open cover $X = \bigcup\limits_{\alpha}U_{\alpha}$ such that each map $f\big|_{U_{\alpha}}$ is (bi-)Lipschitz.
\end{defi}

\begin{rema}\label{bi-Lipschitz-curves}
It is clear that a natural parametrization of a $C$-Lavrentiev curve is a $C$-bi-Lipschitz map and that the image of any $C$-bi-Lipschitz map of a segment is a $C^2$-Lavrentiev curve. A ($C$-)bi-Lipschitz parametrization of a Lavrentiev curve will be called {\it ($C$-)bi-Lipschitz curve}. We do not know is it true that any continuous family of Lavrentiev curves is a set of the images of a continuous isotopy of bi-Lipschitz curves.
\end{rema}

\begin{lemm}\label{bi-Lipschitz-family}
Let $X,Y,T$ be metric spaces, $X,T$ be compact, $f:X\times T\to Y$ be a continuous map such that for any $t\in T$ the map $f_t:X\to Y$ is an embedding and for any $x_0\in X$ and $t_0\in T$ there exist their neighborhoods $U\subset X$ and $\mathcal I\subset T$ and a real number $C_0$ such that for any $t\in \mathcal I$ the map $f_t\big|_{U}$ is $C_0$-bi-Lipschitz. Then there exists a number $C$ such that the map $f_t:X\to Y$ is $C$-bi-Lipschitz for any $t\in T.$
\end{lemm}

\begin{proof}
Suppose that for any $k\in\mathbb N$ there exists $t\in T$ such that the map $f(\bullet, t)$ is not $k$-Lipschitz. Then there exist sequences $x_{0,k}\in X,$ $x_{1,k}\in X$ and $t_k\in T$ such that 

\begin{equation}\label{bi-Lipschitz-family-bound}
\mathrm{dist}(f(x_{0,k}, t_k), f(x_{1,k}, t_k)) > k\cdot\mathrm{dist}(x_{0,k}, x_{1,k}).
\end{equation}

By compactness, we can assume that the following limits exist: $\lim\limits_{k\to+\infty} x_{0,k} = x_0,$ $\lim\limits_{k\to+\infty} x_{1,k} = x_1$ and $\lim\limits_{k\to+\infty} t_{k} = t_0.$

If $x_0 \neq x_1$ then, taking the limit in the inequality~(\ref{bi-Lipschitz-family-bound}) as $k\to\infty$ we obtain $\mathrm{dist}(f(x_0,t_0),f(x_1,t_0)) = +\infty,$ a contradiction.

If $x_0 = x_1$, the inequality~(\ref{bi-Lipschitz-family-bound}) contradicts the existence of neighborhoods $U$ and $\mathcal I$ for the point $x_0$ and the parameter $t_0.$

Suppose then that there exist sequences $x_{0,k}\in X,$ $x_{1,k}\in X$ and $t_k\in T$ such that

$$\mathrm{dist}(f(x_{0,k}, t_k), f(x_{1,k}, t_k)) < \mathrm{dist}(x_{0,k}, x_{1,k})/k.$$

Again we can assume that the limits exist and we use the same notations. If $x_0 \neq x_1$, the limit of the inequality contradicts the fact that $f(\bullet, t_0)$ is an embedding. If $x_0 = x_1,$ the inequality contradicts the assumption of the lemma on the existence of neighborhoods $U$ and $\mathcal I$ for the pair $(x_0, t_0).$
\end{proof}

\begin{lemm}\label{compact-bi-Lipschitz}
A locally bi-Lipschitz embedding $f:X\to Y$ is bi-Lipschitz if $X$ is compact.
\end{lemm}

\begin{proof}
Follows from the preceding lemma for $T$ being a point.
\end{proof}

\begin{lemm}[a particular case of Lemma 2.33 from \cite{LuukkVai}]\label{compact-Lavrentiev}
Any locally Lavrentiev embedded compact curve is Lavrentiev.
\end{lemm}

\begin{proof}
Since the curve is locally Lavrentiev, it is covered by open sets such that the intersection of each open set with the curve is a Lavrentiev curve. Therefore a natural parametrization of the curve is locally bi-Lipschitz and thus by Lemma~\ref{compact-bi-Lipschitz}  it is bi-Lipschitz.
\end{proof}

\subsection{Some constructions of Lavrentiev curves on manifolds}

\begin{defi}
Two metrics $d_1$ and $d_2$ on a set $X$ are {\it comparable} if there exists a number $C$ such that for any $p,q\in X$
$$
\frac1C\cdot d_1(p,q)\le d_2(p,q) \le C\cdot d_1(p,q).
$$
\end{defi}

It is clear that classes of (locally) Lavrentiev curves and (locally) (bi)-Lipschitz maps with respect to comparable metrics coincide.

\begin{conv}\label{Lavrentiev-metric-equivalence}
The restrictions of any two Riemannian metrics to a compact subset of a smooth manifold are comparable. So classes of locally Lavrentiev or compact Lavrentiev curves on a smooth manifold do not depend on the choice of a Riemannian metric. By a locally Lavrentiev or compact Lavrentiev curve on a smooth manifold we will mean such a curve with respect to some (any) smooth Riemannian metric. The same holds for locally (bi-)Lipschitz maps and (bi-)Lipschitz maps from a compact manifold.
\end{conv}

The following lemma allows one to concatenate two Lavrentiev curves provided that the minimal angle (see Definition~\ref{minimal-angle-definition}) at their common endpoint is nonzero.

\begin{lemm}\label{Lavrentiev-concatenation}
Suppose that $L_1$ and $L_2$ are respectively $C_1$- and $C_2$-Lavrentiev arcs in Euclidean space, $\partial L_1 = \{P, Q_1\},$ $\partial L_2 = \{P,Q_2\},$ $L_1\cap L_2 = \{P\}$ and $\angle Q_1 P Q_2 > \alpha.$
Then
$$
\ell(L)\le \dfrac{C_1+C_2}{\sin(\alpha/2)}\mathrm{dist}(Q_1,Q_2).
$$

\end{lemm}

\begin{proof}
Since the distances from the points $Q_1$ and $Q_2$ to a bisector of the angle $\angle Q_1 P Q_2$ are less than $\mathrm{dist}(Q_1,Q_2)$

\begin{multline*}
\ell(L) = \ell(L_1)+\ell(L_2)\le C_1 \mathrm{dist}(P,Q_1) + C_2 \mathrm{dist}(P,Q_2) < \frac{C_1 \mathrm{dist}(Q_1,Q_2)}{\sin(\alpha/2)} + \frac{C_2 \mathrm{dist}(Q_1,Q_2)}{\sin(\alpha/2)} = \\ = \dfrac{(C_1+C_2)\mathrm{dist}(Q_1,Q_2)}{\sin(\alpha/2)}.
\end{multline*}

\end{proof}

\begin{defi}\label{minimal-angle-definition}
By a {\it minimal angle} at the common point $P$ of two curves $L_1$ and $L_2$ we mean $\lim\limits_{\varepsilon\to+0}\inf\limits_{\mathrm{dist}(Q_i, P)<\varepsilon}\angle Q_1 P Q_2,$ where $Q_i\in L_i\setminus\{P\}$ for $i = 1,2.$

By a {\it minimal angle} at a point $P$ of a curve $L$ we mean the minimal angle at $P$ between two subarcs $L_1$ and $L_2$ of $L$ such that $L_1\cap L_2=\{P\}.$
\end{defi}

\begin{coro}\label{Lavrentiev-concatenation2}
Let $\{L_i\}_{i=1}^n$ be a set of Lavrentiev curves in the Riemannian manifold, $L = \bigcup\limits_{i=1}^N L_i$ be a compact curve and $L_i\cap L_j\subset \partial L_i$ for any $i\neq j.$ Suppose that at any interior point $p$ of the curve $L$ such that $p\in\partial L_i$ for some $i = 1,\dots, n$ the minimal angle  is nonzero. Then $L$ is Lavrentiev.
\end{coro}

\begin{proof}
Since the curve is compact, by Lemma~\ref{compact-Lavrentiev} it suffices to check the Lavrentiev condition locally.
At any point of $L_i\setminus \partial L_i$ for every $i = 1,\dots, N$ it is satisfied since $L_i$ is Lavrentiev. At any point of $\partial L_i$ for any $i = 1,\dots,N$ it is satisfied by Lemma~\ref{Lavrentiev-concatenation}.
\end{proof}

Suppose that at each breaking point of a piecewise smooth curve the angle between the one-sided tangent lines to the curve is non-zero. We call such a curve  {\it cusp-free}.

\begin{lemm}\label{Lavrentiev-example}
Compact cusp-free piecewise smooth embedded curve in a manifold is Lavrentiev.
\end{lemm}

\begin{proof}
Let us check that a compact smooth arc is Lavrentiev. Since at each its point we can choose local Euclidean coordinates such that this arc is a straight-line segment, this piece of the arc is a 1-Lavrentiev curve with respect to Euclidean metric. Therefore this curve is locally Lavrentiev with respect to any Riemannian metric (see Convention~\ref{Lavrentiev-metric-equivalence}). Since the arc is compact, it is Lavrentiev by Lemma~\ref{compact-Lavrentiev}. The rest follows from Corollary~\ref{Lavrentiev-concatenation2}.
\end{proof}

The following lemma has a corollary: any Lavrentiev piecewise smooth curve is cusp-free.

\begin{lemm}\label{Lavrentiev-semitangent}
Let $L_1$ and $L_2$ be two arcs in Euclidean space with common endpoint $P$. Let $L_1\cup L_2$ be a $C$-Lavrentiev curve and $L_1$ have the one-sided tangent at the point $P$. Then the minimal angle at $P$ of the curve $L_1\cup L_2$ is not less than $2\arcsin\left(\frac1C\right)$.
\end{lemm}

\begin{proof}
Let $r$ be such number that for any $r'\in[0;r]$ there exists a point $Q_1\in L_1$ such that $\mathrm{dist}(P,Q_1) = r'.$
Let $B(r)$ be a ball of radius $r$ centered at $P$.

Consider any point $Q_2\in B(r)\cap L_2\setminus\{P\}.$ There exists such point $Q_1\in L_1$ that $\mathrm{dist}(P,Q_1) = \mathrm{dist}(P, Q_2).$
Let $\gamma$ be a subarc of $L$ which connects the points $Q_1$ and $Q_2$ and which contains $P$. Then

$$
\ell(\gamma) \ge \mathrm{dist}(P,Q_1) + \mathrm{dist} (P,Q_2) = 2\mathrm{dist}(P,Q_1)
$$

and

$$
\ell(\gamma) \le C\cdot \mathrm{dist}(Q_1, Q_2) =C\cdot 2\sin\left(\frac12\angle Q_1 P Q_2\right)\mathrm{dist}(P,Q_1).
$$

Therefore $\angle Q_1 P Q_2 \ge 2\arcsin \frac1C.$ Since the angle between the vector $\stackrel{\longrightarrow}{P Q_1}$ and the tangent ray to the curve $L_1$ at the point $P$ tends to zero as $r\to 0$, the minimal angle at $P$ between the tangent ray to the curve $L_1$ at $P$ and the arc $L_2$ is not less than $2\arcsin \frac1C.$ The claim follows.
\end{proof}

\begin{lemm}\label{Lavrentiev-covering}
Let $\pi:\widetilde M \to M$ be a covering of a smooth manifold and $L\subset M$ be a locally Lavrentiev curve. Then $\pi^{-1}(L)$ is also a locally Lavrentiev curve.
\end{lemm}

\begin{proof}
Let $g$ be an arbitrary smooth Riemannian metric on the manifold $M$. Since $\pi: (\widetilde M, \pi^*g)\to (M,g)$ is a local isometry and $L$ is locally Lavrentiev, then $\pi^{-1}(L)$ can be covered by Lavrentiev subarcs which are open in $\pi^{-1}(L).$ This means that the curve $\pi^{-1}(L)$ is locally Lavrentiev.
\end{proof}

\begin{lemm}\label{Lavrentiev-C1-isotopy-nonclosed}
Let a map $F(s,t): [0;1]\times[0;1]\to \mathbb R^n$ have a continuous derivative $\frac{\partial F}{\partial s}(s, t)$, which is nonzero at every point, and $F(\bullet,t)$ be injective for all $t\in[0;1].$ Then there exists such number $C$ that for all $t\in[0;1]$ the map $F(\bullet, t)$ is $C$-bi-Lipschitz.
\end{lemm}

\begin{proof}
Let $M$ be the maximum of the length of the vector $\frac{\partial F}{\partial s}(s, t)$. Then for any $s_0\in[0;1]$, $s_1\in(s_0;1]$  and $t\in[0;1]$

$$
\ell(F([s_0;s_1]\times\{t\}))\le M\cdot(s_1-s_0).
$$

Thus

$$
\mathrm{dist}(F(s_0, t), F(s_1, t))\le M\cdot(s_1-s_0).
$$

Therefore for all $t\in[0;1]$ the map $F(\bullet, t)$ is $M$-Lipschitz.

Let $m$ be a minimum of the length of the vector $\frac{\partial F}{\partial s}(s, t)$. Since this vector is everywhere nonzero, $m>0.$

Let $s_0\in [0;1],$ $t_0\in[0;1].$ We choose such neighborhood $U$ of the parameter $s_0$ and such neighborhood $\mathcal I$ of the moment $t_0$ that for any $s\in U$ and any $t\in \mathcal I$

$$
\angle\left(\frac{\partial F}{\partial s}(s, t), \frac{\partial F}{\partial s}(s_0, t_0)\right) < \frac{\pi}4.
$$

Such neighborhoods exist since the velocity vector $\frac{\partial F}{\partial s}(s, t)$ is everywhere nonzero and depends on $s$ and $t$ continuously. Let $\mathrm{pr}$ be the orthogonal projection to the tangent line at the point $F(s_0, t_0)$ of the curve $F(\bullet, t_0).$ Then for any $s', s''\in U$ and any $t\in \mathcal I$ 

$$
\mathrm{dist}(F(s',t), F(s'', t)) \ge \mathrm{dist}\left(\mathrm{pr}\left(F(s',t)\right), \mathrm{pr}\left(F(s'', t)\right)\right) = \int\limits_{s'}^{s''} \frac{d}{ds}(\mathrm{pr}\circ F)(s,t) ds > \cos(\pi/4)\cdot |s'' - s'|\cdot m.
$$

Therefore for any $t\in\mathcal I$ the map $F(\bullet ,t)\big|_{U}$ is $\max(M, 1/(m\cos(\pi/4)))$-bi-Lipschitz. By Lemma~\ref{bi-Lipschitz-family} for $T = [0;1]$ we are done.
\end{proof}

\begin{lemm}\label{Lavrentiev-C1-isotopy}
Let $\{L_t\}_{t\in[0;1]}$ be an isotopy of $C^1$-smooth compact curves embedded in a Riemannian manifold. Suppose that the isotopy is continuous in the $C^1$-topology. Then for some numbers $C$ and $C'$ the isotopy $\{L_t\}_{t\in[0;1]}$ is an isotopy of $C$-bi-Lipschitz curves and for all $t\in[0;1]$ the curve $L_t$ is $C'$-Lavrentiev.
\end{lemm}

\begin{proof}
The uniform Lavrentiev condition follows from the uniform bi-Lipschitz condition, so we prove the latter.

Since the curves are compact, by Lemma~\ref{bi-Lipschitz-family} it is sufficient to prove the statement locally both on the point on the curve and the parameter of the isotopy. So we can assume that the Riemannian manifold is the Euclidean space $\mathbb R^n$ and the curves are non-closed. So we are done by Lemma~\ref{Lavrentiev-C1-isotopy-nonclosed}.
\end{proof}

\subsection{Bypasses for Lavrentiev curves on surfaces}
Here we introduce bypasses for Lavrentiev curves on surfaces and associate with them isotopies of Lavrentiev curves. Such isotopies are Lipschitz and for any such isotopy there exists a number $C$ such that this isotopy is a continuous isotopy of $C$-bi-Lipschitz curves.

$\mathbb C$ denotes a complex plane with the standard Euclidean metric, $\mathbb D\subset\mathbb C$ is an interior of a unit disk centered at zero and $\overline{\mathbb D}$ is its closure, 
$\partial_{\pm}\mathbb D = \{z\in\partial\mathbb D:\ \pm\mathrm{Im} z\ge 0\}.$

\begin{defi}
Let $L$ be a Lavrentiev curve on a surface $S$. A map $\chi:\overline{\mathbb D}\to S$ is called {\it a bypass} for the curve $L$ if $\chi$ is bi-Lipschitz, $\chi(\overline{\mathbb D})\cap L = \chi(\partial_+\mathbb D)$ and $\left(L\setminus \chi(\partial_+\mathbb D)\right)\cup\chi(\partial_-\mathbb D)$ is a Lavrentiev curve. See Figure~\ref{bypass-figure}.

We say that the bypass $\chi$ {\it is attached along} the arc $\chi(\partial_+ \mathbb D).$
\end{defi}

\begin{figure}[h]
\center{\includegraphics{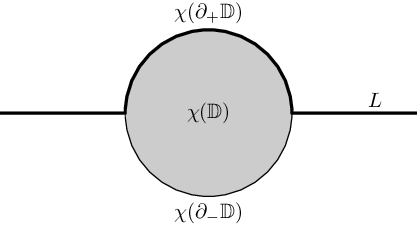}}
\caption{A bypass for the Lavrentiev curve}
\label{bypass-figure}
\end{figure}

Consider the isotopy 
$$
F_{+\mapsto-}:[-1;1]\times[0;1]\to\overline{\mathbb D},\quad
F_{+\mapsto-}(s,t) = \dfrac{s+(1-2t)\mathbbm i}{(1-2t)\mathbbm i s + 1}
$$
from the arc $\partial_+\mathbb D$ to the arc $\partial_-\mathbb D$ through subarcs of geometric (generalized) circles fixing the endpoints of the arcs. The isotopy $F_{+\mapsto-}$ is Lipschitz. It is the direct check that for any $t\in[0;1]$ the arc $F_{+\mapsto}([-1;1], t)$ is $\pi/2$-Lavrentiev and the map $F_{+\mapsto}(\bullet, t)$ is $\pi$-bi-Lipschitz.

\begin{defi}
Let $\chi$ be a bypass for a Lavrentiev curve $L$ on the surface $S$. By {\it an isotopy associated with the bypass $\chi$} we call the isotopy
$$
F: L\times[0;1]\to S,
\quad
F(p, t) = 
\left\{
\begin{aligned}
\chi\left( F_{+\mapsto-}(\chi^{-1}(p), t)\right),&\quad p\in \chi(\partial_+\mathbb D),\\
p,&\quad p\in L\setminus \chi(\partial_+\mathbb D).
\end{aligned}
\right.
$$
\end{defi}

It is clear that during this isotopy the attaching arc $\chi(\partial_+ \mathbb D)$ moves to the arc $\chi(\partial_- \mathbb D)$ while the other part of the curve $L$ stays fixed.

\begin{prop}\label{associated-isotopy-with-bypass-proposition}
The isotopy $F$ associated with the bypass is Lipschitz and there exists $C$ such that for all $t\in[0;1]$ the map $F(\bullet,t)$ is $C$-bi-Lipschitz.
\end{prop}

\begin{proof}
Let $L$ denote the curve, $\chi$ denote the bypass, $L_0$ be the attaching arc and $L_1=\chi(\partial_-\mathbb D).$

The isotopy $F$ is Lipschitz since it is the composition of Lipschitz maps. By Lemma~\ref{bi-Lipschitz-family} it is sufficient to check the $C$-bi-Lipschitz condition locally.

For any $t\in[0;1]$ the map $F(\bullet, t)\big|_{L\setminus L_0}$ is an identity, hence $1$-bi-Lipschitz.

Let $C_{\chi}$ be a number such that $\chi$ is $C_{\chi}$-bi-Lipschitz. Then for any $t\in[0;1]$ the map $F(\bullet, t)\big|_{L_0\setminus\partial L_0}$ is $C_{\chi}^2\cdot\pi$-bi-Lipschitz.

Let $p\in\partial L_0.$ We choose a Euclidean chart containing the point $p$ and consider a closed geometric disk $B$ which contains $p$ and does not contain the other point of $\partial L_0$. Let $U$ be a neighborhood of the point $p$ such that for any $t\in[0;1]$ for any two points $q_1, q_2\in F(L\times\{t\})\cap U$ the subarc $F(L\times\{t\})\big|_{q_1}^{q_2}$ lies inside $B$. There exists a number $C$ such that the curves $L$, $(L\setminus L_0)\cup L_1$ and $F(L_0\times\{t\})$ for any $t\in[0;1]$ are $C$-Lavrentiev.

\begin{figure}[h]
\center{\includegraphics{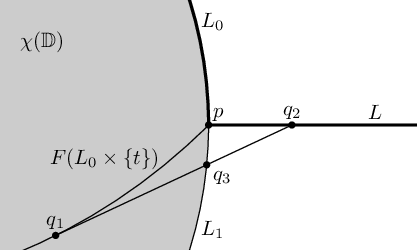}}
\caption{Checking the bi-Lipschitz condition for the isotopy associated with a bypass}
\label{bypass-isotopy-case2-figure}
\end{figure}

Let $\mathcal J$ be an open subarc of $L$ such that $p\in\mathcal J$ and $F(\mathcal J, t)\subset U$ for any $t\in[0;1].$ We will prove that the map $F(\bullet, t)\big|_{\mathcal J}$ is $\left(C^3\cdot C_{\chi}^2\cdot\pi\right)$-bi-Lipschitz for any $t.$

Let $p_1\in L_0\cap\mathcal J$ and $p_2\in \left(L\setminus L_0\right)\cap\mathcal J.$ Then $q_1 = F(p_1, t)\in U$ and $q_2 = p_2\in U.$ See Figure~\ref{bypass-isotopy-case2-figure}. Then

$$
\ln\left|\frac{\mathrm{dist}(q_1, q_2)}{\mathrm{dist}(p_1, p_2)}\right| \le 
\ln\left|\frac{\mathrm{dist}(q_1, q_2)}{\ell\left(L\big|_{p_1}^{p_2}\right)}\right| + \ln C. 
$$

Since $\ell\left(L\big|_{p_1}^{p_2}\right) = \ell\left(L\big|_{p_1}^{p}\right) + \ell\left(L\big|_{p}^{p_2}\right)$ and $\ln \left|\frac{\ell\left(L\big|_{p_1}^{p}\right)}{\ell\left(F(L\times\{t\})\big|_{q_1}^{p}\right)}\right| \le \ln \left(C_{\chi}^2\cdot \pi\right)$,

$$
\ln\left|\frac{\mathrm{dist}(q_1, q_2)}{\mathrm{dist}(p_1, p_2)}\right| \le 
\ln\left|\frac{\mathrm{dist}(q_1, q_2)}{\ell\left(F(L\times\{t\})\big|_{q_1}^{q_2}\right)}\right| + \ln \left(C\cdot C_{\chi}^2\cdot\pi\right). 
$$

Since $q_1$ lies inside the closed disk $\chi(\overline{\mathbb D})$, $q_2$ lies outside this disk, and $\partial\chi(\overline{\mathbb D}) = L_0\cup L_1$, there is a point $q_3\in L_i,$ where $i = 0$ or $i = 1$, lying on the straight line segment joining the points $q_1$ and $q_2.$ Also

\begin{multline*}
\ell\left(F(L\times\{t\})\big|_{q_1}^{q_2}\right) = \ell\left(F(L\times\{t\})\big|_{q_1}^{p}\right) + \ell\left(F(L\times\{t\})\big|_{p}^{q_2}\right)
\le C\cdot\mathrm{dist}(q_1, p) + \ell\left(L\big|_{p}^{q_2}\right) \le \\
\le C(\mathrm{dist}(q_1, q_3) + \mathrm{dist}(q_3, p)) + \ell\left(L\big|_{p}^{q_2}\right)=
C\cdot\mathrm{dist}(q_1, q_3) + C\cdot \mathrm{dist}(q_3, p) + \ell\left(L\big|_{p}^{q_2}\right)\le\\
\le C\cdot\mathrm{dist}(q_1, q_3) + C\cdot\ell\left(F(L\times \{i\})\big|_{q_3}^p\right) + C\cdot \ell\left(L\big|_{p}^{q_2}\right) = C\cdot\mathrm{dist}(q_1, q_3) + C\cdot\ell\left(F(L\times\{i\})\big|_{q_3}^{q_2}\right)\le\\
\le C\cdot\mathrm{dist}(q_1, q_3) + C^2\mathrm{dist}(q_3, q_2) \le  C^2\cdot\mathrm{dist}(q_1, q_3) + C^2\mathrm{dist}(q_3, q_2) = C^2\cdot \mathrm{dist}(q_1, q_2).
\end{multline*}

Therefore

$$\ln\left|\frac{\mathrm{dist}(q_1, q_2)}{\mathrm{dist}(p_1, p_2)}\right| \le \ln \left(C^3\cdot C_{\chi}^2\cdot\pi\right).$$

\end{proof}

\section{Legendrian curves and Legendrian isotopies}
In this section we introduce Legendrian Lavrentiev curves in dimension 3. We prove that a Lavrentiev curve is Legendrian if and only if the integral of the contact form on any its subarc is zero. This property closely resembles the definition of a Legendrian smooth curve. 

We introduce Legendrian isotopies of Legendrian Lavrentiev curves and prove that for any number $C$ any $C^0$-continuous isotopy of Legendrian $C$-Lavrentiev curves is Legendrian. 
%The definition of Legendrian isotopies allows more general isotopies but actually 
%During the smoothing procedure in Section~\ref{Legendrian-smoothing-section}q we will construct only Lipschitz Legendrian isotopies of $C$-bi-Lipschitz curves with $C$ depending only on the isotopy.

%For the Legendrian isotopy, we have no appropriate topology on the space of Lavrentiev curves at this moment to define a Legendrian isotopy to be a continuous isotopy of Legendrian Lavrentiev curves. To overcome this we require that the isotopy is $C^0$-continuous and that in any Euclidean chart the projection of the isotopy to the contact plane is locally an isotopy of embedded curves (on the plane). For example, this property forbids the isotopy to tie tiny knots on the curve thus the topological type of the curve preserves. We will discuss this approach in Section~\ref{Continuous-Legendrian-curves-section} (see Definition~\ref{def-projection}). In this section we give a different but equivalent approach to define Legendrian curves and their isotopies. Its advantage is that it allows one to choose a Euclidean chart when the Legendrian condition needs checking.

We discuss the regular projection of Legendrian curves. We prove that any continuous isotopy of $C$-bi-Lipschitz curves on the plane lifts to a Legendrian isotopy.

% Together with the results of Section~\ref{regular-neighborhood-section} on the existence of regular neighborhoods of Legendrian curves this reduces the problem of smoothing Legendrian links to the problem of smoothing Lavrentiev curves on a surface which will be discussed in the next section.

\begin{defi}
A 2-plane distribution on a smooth 3-manifold is called {\it contact structure} if locally it is the kernel of a smooth differential 1-form $\alpha$ such that $\alpha\wedge d\alpha\neq0$ everywhere.

We denote by $\xi_p$ a plane of the contact structure $\xi$ at the point $p$.

Planes of contact structure are called {\it contact}. The 1-form $\alpha$ is also called {\it contact}.
\end{defi}

We fix the orientation of the contact manifold determined locally by the form $\alpha\wedge d\alpha.$ Clearly the orientation does not depend on the contact form representing the contact structure. Therefore the orientation is well defined.

\begin{defi}
\label{def-angle}
A locally Lavrentiev curve is called {\it Legendrian} if for any point $p$ on the curve and any (some) Euclidean coordinates of class $C^1$ in a neighborhood of $p$ it is true that for any $\varepsilon > 0$ there exists a smaller neighborhood $U\ni p$ such that for any two distinct points $p',p''\in U$ on the curve
\begin{equation}\angle(p''-p', \xi_p) < \varepsilon.
\label{angle-condition}
\end{equation}

A continuous isotopy of Legendrian curves $\{L_t\}_{t\in[0;1]}$ is called {\it Legendrian} if for any moment $t_0\in[0;1]$, for any point $p\in L_{t_0}$ and any (some) Euclidean coordinates of class $C^1$ in a neighborhood of $p$ it is true that for any $\varepsilon > 0$ there exist a smaller neighborhood $U\ni p$ and an interval $I\ni t_0$ such that for any moment $t\in I$ and any two distinct points $p',p''\in L_t\cap U$ the condition~(\ref{angle-condition}) is satisfied.
\end{defi}

%The statements in the following examples are straight forward computations or corollaries of Lemmas~\ref{Lavrentiev-projection-lemma} and \ref{Lavrentiev-projection-lemma2} and Proposition~\ref{leg-integral}.

\begin{exam}\label{piecewise-smooth-example}
A cusp-free piecewise smooth embedded curve, which is a union of Legendrian smooth arcs, is Legendrian. It follows from Proposition~\ref{leg-integral} and the fact that it is locally Lavrentiev (by Lemma~\ref{Lavrentiev-example}).
\end{exam}

\begin{exam}[Cusp]
The curve $(t^3,t^2,-\frac25 t^5)$ in $(\mathbb R^3, \ker\left(dz + xdy\right))$ is a union of two Legendrian arcs but is not Legendrian since it is not locally Lavrentiev by Lemma~\ref{Lavrentiev-semitangent}.
\end{exam}

\begin{exam}[Logarithmic spiral]\label{spiral-example}
The closure of a leaf of the contact foliation on the surface $z = x^2 + y^2$ with respect to the contact structure $\ker(dz+xdy-ydx)$ is a Legendrian curve. To see this, note that the orthogonal projection of this curve to the $xy$-plane is a logarithmic spiral. This follows from the fact that the surface and the contact structure are invariant under dilatations $(x,y,z)\mapsto (kx,ky,k^2z)$ where $k>0$ and rotations around the $z$-axis. Logarithmic spirals are Lavrentiev (see Example~\ref{Lavrentiev-example-2}). So the initial curve is Legendrian by Lemma~\ref{Lavrentiev-projection-lemma2}. The union of two such curves is also a Legendrian curve.
\end{exam}

\begin{exam}[Too slow spiraling]
The closure of a leaf of the contact foliation on the surface $z^{3/2} = x^2 + y^2$ with respect to the contact structure $\ker(dz+xdy-ydx)$ is not a Legendrian curve since it is not locally Lavrentiev.
\end{exam}

First, we prove that in Definition~\ref{def-angle} we can evenly write "any" or "some".

\begin{lemm}\label{def-angle-any-some}
Let $V$ and $W$ be open subsets of Euclidean space $\mathbb R^3$, $p\in V,$ $\xi_p\subset T_p\mathbb R^3$ be a 2-plane and $h:V\to W$ be a $C^1$-diffeomorphism. Then for any $\varepsilon > 0$ there exists $\delta > 0$ such that for any two distinct points $p',p''$ such that $\mathrm{dist}(p,p')<\delta,$ $\mathrm{dist}(p,p'')<\delta$ and $\angle(p''-p', \xi_p) < \delta$
$$\angle(h(p'')-h(p'), Dh(\xi_p)) < \varepsilon.$$
\end{lemm}

\begin{proof}
If $h$ is an affine map then the statement is obvious since any affine map acts continuously on the set of directions.

In the general case by continuity of $Dh$ 
\begin{multline*}
h(p'') - h(p') = \int\limits_0^1 \frac{d}{dt}\left(h((1-t)p'+tp'')\right)dt = \\ = \int\limits_0^1 \left(D_{(1-t)p'+tp''}h \right)(p''-p') dt = \left(D_p h\right) (p''-p') + o(1)\cdot||p''-p'||
\end{multline*}
and hence
$$
\angle(h(p'') - h(p'), Dh(\xi_p)) = \angle(\left(D_p h\right)(p'' - p'), Dh(\xi_p)) + o(1),
$$
where $o(1)\to0$ uniformly as $\delta\to0.$ So we reduced the problem to the affine case.
\end{proof}

%\begin{defi}
%A {\it coline element} is a(n unordered) pair of two opposite covectors. A field of coline elements is called smooth if the restriction of this field to any simply connected open subset has the form $\{\alpha, -\alpha\}$ where $\alpha$ is a smooth differential 1-form.
%\end{defi}
%
%A contact structure can be defined by the kernel of a globally defined contact form if and only if it is coorientable. In any case it can be defined as the kernel of a coline element field which we call {\it contact}.
%

\begin{prop}
\label{leg-integral}
Let $\alpha$ be a contact form defined in some neighborhood of a locally Lavrentiev curve $L$. Then $L$ is Legendrian if and only if the Riemann--Stieltjes integral of the contact form on any compact subarc of $L$ is zero.
\end{prop}

\begin{proof}
Let $L$ be Legendrian. Let $\gamma$ be a compact subarc of $L$ and let us prove that the integral of $\alpha$ on $\gamma$ equals zero.

We can assume that $\gamma$ is contained in some Euclidean chart. Indeed, it is true for sufficiently small subarcs of $\gamma$, so we can use compactness of $\gamma$ and finite additivity of integral.

Let $\gamma:[a;b]\to M$ be a parametrization of the arc $\gamma.$ By Definition~\ref{def-angle} for any $\varepsilon > 0$ there exists $\delta>0$ such that if $s_0, s_1, s'\in[a;b]$, $|s_0 - s'|<\delta,$ $|s_1 - s'| < \delta$ and $s_0\neq s_1$ then $\angle(\gamma(s_1)-\gamma(s_0), \xi_{\gamma(s')}) < \varepsilon.$ 

Let $a=s_0 < s_1 < \ldots < s_n=b$ be a subdivision of the interval $[a;b]$ such that $|s_i-s_{i-1}|<\delta$ for any $i = 1,\dots, n.$ Then

\begin{multline*}
\left|\sum\limits_{j=1}^n \alpha_{\gamma(s_j')} \left(\gamma(s_j) - \gamma(s_{j-1}) \right)\right| \le \sum\limits_{j=1}^n ||\alpha_{\gamma(s_j')}|| \cdot ||\gamma(s_j) - \gamma(s_{j-1}|| \sin\angle(\gamma(s_j) - \gamma(s_{j-1}), \xi_{\gamma(s_j')}) < \\ < ||\alpha|| \ell(\gamma) \sin\varepsilon,
\end{multline*}

where $s_j'\in [s_{j-1};s_j]$, $\alpha_p = \alpha\big|_{T_pM}$ and $||\alpha|| = \sup\limits_{p\in\gamma} ||\alpha_p||.$
Therefore the integral sum tends to zero as $\delta\to0.$

Now let the integral on any compact subarc of $L$ be zero. Let us prove that the curve is Legendrian.

Let $p\in L$ and suppose that $p$ is contained in some Euclidean chart. We can also assume that the curve $L$ is compact and is contained in this chart. Since $L$ is locally Lavrentiev, it is $C$-Lavrentiev for some $C$ with respect to Euclidean metric.

Let $\varepsilon > 0.$ Let $B_{\delta}(p)$ be a ball of radius $\delta$ centered at $p$ such that for any two points $p',p''\in B_{\delta}(p)$

\begin{equation}
\label{alpha-range-bound}
||\alpha_{p''}-\alpha_{p'}||<\varepsilon.
\end{equation}

Let $U$ be an open ball centered at $p$ of radius $\delta / (C+1).$ Then for any two points $p',p''\in L\cap U$
$$
L\big|_{p'}^{p''}\subset B_{\delta}(p)
$$
because if there exists a point $q\in L\big|_{p'}^{p''}\setminus B_{\delta}(p)$ then

\begin{multline*}
2(\delta - \delta / (C+1) ) < \mathrm{dist}(p', q) + \mathrm{dist}(q, p'') \le \ell(L\big|_{p'}^{p''}) \le C\mathrm{dist}(p', p'') < C\cdot 2\delta/(C+1),
\end{multline*}
a contradiction.

Let $p',p''\in L\cap U.$ Then

\begin{equation}
\label{angle-equation}
\sin \angle(p''-p', \xi_p) = \frac{|\alpha_p(p''-p')|}{||\alpha_p||\cdot||p''-p'||}.
\end{equation}

Let us give an upper bound for an expression above the fraction line. Let $\gamma:[a;b]\to M$ be a parametrization of a curve $L\big|_{p'}^{p''}$ and let $a=s_0 < s_1 < \ldots < s_n=b$ be a subdivision of the interval $[a;b]$ such that 

\begin{equation}
\label{integral-sums}
\left|\sum\limits_{j=1}^n \alpha_{\gamma(s'_j)}(\gamma(s_j)-\gamma(s_{j-1})) - \int\limits_{\gamma}\alpha\right|<\varepsilon ||p''-p'||
\end{equation}
for any $s'_j\in[s_{j-1};s_j].$

Then

\begin{multline*}
\left|\alpha_p(p''-p')\right| = \left|\sum\limits_{j=1}^n \alpha_p(\gamma(s_j) - \gamma(s_{j-1}))\right| = \\
= \left|\sum\limits_{j=1}^n \alpha_{\gamma(s'_j)}(\gamma(s_j) - \gamma(s_{j-1})) + \sum\limits_{j=1}^n (\alpha_p - \alpha_{\gamma(s'_j)})(\gamma(s_j) - \gamma(s_{j-1}))  \right| \stackrel{(\ref{alpha-range-bound}),(\ref{integral-sums})}{\le} \\
\stackrel{(\ref{alpha-range-bound}),(\ref{integral-sums})}{\le} \left|\int\limits_{\gamma}\alpha\right| + \varepsilon||p''-p'|| + \varepsilon \ell(L_{p'}^{p''})\le \varepsilon(1 +  C) ||p''-p'||.
\end{multline*}

So by~(\ref{angle-equation})
$$
\sin \angle( p''-p', \xi_p) \le \frac{\varepsilon(1 +  C)}{||\alpha_p||}.
$$ 

This means that the curve $L$ is Legendrian.
\end{proof}

The natural parametrization of a Lavrentiev curve is a bi-Lipschitz map, hence in Euclidean chart it is given by a triple of Lipschitz functions. By Lebesgue theorem (on functions of bounded variation) every Lipschitz function on the line is almost everywhere differentiable and it is equal to the integral of its derivative which for Lipschitz functions belongs to the class $L^{\infty}(\mathbb R)$. Hence by Radon--Nikodym theorem $\int\limits_{\gamma} \alpha = \int \alpha(\dot\gamma) dt$, where the Riemann--Stiltjes integral is on the left side and the Lebesgue integral is on the right. See for example~\cite{ShilGur} for details. So we obtain the following.

\begin{coro}\label{tangent-to-contact-plane-almost-everywhere}
A locally Lavrentiev curve is Legendrian if and only if almost everywhere it has a derivative which belongs to the contact plane.
\end{coro}

\begin{prop}
\label{leg-C-isotopy}
Any continuous isotopy of Legendrian $C$-Lavrentiev curves is Legendrian.
\end{prop}

\begin{proof}
We can assume that the curves are compact and non-closed and lie in some fixed Euclidean chart. Since all the curves are $C$-Lavrentiev with respect to some Riemannian metric on the manifold, they all are $\widetilde C$-Lavrentiev with respect to Euclidean metric. The rest of the proof is similar to the proof of "if" part of the proof of Proposition~\ref{leg-integral} since the bound on the angle depends only on $\alpha$ and $C.$
\end{proof}

\subsection{Regular projection}
\label{regular-projection-section}
Let $dz + \mathrm{pr}^*\beta$ be the contact form on $S\times\mathbb R$, where $S$ is a smooth surface, $\mathrm{pr}$ is the projection $S\times\mathbb R\to S$, $z$ is the projection $S\times\mathbb R\to\mathbb R$, $\beta$ is a smooth differential 1-form on $S$. We call the map $\mathrm{pr}$ the {\it regular projection}. 

For all curves in this subsection we require that their regular projection is injective. As usual we equip $S$ with some Riemannian metric, $\mathbb R$ with the Euclidean metric and $S\times\mathbb R$ with the (tensor) product metric.

\begin{lemm}\label{z-uniqueness}
Let two Legendrian curves have a common point and their regular projections coincide. Then these Legendrian curves coincide.
\end{lemm}

\begin{proof}
Denote the curves by $\gamma_1$ and $\gamma_2$ and let $p\in\gamma_1\cap\gamma_2,$ $q_1\in\gamma_1,$ $q_2\in\gamma_2$ and $\mathrm{pr} (q_1) = \mathrm{pr} (q_2).$

Since the curves are rectifiable then their regular projections are also rectifiable and therefore we can integrate the 1-form $\beta$ on these projections.

By proposition~\ref{leg-integral}

$$
z(q_i) - z(p) = -\int\limits_{\mathrm{pr}\left(\gamma_i\big|_p^{q_i}\right)}\beta,
$$
for $i = 1, 2.$ Since the right side of the equality does not depend on $i$, $z(q_1) = z(q_2).$ Hence the curves coincide.

\end{proof}

\begin{lemm}
\label{Lavrentiev-projection-lemma}
The regular projection of a Legendrian curve is locally Lavrentiev.
\end{lemm}

\begin{proof}
Let $p$ be a point on the regular projection. It is sufficient to prove that there exists an open neighborhood $U\subset S$ of the point $p$ such that the intersection of $U$ with the regular projection of the Legendrian curve is a locally Lavrentiev curve. We can choose such a neighborhood $U$ that this intersection is a curve and $U$ is diffeomorphic to an open subset of the Euclidean plane. Let $\gamma$ be a compact subarc of the Legendrian curve whose regular projection lies inside $U$. It is sufficient to prove that $\mathrm{pr}(\gamma)$ is Lavrentiev with respect to the Euclidean metric (see Convention~\ref{Lavrentiev-metric-equivalence}).

For some $C$ the arc $\gamma$ is $C$-Lavrentiev with respect to the standard Euclidean metric in $U\times\mathbb R.$

Let

$$\varphi = \inf\limits_{p'\neq p''\in\gamma}\angle(p''-p',\partial/\partial z).$$

Let us prove that $\varphi > 0.$ Suppose the contrary. Then by compactness there are sequences $\{p'_i\}_{i\in\mathbb N}$, $\{p''_i\}_{i\in\mathbb N}$ such that $p'_i, p''_i\in \gamma$ for any $i\in\mathbb N$ and  $p'_i\to p',$ $p''_i\to p''$ and $\angle(p''_i-p'_i,\partial/\partial z)\to 0$ as $i\to+\infty.$ If $p' = p''$ we get a contradiction with the fact that $\gamma$ is Legendrian since $\partial/\partial z$ is transverse to the contact structure. If $p'\neq p''$ we get a contradiction with the injectivity of the regular projection.

Let $p'$ and $p''$ be two distinct points on the curve $\gamma.$ Then

$$
|\mathrm{pr}_{xy}p''-\mathrm{pr}_{xy}p'| \ge |p'' - p'|\sin\varphi \ge \ell\left(\gamma\big|_{p'}^{p''}\right)\sin\varphi/C \ge \ell\left(\mathrm{pr}_{xy}\left(\gamma\big|_{p'}^{p''}\right)\right)\sin\varphi/C.
$$

So the regular projection of $\gamma$ is $(C/\sin\varphi)$-Lavrentiev with respect to the standard Euclidean metric in the $xy$-plane.
\end{proof}

\begin{lemm}\label{Lavrentiev-projection-lemma3}
Let the regular projection of a compact arc $\gamma$ be $C$-Lavrentiev. Suppose that for any two points $p,q\in\gamma$
$$
z(q) - z(p) = -\int\limits_{\mathrm{pr}\left(\gamma\big|_{p}^{q}\right)}\beta.
$$
Then
\begin{enumerate}
\item For any two points $p,q\in\gamma$
\begin{equation}
\label{projection-distance-bound}
\mathrm{dist}(p,q) \le \sqrt{1+\left(\sup\limits_{p'\in\gamma} ||\beta_{p'}||\cdot C\right)^2}\cdot\mathrm{dist}(\mathrm{pr}(p), \mathrm{pr}(q)).
\end{equation}

\item For any two points $p,q\in\gamma$
\begin{equation}
\label{projection-length-bound}
\ell\left(\gamma\big|_{p}^{q}\right)\le \sqrt{1+\left(\sup\limits_{p'\in\gamma} ||\beta_{p'}||\cdot C\right)^2}\cdot \ell\left(\mathrm{pr}\left(\gamma\big|_{p}^{q}\right)\right).
\end{equation}

\item $\gamma$ is $\sqrt{1+\left(\sup\limits_{p'\in\gamma} ||\beta_{p'}||\cdot C\right)^2}\cdot C$-Lavrentiev.
\end{enumerate}
\end{lemm}

\begin{proof}
Let $p,q\in\gamma.$
Then 

$$|z(q)-z(p)| = \left|\phantom{.}\int\limits_{\mathrm{pr}\left(\gamma\big|_{p}^{q}\right)}\beta\phantom{.}\right| \le \sup\limits_{p'\in\gamma} ||\beta_{p'}|| \cdot \ell\left(\mathrm{pr}\left(\gamma\big|_{p}^{q}\right)\right)\le \sup\limits_{p'\in\gamma} ||\beta_{p'}||\cdot C \cdot\mathrm{dist}(\mathrm{pr}(p), \mathrm{pr} (q)).$$

Put $C_1 = \sup\limits_{p'\in\gamma} ||\beta_{p'}||\cdot C.$

Then
$$
\mathrm{dist}(p,q) = \sqrt{|z(q)-z(p)|^2 + \left(\mathrm{dist}(\mathrm{pr}(p), \mathrm{pr}(q))\right)^2}\le \sqrt{1+C_1^2}\cdot\mathrm{dist}(\mathrm{pr}(p),\mathrm{pr}(q)).
$$

Thus the inequality~(\ref{projection-distance-bound}) is proved.
Since the inequality~(\ref{projection-distance-bound}) is fulfilled for any pair of points on the curve $\gamma$, a similar inequality holds for the sums in the definition of the length (Definition~\ref{Lavrentiev-definition}), thus the inequality~(\ref{projection-length-bound}) also holds.

Let us continue the inequality~(\ref{projection-length-bound}) using the fact that $\mathrm{pr}(\gamma)$ is $C$-Lavrentiev:

$$
\ell\left(\gamma\big|_{p}^{q}\right)\le \sqrt{1+C_1^2}\cdot \ell\left(\mathrm{pr}\left(\gamma\big|_{p}^{q}\right)\right)\le \sqrt{1+C_1^2}\cdot C\cdot \mathrm{dist}(\mathrm{pr}(p), \mathrm{pr}(q))\le \sqrt{1+C_1^2} \cdot C \cdot \mathrm{dist}(p,q).
$$

Hence $\gamma$ is $\sqrt{1+C_1^2}\cdot C$-Lavrentiev.

\end{proof}

\begin{lemm}\label{Lavrentiev-projection-lemma4}
Let $\gamma$ be a compact Legendrian curve. Then its regular projection is Lavrentiev and the projection map is bi-Lipschitz.
\end{lemm}

\begin{proof}
Since $\gamma$ is Legendrian, its regular projection $\mathrm{pr}(\gamma)$ is locally Lavrentiev by Lemma~\ref{Lavrentiev-projection-lemma} thus it is Lavrentiev by Lemma~\ref{compact-Lavrentiev} since it is compact. Therefore by Lemma~\ref{Lavrentiev-projection-lemma3} (1) the projection map is bi-Lipschitz.
\end{proof}

\begin{lemm}\label{Lavrentiev-projection-lemma2}
Let the regular projection of an embedded curve $\gamma$ be locally Lavrentiev and for any two points $p,q\in\gamma$
$$
z(q) - z(p) = -\int\limits_{\mathrm{pr}\left(\gamma\big|_{p}^{q}\right)}\beta.
$$

Then $\gamma$ is Legendrian.
\end{lemm}

\begin{proof}
By proposition~\ref{leg-integral} it suffices to prove that $\gamma$ is locally Lavrentiev. We can assume that $\gamma$ is a compact arc. Since $\mathrm{pr}( \gamma)$ is locally Lavrentiev and compact, it is $C$-Lavrentiev for some $C$. By Lemma~\ref{Lavrentiev-projection-lemma3}~(3) the curve $\gamma$ is Lavrentiev.
\end{proof}

\begin{lemm}\label{Lavrentiev-projection-isotopy-lemma}
Let $C$ be a number and $\{\gamma_t\}_{t\in[0;1]}$ be a continuous isotopy of Legendrian curves whose regular projections are $C$-Lavrentiev. Then $\{\gamma_t\}_{t\in[0;1]}$ is a Legendrian isotopy.
\end{lemm}

\begin{proof}
We can assume that for all $t\in[0;1]$ the curve $\gamma_t$ is a compact arc.
Then for all $t\in[0;1]$ the curve $\gamma_t$ is $\sqrt{1+\left(\sup\limits_{\exists t:\ p'\in\gamma_t} ||\beta_{p'}||\cdot C\right)^2}\cdot C$-Lavrentiev by Lemma~\ref{Lavrentiev-projection-lemma3}~(3).

By Proposition~\ref{leg-C-isotopy} the isotopy $\{\gamma_t\}_{t\in[0;1]}$ is Legendrian.
\end{proof}

Now we utilize some results from Whitney's book~\cite{Whi}. Suppose that $S = \mathbb R^2$ for a moment. Let $L$ be a compact Lavrentiev curve in $\mathbb R^2$, $f_1,f_2:L\to\mathbb R^2$ be two $C$-Lipschitz maps. Lavrentiev compact curves are flat chains (see Section V.3 for the definition). Smooth bounded differential 1-forms with bounded differential are flat cochains (see Section V.4 and Theorem 7A in Chapter IX). So the inequality from Theorem 13A in Chapter X gives us
\begin{equation}
\label{integral-continuity-bound2}
\left|\phantom{,}\int\limits_{f_1(L)}\beta - \int\limits_{f_2(L)}\beta\phantom,\right| \le 
\sup\limits_{p\in L}\mathrm{dist}\left(f_1(p), f_2(p)\right)\cdot\left(C\cdot\ell(L) + 2\right)
\cdot \sup\limits_{q\in V} \{ ||\beta_q||, ||d\beta_q||\},
\end{equation}
where $V$ is an open subset of $\mathbb R^2$ such that for any $p\in L$ the straight-line segment joining the points $f_1(p)$ and $f_2(p)$ lies in $V.$ So we have the following corollary.

\begin{coro}\label{integral-continuity-corollary}
Let $L$ be a compact Lavrentiev curve on a smooth surface $S$, $p\in L$ and $\beta$ be a differential 1-form on $S$. Let $F_0:L\times[0;1]\to S$ be an isotopy. If there exists a real number $C$ such that the map $F_0(\bullet, t)$ is $C$-bi-Lipschitz for any $t\in[0;1]$, the function 
$$(q, t) \mapsto \int\limits_{F_0\left(L\big|_p^q\times\{t\}\right)}\beta$$ is continuous. If the map $F_0$ is Lipschitz, this function is Lipschitz.
\end{coro}

\begin{proof}
We choose some Riemannian metric $g$ on $S$. We can assume that $S$ is compact.

Let us fix the moment $t_1\in[0;1].$ We cover the curve $F_0(L\times\{t_1\})$ by a finite number of open subsets of $S$ whose closures are compact and diffeomorphic to the unit disk. We equip each element of the covering by Euclidean metric. Let $C_g$ be such number that Euclidean metric in any element of the covering is $C_g$-comparable with the metric $g$. For $q_1,q_2\in L$ and $t_2\in[0;1]$
$$
\int\limits_{F_0\left(L\big|_p^{q_2}\times\{t_2\}\right)}\beta - \int\limits_{F_0\left(L\big|_p^{q_1}\times\{t_1\}\right)}\beta = \int\limits_{F_0\left(L\big|_{q_1}^{q_2}\times\{t_1\}\right)}\beta + \left(\int\limits_{F_0\left(L\big|_p^{q_2}\times\{t_2\}\right)}\beta - \int\limits_{F_0\left(L\big|_p^{q_2}\times\{t_1\}\right)}\beta \right).
$$

The first summand on the right side is not greater than $C \cdot\ell\left(L\big|_{q_1}^{q_2}\right) ||\beta||.$ We choose a neighborhood $\mathcal I$ of $t_1$ and a subdivision of $L$ into subarcs such that for any subarc there is an element of the covering which contains the image of this subarc under the map $F_0(\bullet, t)$ for any $t\in\mathcal I.$ Suppose that $t_2\in\mathcal I.$ Then we apply inequality~(\ref{integral-continuity-bound2}) for each subarc of the subdivision and obtain that the second summand is not greater than 
$$\left(C_g\right)^3 \cdot\sup\limits_{q\in L}\mathrm{dist}(F(q, t_1), F(q, t_2))\cdot(C\cdot \left(C_g\right)^2\cdot\ell(L) + 2n)\cdot\max\{||\beta||, ||d\beta||\}$$
where $n$ is a number of arcs in the subdivision. So the right side of the equality is small if $q_1$ and $q_2$ are close to each other and $t_2$ is sufficiently close to $t_1.$

The proof of the fact that the integral is Lipschitz if $F$ is Lipschitz is similar.
\end{proof}

\begin{prop}\label{bi-Lipschitz-isotopy-lifting}
Let $\gamma$ be a compact curve in $S\times\mathbb R$ and $F:\gamma\times[0;1]\to S\times \mathbb R$ be a map. Let $F_0:\mathrm{pr}(\gamma)\times[0;1]\to S$ be a map such that $\mathrm{pr}\circ F = F_0\circ\mathrm{pr}.$ Let $p\in\gamma$ and $C$ be a number such that
\begin{enumerate}
\item $F(\bullet, 0) = \mathrm{id}_{\gamma}.$
\item For any $t\in[0;1]$ the image $F(\gamma, t)$ is a Legendrian curve.
\item For any $t\in[0;1]$ the map $F_0(\bullet, t)$ is $C$-bi-Lipschitz.
\item The map $F_0$ is continuous.
\item The function $z\circ F(p, \bullet)$ is continuous.
\end{enumerate}
Then $F$ is a Legendrian isotopy and there exists a real number $C'$ such that for any $t\in[0;1]$ the map $F(\bullet, t)$ is $C'$-bi-Lipschitz.
\end{prop}
\begin{proof}
To prove that $F$ is a Legendrian isotopy we check the assumptions of Lemma~\ref{Lavrentiev-projection-isotopy-lemma}.
First we prove that $F$ is continuous.

Since the map $\mathrm{pr}$ is continuous and the map $F_0$ is continuous by condition (4), the map $\mathrm{pr}\circ F = F_0\circ\mathrm{pr}$ is continuous. Therefore it is sufficient to prove that $z\circ F$ is continuous.

Let $q\in\gamma$. Then

$$
z(F(q, t)) = z(F(p, t)) -\int\limits_{F_0(\mathrm{pr}(\gamma)\times \{t\})\big|_p^q}\beta.
$$

By condition (5) the function $z(F(p,t))$ is continuous on $t$. Since $\gamma = F(\gamma, 0)$, by condition (2) the curve $\gamma$ is Legendrian. By Lemma~\ref{Lavrentiev-projection-lemma4} the curve $\mathrm{pr}(\gamma)$ is Lavrentiev. So the integral is continuous by Corollary~\ref{integral-continuity-corollary}.

The curve $\mathrm{pr}(\gamma)$ is Lavrentiev and by condition (3) all maps $F_0(\bullet,t)$ are bi-Lipschitz with a common constant, therefore all curves $\mathrm{pr}(F(\gamma\times\{t\}))$ are Lavrentiev with a common constant. By Lemma~\ref{Lavrentiev-projection-isotopy-lemma} $F$ is a Legendrian isotopy.

Now we prove that the maps $F(\bullet, t)$ are bi-Lipschitz with a common constant. By Lemma~\ref{Lavrentiev-projection-lemma3} all maps $\mathrm{pr}\big|_{F(\gamma\times\{t\})}$ are bi-Lipschitz maps with a common constant. Therefore the maps $F(\bullet, t)$ are bi-Lipschitz with a common constant by condition (3) as the compositions of such maps.
\end{proof}

\begin{lemm}\label{Lipschitz-isotopy-lifting}
Let $\gamma$ be a compact curve in $S\times\mathbb R$ and $F:\gamma\times[0;1]\to S\times \mathbb R$ be a Legendrian isotopy with $F(\bullet, 0) = \mathrm{id}_{\gamma}$.  Let $F_0:\mathrm{pr}(\gamma)\times[0;1]\to S$ be a map such that $\mathrm{pr}\circ F = F_0\circ\mathrm{pr}.$
Let $p\in\gamma$ be a point such that the function $z\circ F(p, \bullet)$ is Lipschitz. Suppose that $F_0$ is Lipschitz. Then $F$ is Lipschitz.
\end{lemm}

\begin{proof}
Since $\gamma = F(\gamma\times\{0\})$ is Legendrian, the map $\mathrm{pr}:\gamma\to\mathrm{pr}(\gamma)$ is bi-Lipschitz by Lemma~\ref{Lavrentiev-projection-lemma4}. So the map $\mathrm{pr}\circ F = F_0\circ\mathrm{pr}$ is Lipschitz. We see that it is sufficient to prove that $z\circ F$ is Lipschitz.

Let $q\in\gamma$. Then

$$
z(F(q, t)) = z(F(p, t)) -\int\limits_{F_0(\mathrm{pr}(\gamma)\times \{t\})\big|_p^q}\beta.
$$

We have that $z(F(p, t))$ is Lipschitz. The integral is Lipschitz by Corollary~\ref{integral-continuity-corollary}.
\end{proof}

\subsection{Pushing the regular projection smoothly to correct the integral}
\label{correct-integral-section}
On the way of constructing Legendrian isotopies we will deal with the following situation. Suppose we have a Legendrian Lavrentiev curve and we want to Legendrian isotope its small subarc. Suppose we have constructed an isotopy of the regular projection of this arc fixing the rest of the regular projection of the curve. Suppose that this isotopy satisfies all restrictions on the isotopy $F_0$ in Proposition~\ref{bi-Lipschitz-isotopy-lifting}. Then by this proposition we can lift this isotopy to a Legendrian isotopy of space curves. But if we want the complement of the subarc on the Legendrian curve to stay fixed during the isotopy, the integral of the 1-form $\beta$ on the regular projection of the subarc must be constant. In this subsection we show how to smoothly correct the given isotopy of plane curves to make the integral constant preserving nice properties of the isotopy.

%We will assume for any isotopy $F:L\times[0;1]\to S$ of a curve $L$ in a surface $S$ that $F(\bullet, 0) = \mathrm{id}_L.$

\begin{defi}\label{integral-correction-square-definition}
Let $L$ be an oriented compact embedded curve in the smooth surface $S$ with a differential 1-form $\beta$ such that $d\beta\neq0$. Let $F:L\times[0;1]\to S$ be a continuous isotopy, $C$ be a real number such that $F(\bullet, t)$ is $C$-bi-Lipschitz for any $t\in[0;1]$. We denote $F(L\times\{t\})$ by $L_t$ for any $t\in[0;1]$. Let $R$ be 
the image of the square $\{(u,v):\ |u|\le1,\ |v|\le1\}$ under a smooth embedding to the surface $S$ such that for any $t\in[0;1]$

\begin{enumerate}
\item The intersection $R\cap L_t$ is a subarc of $L_t.$
\item $\partial R\cap L_t = \partial(R\cap L_t).$
\item $L_t$ does not intersect any side $u = \pm1$ of the square and intersects both sides $v=\pm1.$
\item Let $\gamma^{\pm}_t$ be the oriented subarc of $\partial R$ with the same beginning and the same end as the arc $R\cap L_t$ which contains the side $u=\pm1$. Then 
$$
\int\limits_{L_t\setminus(R\cap L_t)\cup \gamma^+_t}\beta > \int\limits_{L}\beta > \int\limits_{L_t\setminus(R\cap L_t)\cup \gamma^-_t}\beta.
$$
\end{enumerate}

Then $R$ is called an {\it integral correction square} for the isotopy $F$. See Figure~\ref{integral-correction-square-figure}.
\end{defi}

\begin{figure}[ht]
\center{\includegraphics{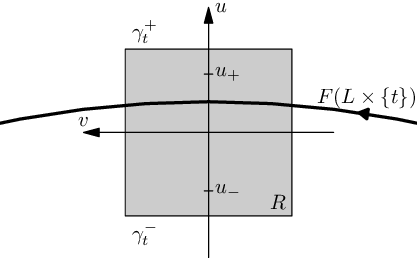}}
\caption{An integral correction square, the coordinate system $(u,v)$ is positive with respect to $d\beta$}
\label{integral-correction-square-figure}
\end{figure}

\begin{defi}
Let $R$ be an integral correction square for the isotopy $F:L\times[0;1]\to S$. A continuous map $\Phi:S\times[0;1]\to S$ is called an isotopy {\it correcting $F$ supported} in $R$ if 
\begin{enumerate}
\item $\Phi(\bullet, 0) = \mathrm{id}_S.$
\item $\Phi(p, t) = (p,t)$ for any $t\in[0;1]$ and for any $p\in S\setminus R.$
\item $$\int\limits_{\Phi(F(L\times\{t\}), t)}\beta = \int\limits_L\beta.$$
\item There exists a real number $C$ such that $\Phi(\bullet, t)$ is $C$-bi-Lipschitz diffeomorphism for any $t\in[0;1].$
\end{enumerate}
\end{defi}

The following proposition asserts that for any integral correction square there exists a correcting isotopy supported in this square, and if the isotopy to be corrected is Lipschitz then there exists a Lipschitz correcting isotopy. Also it asserts that the constructed isotopy depends only on $F(L\times\{t\})\cap R$ and $\int\limits_{F(L\times\{t\})}\beta$ not depending on the precise form of the isotopy $F$ outside $R$.

\begin{prop}
\label{correcting-isotopy-existence}
Let $R = \{(u,v):\ |u|\le1,\ |v|\le1\}$ be a square on the smooth surface $S$ with a differential 1-form $\beta$ such that $d\beta\neq0$ everywhere, and $G:[0;1]\times[0;1]\to R$ be a map. Then there exists a map $\Psi:S\times\mathbb R\times[0;1]\to S$ such that if $R$ is an integral correction square for the isotopy $F:L\times[0;1]\to S$ and $F(L\times\{t\})\cap R = G([0;1]\times\{t\})$ for any $t\in[0;1]$ then $\Phi$, where $\Phi(p, t):= \Psi(p, \int\limits_{F(L\times\{t\})}\beta - \int\limits_L\beta, t)$, is an isotopy correcting $F$ supported in $R$ and, moreover, if $F$ is Lipschitz then $\Phi$ is Lipschitz.
\end{prop}

\begin{lemm}\label{contracting-flow-on-disc}
Let $\Omega$ be an area form on the square $R=\{(u,v):\ |u|\le 1,\ |v|\le1\}.$ Let $u_+, u_-$ be real numbers such that $-1 < u_- < u_+ < 1.$ Then there exist smooth flows of diffeomorphisms $\{\phi_+^t:R\to R\}_{t\in [0;+\infty)}$ and $\{\phi_-^t:R\to R\}_{t\in [0;+\infty)}$ such that
\begin{enumerate}
\item $\phi_{\pm}^0 = \mathrm{id}_R.$
\item For any $t\in[0;+\infty]$ and for any $n,m\in\mathbb Z_{\ge0}$ $\left.\dfrac{\partial^{n+m}}{\partial u^n\partial v^m}\phi_{\pm}^t\right|_{\partial R} = \left.\dfrac{\partial^{n+m}}{\partial u^n\partial v^m} \mathrm{id}\right|_{\partial R}.$
\item There exists a constant $C_1$ such that if $u>u_-$ (respectively, $u < u_+$) at each point of the open set $U\subset R$ then for any positive real numbers $s$ and $t$ such that $s<t$
$$0 < \mathrm{Area}(\phi_{+}^s(U)) - \mathrm{Area}(\phi_{+}^t(U)) \le C_1\cdot \mathrm{Area}(U)\cdot(t-s)$$
(respectively, 
$$\left.0 < \mathrm{Area}(\phi_{-}^s(U)) - \mathrm{Area}(\phi_{-}^t(U)) \le C_1\cdot \mathrm{Area}(U)\cdot(t-s)\right).$$
\item
$$\mathrm{Area}(\phi_+^t(\{(u,v)\in R:\ u > u_-\})) \to 0 \quad \text{ as } t\to+\infty$$
and
$$\mathrm{Area}(\phi_-^t(\{(u,v)\in R:\ u < u_+\})) \to 0 \quad \text{ as } t\to+\infty.$$
\item For any $T>0$ there exists a constant $C_2$ such that if $U$ is an open subset of $R$ such that 
$$\{(u,v)\in R:\ u > u_-\} \supset U \supset \{(u,v)\in R:\ u > u_+\}$$ 
(respectively, 
$$\left.\{(u,v)\in R:\ u < u_-\} \subset U \subset \{(u,v)\in R:\ u < u_+\}\right)$$
then for any $s\in [0;T]$ and $t\in[s;T]$
$$C_2\cdot(t-s) \le  \mathrm{Area}(\phi_{+}^s(U)) - \mathrm{Area}(\phi_{+}^t(U))$$
(respectively, 
$$\left.C_2\cdot(t-s) \le  \mathrm{Area}(\phi_{-}^s(U)) - \mathrm{Area}(\phi_{-}^t(U))\right).$$
\end{enumerate}
\end{lemm}

\begin{proof}
We will prove the existence of the flow $\phi^t_+$. For $\phi^t_-$ the proof is similar.

It is sufficient to construct a vector field $X$ on the square $R$ such that

\begin{enumerate}
\item The vector field $X$ and all partial derivatives of any order of $X$ are zero on $\partial R$.
\item $X = g(u,v)\dfrac{\partial}{\partial u}$ with $g(u,v)>0$ on the set $\{(u,v)\in R\setminus\partial R:\ u\ge u_-\}$.
\item The divergence of $X$ is negative on the set $\{(u,v)\in R\setminus\partial R:\ u\ge u_-\}$.
\end{enumerate}

If we construct the vector field $X$, its flow $\phi^t_+$ will be sought-for. Indeed, properties (1) and (2) are trivially satisfied. For the other properties we can use Newton--Leibniz formula and that 
$$\dfrac{d}{dt} Area (\phi_+^t (U)) = \int\limits_U \dfrac{d}{dt}\left(\left(\phi_+^t\right)^*\Omega\right) = \int\limits_{\phi_+^t(U)} L_X\Omega = \int\limits_{\phi_+^t(U)} \mathrm{div} X\cdot\Omega.$$

So we can take $C_1 = \max\limits_R \mathrm{div}X$ and $C_2 = \int\limits_{\phi_+^T(\{(u,v)\in R:\ u>u_+\})} \mathrm{div}X\cdot \Omega.$

Let $\Omega = f(u,v) du\wedge dv.$ We can assume that $f>0.$ Then

$$
L_X \Omega = di_X\Omega= d(fg dv) = \frac{\partial(fg)}{\partial u}\ du\wedge dv.
$$

There exists a smooth function $g$ on $R$ such that

$$g(u,v) = \dfrac1f exp\left(\dfrac1{u-1}\right) exp\left(\dfrac1{v^2-1}\right),\quad \text{ if } u\ge u_- \text{ and } (u,v)\notin\partial R,
$$
and $\left.\dfrac{\partial^{n+m}g}{\partial u^n\partial v^m}\right|_{\partial R} = 0$ for any $n,m\in\mathbb Z_{\ge0}.$

So $\dfrac{\partial(fg)}{\partial u} < 0$ if $u\ge u_-$ and $(u,v)\notin\partial R$,  and the vector field $X = g(u,v)\dfrac{\partial}{\partial u}$ is sought-for.
\end{proof}

\begin{proof}[Proof of Proposition~\ref{correcting-isotopy-existence}]
Let $u_-, u_+$ be real numbers such that
$F(L\times\{t\})\cap R \subset \{(u,v):\ u_-<u<u_+\}$ for any $t\in[0;1].$

First, we apply Lemma~\ref{contracting-flow-on-disc} to the square $R$, numbers $u_-$ and $u_+$ and the 2-form $\Omega=d\beta.$ Let $\phi^t_+$ and $\phi^t_-$ be the constructed flows. 
%Note that the flows $\phi^t_+$ and $\phi_-^t$ depend only on $R$ and $G$. 
We extend the flows $\phi^t_+$ and $\phi^t_-$ by setting them to the identity outside the square $R$. Let us prove that there exist the unique functions $f_+, f_-:[0;1]\to[0;+\infty)$ such that for any $t\in[0;1]$ $f_+(t)\cdot f_-(t) = 0$ and
$$\int\limits_{ \left(\phi^{f_+(t)}_+\circ \phi^{f_-(t)}_-\right)(F(L\times \{t\}))}\beta = \int\limits_L\beta.$$

We set $f_+(t) = 0$ if 
$$\int\limits_{ F(L\times \{t\})}\beta \ge \int\limits_L\beta$$
and we set $f_-(t) = 0$ if 
$$\int\limits_{ F(L\times \{t\})}\beta \le \int\limits_L\beta.$$

So $f_+(t)\cdot f_-(t) = 0$ for any $t\in[0;1].$

Suppose that $\int\limits_{ F(L\times \{t_0\})}\beta > \int\limits_L\beta.$ 
Let us prove that there exists the unique real number $f_-(t_0)$ such that
\begin{equation}\label{definition-of-f-}
\int\limits_{\phi_-^{f_-(t_0)}(F(L\times\{t_0\}))}\beta = \int\limits_L\beta.
\end{equation}

Let $U$ be the open subset of $R$ such that $$\partial U = \left(F(L\times\{t_0\})\cap R\right)\cup\gamma^-_{t_0}.$$  

Let $t_-$ be some non-negative real number. Then by Stokes theorem (see \cite[Chapter IX, Theorems 5A and 7A]{Whi})

\begin{multline*}
\int\limits_{\phi_-^{t_-}(F(L\times\{t_0\}))}\beta = \int\limits_{\phi_-^{t_-}(F(L\times\{t_0\}))\cap R}\beta + \int\limits_{\phi_-^{t_-}(F(L\times\{t_0\}))\setminus R}\beta = \\
= \left(\int\limits_{\phi_-^{t_-}(U)} d\beta + \int\limits_{\gamma^-_{t_0}}\beta\right) + \int\limits_{F(L_0\times\{t_0\})\setminus R}\beta = 
\int\limits_{\phi_-^{t_-}(U)} d\beta + \int\limits_{(F(L\times\{t_0\})\setminus R)\cup\gamma^-_{t_0}}\beta.
\end{multline*}

Thus by~(\ref{definition-of-f-}) the number $f_-(t_0)$ is the solution of the following equation on $t_-$:
\begin{equation}\label{equation-on-t-}
\int\limits_{\phi_-^{t_-}(U)} d\beta + \int\limits_{(F(L\times\{t_0\})\setminus R)\cup\gamma^-_{t_0}}\beta = \int\limits_L\beta.
\end{equation}

By the third property of $\phi_-^t$ the first summand on the left side of the last equation depends on $t_-$ monotonically and continuously. By the fourth property it tends to zero as $t_-\to +\infty$. So to define $f_-(t_0)$ it is sufficient to prove that

$$
\int\limits_{(F(L\times\{t_0\})\setminus R)\cup\gamma^-_t}\beta \le \int\limits_L\beta \le \int\limits_{F(L\times\{t_0\})}\beta.
$$ 

The first inequality is true by the definition of an integral correction square. 
The second inequality is already assumed.

For $f_+(t)$ the argument is similar.

Now let us prove that the function $f_-$ is continuous. For the function $f_+$ the argument is similar. By the definition the number $f_-(t)$ is equal to the maximum of zero and the solution of the following equation on $t_-$:
\begin{equation}\label{equation-on-t-2}
\int\limits_{\phi_-^{t_-}(F(L\times\{t_0\}))}\beta = \int\limits_L\beta.
\end{equation}

The left side of this equation is a continuous function on $t_0$ and $t_-$ by Corollary~\ref{integral-continuity-corollary}. It is also decreasing on $t_-.$ Thus the implicit function $t_-(t_0)$ is continuous.

Now we set $\Phi(p, t) = \left(\phi^{f_+(t)}_+\circ \phi^{f_-(t)}_-\right)(p)$ for any $p\in S$ and any $t\in[0;1].$ $\Phi$ is continuous as the composition of continuous maps.

Since the functions $f_+(t)$ and $f_-(t)$ are continuous, they are bounded. Let $T = \max\limits_{t\in[0;1]}\{f_+(t), f_-(t)\}.$ Since the flows $\phi_+^t$ and $\phi_-^t$ are smooth, the diffeomorphisms $\phi_+^t$ and $\phi_-^t$ are bi-Lipschitz with a common constant for $t\in[0;T].$ So there exists a real number $C$ such that $\Phi(\bullet, t)$ is $C$-bi-Lipschitz for any $t\in[0;1].$

So by construction $\Phi$ is a correcting isotopy supported in the square $R$.

Now suppose that $F$ is Lipschitz.
Let us prove that the functions $f_+, f_-$ are also Lipschitz. By the third and the fifth properties of $\phi_-^t$ in Lemma~\ref{contracting-flow-on-disc} there exists $C$ such that the left side of~(\ref{equation-on-t-}) is $C$-bi-Lipschitz on $t_-$ for any $t_0\in[0;1].$ Since $F$ is Lipschitz and the flow $\phi_-^t$ is smooth, by Corollary~\ref{integral-continuity-corollary} the left side of~(\ref{equation-on-t-2}) is Lipschitz on $(t_0, t_-).$ Note that the left sides of~(\ref{equation-on-t-}) and~(\ref{equation-on-t-2}) are equal functions. Thus the implicit function $t_-(t_0)$ is Lipschitz.

Since $f_-(t_0) = \max (0, t_-(t_0))$, $f_-$ is also Lipschitz. For $f_+$ the argument is similar. Therefore $\Phi$ is Lipschitz as the composition of Lipschitz functions.

Finally, we define $\Psi$:
$$\Psi(\bullet, I, t) =
\left\{
\begin{aligned}
\phi_+^{t_+},&\quad I\le0,\\
\phi_-^{t_-},&\quad I\ge0,
\end{aligned}
\right.$$
where $\int\limits_{\phi_{\pm}^{t_{\pm}}(F(L\times\{t\}))\cap R}\beta - \int\limits_{F(L\times\{t\})\cap R}\beta = -I.$ If there is no such $t_{\pm}$ set $\Psi(\bullet, I, t) = \mathrm{id}_S.$ It is clear that $\Psi$ depends only on $R$ and $G$ and that $\Phi(\bullet, t) = \Psi(\bullet, \int\limits_{F(L\times\{t\})}\beta - \int\limits_L\beta, t).$
\end{proof}

\section{Smoothing Lavrentiev curves on surfaces}
\label{Lavrentiev-smoothing-section}
In this section we smooth Lavrentiev curves on surfaces.

Almost all work is already done by Tukia. He proved the so-called bi-Lipschitz version of the Sch\"onflies theorem for a circle on the plane. 

\begin{theo}[\cite{Tuk}, Theorem A; \cite{Tuk2}]
For any number $C$ there exists a number $C'$ such that any $C$-bi-Lipschitz map from $\partial\mathbb D$ to $\mathbb C$ can be extended to a $C'$-bi-Lipschitz self-homeomorphism of $\mathbb C$ that is either piecewise linear or smooth outside $\partial\mathbb D$.
\end{theo}

Of course one can choose one of the two alternatives, piecewise linear or smooth.
In~\cite{DanPrat} the strengthening of this theorem is proved where some explicit function $C'(C)$ is given. %See also \cite[Theorem 7.10]{Pomm} for the smooth variant of the theorem.

Since the natural parametrization of a Lavrentiev curve is a bi-Lipschitz map, we have an immediate consequence of Tukia's theorem.

\begin{coro}\label{Lavrentiev-bilipschitz}
For any number $C$ there exists a number $C'$ such that any closed $C$-Lavrentiev curve $L\subset \mathbb C$ is the image of the circle $\partial \mathbb D$ under the composition of a $C'$-bi-Lipschitz self-homeomorphism of $\mathbb C$ which is either piecewise linear or smooth outside $\partial\mathbb D$ and a homothety with a coefficient $\ell(L).$
\end{coro}

So to smooth a closed Lavrentiev curve on the plane by an isotopy we can conjugate an isotopy of $\partial \mathbb D$ given by the formula $(z,t)\mapsto (1-t/2)z$ by a map provided by Tukia's theorem. At the moment $t>0$ the curve is therefore either piecewise linear or smooth depending on which variant of Tukia's theorem we use.

A little more work is needed to smooth a Lavrentiev curve on a surface, especially when the curve is not coorientable.
We will smooth curves arc by arc, fixing one part of the curve while smoothing another part. It will be done, first, by the isotopies associated with bypasses provided by Tukia's theorem, and, second, by smoothing the obtained curve at the endpoints of the attaching arc of the bypass.

In Section~\ref{Legendrian-smoothing-section} we will lift the isotopy constructed in the present section to a Legendrian isotopy. So, first, in the present section we should preserve the integral of the form $\beta$ (see Subsection~\ref{regular-projection-section} on the regular projection) on the whole curve. This is easily satisfied by slightly deforming the curve by a smooth ambient isotopy (Proposition~\ref{correcting-isotopy-existence}). Second, since the Legendrian isotopy must be continuous, the integral of the form $\beta$ on any subarc of the curve must change continuously. This is automatically satisfied because the isotopy which we are going to construct is an isotopy of $C$-bi-Lipschitz curves with common $C$ (see Proposition~\ref{bi-Lipschitz-isotopy-lifting}). Third, the overall change of the integral of the form $\beta$ on any arc must be small to guarantee that the resulting Legendrian isotopy will be small. This fact will be needed in the proof of Lemmas~\ref{smooth-isotopy-lemma2} and~\ref{smooth-isotopy-lemma3}. Because of this issue either we can only apply the isotopy associated with the bypass on the short period of time (Lemma~\ref{first-smoothing-lemma}) and in this case a bad behavior of the curve preserves at the endpoints of the attaching arc or we isotope the attaching arc to the other half of the boundary of the bypass but in this case the bypass must be small (Lemma~\ref{second-smoothing-lemma}).
% So at the beginning we smooth the curve on big subarcs using Lemma~\ref{first-smoothing-lemma} preserving a bad behavior of the curve near the endpoints of the subarcs, and then with help of Lemma~\ref{second-smoothing-lemma} we smooth the curve at these endpoints. 

So the main result of this section is the following.

\begin{prop}\label{Lavrentiev-smoothing}
Let $\beta$ be a differential 1-form on the smooth surface $S$ without boundary. Let $L$ be a compact Lavrentiev curve on $S$. Let $d\beta\neq 0$ everywhere and $\varepsilon > 0.$ Let $\gamma_0\subset L$ be a compact subarc, $\partial \gamma_0 \cap \partial L = \varnothing$, $V$ be an open subset of $S$, $V\supset\gamma_0$. Then there exist a number $C$ and an isotopy $F:L\times[0;1]\to S$ such that
\begin{enumerate}
\item\label{LScond1} $F(\bullet, 0) = \mathrm{id}_L.$
\item\label{LScond2} $F(\bullet,t)$ is $C$-bi-Lipschitz for any $t\in[0;1].$
\item\label{LScond3} $F$ is Lipschitz.
%For any compact subarc $L'\subset L$ the function $\int\limits_{F(L'\times\{t\})}\beta$ on $t$ is continuous.
\item\label{LScond4} For any subarc $L'\subset L$ and any $t\in[0;1]$ it is true that $\left|\int\limits_{F(L'\times\{t\})}\beta - \int\limits_{L'}\beta\right|<\varepsilon.$
\item\label{LScond5} $\int\limits_{F(L\times\{t\})}\beta$ is independent of $t$.
\item\label{LScond6} $F(p, t) = p$ for any $p\in L\setminus V$ and any $t\in[0;1].$
\item\label{LScond7} $F((L\cap V)\times[0;1])\subset V.$
\item\label{LScond8} The arc $F(\gamma_0\times\{1\})$ is contained in the interior of a smooth subarc of $F(L\times\{1\}).$
\item\label{LScond9} If the subarc $\gamma$ of the curve $L$ is smooth then the arc $F(\gamma\times\{1\})$ is smooth.
\end{enumerate}
\end{prop}

\begin{proof}
Let $L_0\subset L$ be a compact subarc such that the curve $L$ has the tangents at the endpoints of $L_0$, $\gamma_0\subset L_0$, $\partial L_0\cap \partial L=\varnothing$ and $L_0\subset V.$ Such a curve exists since any Lavrentiev curve is differentiable almost everywhere by the discussion before Corollary~\ref{tangent-to-contact-plane-almost-everywhere}.

\begin{lemm}\label{first-smoothing-lemma}
Let $L$ be a compact Lavrentiev curve on a smooth surface without boundary, $L_0$ be a compact subarc of the curve $L$ which has one-sided tangents at the endpoints. Then there exists a bypass $\chi_0$ for the curve $L$ attached along the arc $L_0$ such that $\chi_0\big|_{\mathbb D}$ is piecewise smooth.
\end{lemm}

\begin{proof}
Since $L$ is Lavrentiev and $L_0$ has one-sided tangents at the endpoints, by Lemma~\ref{Lavrentiev-semitangent} the minimal angles between the arc $L_0$ and the curve $L\setminus(L_0\setminus\partial L_0)$ at their common endpoints are greater than zero. Let $\alpha_0$ be a minimum over these two angles.

Since the curve $L$ is embedded, the endpoints of $L_0$ can be joined by a smooth arc $L_1$ such that
\begin{enumerate}
\item The angles between the curves $L_0$ and $L_1$ at their common endpoints equal $\alpha_0/2$.
\item There exists a closed disk $D\subset S$ such that $\partial D = L_0\cup L_1$ and $D\cap L = L_0.$
\end{enumerate}

Since the arc $L_1$ is smooth, by Lemma~\ref{Lavrentiev-example} it is Lavrentiev. So the curve $\partial D$ is a union of the Lavrentiev arcs $L_0$ and $L_1$. These arcs has tangents at their common endpoints and the angles between them is greater than zero. Therefore by Lemma~\ref{Lavrentiev-concatenation} the curve $\partial D$ is locally Lavrentiev thus since it is compact, by Lemma~\ref{compact-Lavrentiev} it is Lavrentiev. Then by Corollary~\ref{Lavrentiev-bilipschitz} there exists a bi-Lipschitz homeomorphism $\chi_0:\overline{\mathbb D}\to D$ which is piecewise smooth in $\mathbb D.$

If we substitute $\chi_0$ with a composition of $\chi_0$ with some M\"obius transformation of $\mathbb CP^1$ which preserves $\mathbb D,$ we can also assume that $\chi_0(\partial_+ D) = L_0.$ Since any such M\"obius transformation is bi-Lipschitz on $\overline{\mathbb D}$, $\chi_0$ will stay bi-Lipschitz.

Similarly the curve $(L\setminus L_0)\cup L_1$ is Lavrentiev. Thus $\chi_0$ is a bypass.
\end{proof}

Let $F_0$ be the isotopy associated with a bypass $\chi_0$ given by Lemma~\ref{first-smoothing-lemma} applied to the subarc $L_0\subset L$ and an open subsurface of $V$ which contains $L_0$ and intersects $L$ by an arc. Conditions~(\ref{LScond2}) and~(\ref{LScond3}) of the present proposition are satisfied by the isotopy $F_0$ by Proposition~\ref{associated-isotopy-with-bypass-proposition}. Conditions~(\ref{LScond1}),~(\ref{LScond6}) and~(\ref{LScond7}) are satisfied by construction.

For any $C>1$ during any isotopy of $C$-bi-Lipschitz curves the integral of any smooth differential 1-form on any subarc of the curve changes continuously by Corollary~\ref{integral-continuity-corollary}. So we can fix a sufficiently small $\delta > 0$ such that the isotopy $F_0\big|_{L\times[0;\delta]}$ satisfies condition~(\ref{LScond4}).

We have some advance in satisfying condition~(\ref{LScond8}): the curve $F_0(\gamma_0\times\{\delta\})$ is piecewise smooth.

Now we address condition~(\ref{LScond5}).
First, we orient the surface $S$ using the 2-form $d\beta$ which is nonzero everywhere by the assumption of the present proposition. Second, we note that the isotopy $F_0$ always pushes the subarc $L_0$ to one side of the curve. So by Stokes theorem (see \cite[Chapter IX, Theorems 5A and 7A]{Whi}) the integral $\int\limits_{F(L\times\{t\})} \beta$ is a monotone function on $t$. We will smoothly push the curve to the opposite side to make this integral constant.

Choose two points $A,B\in L\cap V$ such that $L$ has tangents at $A$ and $B$ and some subarc of $L$ connecting $A$ and $B$ contains the subarc $L_0$. It is easy to construct a square $R_0 = \{(u,v):\ |u|\le1,\ |v|\le1\}$ smoothly embedded in the open set $V$ such that $R_0$ contains the bypass $\chi_0(\overline{\mathbb D})$ in its interior and $\partial R_0$ intersects $L$ in two points $A$ and $B$ lying on the opposite sides $v=\pm1$ of the square. It is clear that $R_0$ is an integral correction square (see Definition~\ref{integral-correction-square-definition}) for the isotopy $F_0$. Let $\Phi_0$ be an isotopy of $S$ supported in $R_0$ correcting the isotopy $F_0$. Since $F_0$ is Lipschitz, by Proposition~\ref{correcting-isotopy-existence} $\Phi_0$ is also Lipschitz.

So the isotopy $(p,t) \mapsto \Phi_0(F_0(p , t), t)$, where $p\in L$ and $t\in[0;\delta]$, satisfies conditions~(\ref{LScond1}), (\ref{LScond2}), (\ref{LScond3}), (\ref{LScond6}) and (\ref{LScond7}) since these conditions were satisfied by the isotopy $F_0$, and $\Phi_0$ preserves them by the definition. If we choose $\delta$ sufficiently small, it also satisfies condition (\ref{LScond4}). Condition~(\ref{LScond5}) is satisfied by construction. We preserved our advance in satisfying the condition~(\ref{LScond8}): the curve  $\Phi_0(F_0(\gamma_0\times\{\delta\}),\delta)$ is piecewise smooth since $\Phi_0(\bullet,\delta)$ is smooth.

Let $\partial L_0 = \{P, Q\}.$

\begin{lemm}\label{second-smoothing-lemma}
For any real number $C$ there exists a real number $C'$ such that if $L$ is a $C$-Lavrentiev curve on the surface $S$ with a Riemannian metric $g$ and $P\in L\setminus\partial L$, then there exists a bypass $\chi$ for $L$ such that
\begin{enumerate}
\item $P\in L_0\setminus \partial L_0$ where $L_0$ is the attaching arc of the bypass.
\item $\chi(\partial_- \mathbb D)$ is piecewise smooth.
\item $\chi$ is $C'$-bi-Lipschitz with respect to the Riemannian metric $g / \left(\ell(L_0)\right)^2.$
\end{enumerate}
\end{lemm}

\begin{proof}
We will construct a bypass in a small neighborhood $V$ of $P.$ Taking $V$ sufficiently small, we can assume that $g$ is Euclidean in $V$. For any $r$ let $B(r)$ denote the ball centered at $P$ of radius $r.$ 

Let $\gamma_1$ and $\gamma_2$ be two compact subarcs of $L$ such that $\gamma_1\cap \gamma_2 = \{P\},$ $\gamma_1\cup \gamma_2\subset L\cap V$ and for any two points $p_1\in \gamma_1,$ $p_2\in\gamma_2$ the subarc $L\big|_{p_1}^{p_2}$ contains $P$. We take $r_0$ such that $\gamma_1\cup\gamma_2\supset L\cap B(r_0)$ and $\partial(\gamma_1\cup\gamma_2)\cap B(r_0)=\varnothing.$ Let $O\in B(r_0/2)$ be a point such that $\mathrm{dist}(O, \gamma_1) = \mathrm{dist}(O, \gamma_2)\neq0.$ Let $Q_1\in\gamma_1$ and $Q_2\in\gamma_2$ be two points such that $\mathrm{dist}(O, Q_1) = \mathrm{dist}(O, \gamma_1) = \mathrm{dist}(O, \gamma_2) = \mathrm{dist}(O, Q_2).$ Let $D$ be a closed geometric disk centered at $O$ of radius $\mathrm{dist}(O, Q_1).$ It is clear that $\mathrm{dist}(O, Q_1)\le\mathrm{dist}(O, P)\le r_0/2$ and thus $D\subset B(r_0).$ Then $L\cap D = (\gamma_1\cup\gamma_2)\cap D$ and the interior of $D$ does not contain any point of $L$.

\begin{figure}[ht]
\center{\includegraphics{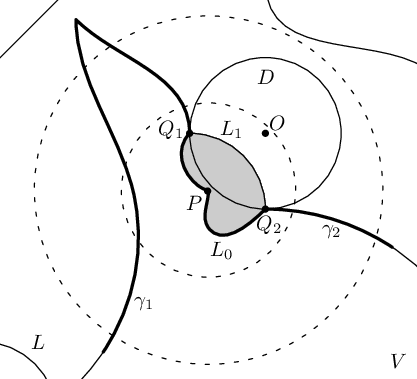}}
\caption{constructing a bypass in Case 1}
\label{second-smoothing-picture}
\end{figure}

\smallskip
\noindent{\bf Case 1.} $Q_1\neq P$ and $Q_2\neq P.$

Let $L_1$ be an arc of a (generalized) circle which is perpendicular to the circle $\partial D$ such that $\partial L_1 = \{Q_1, Q_2\}$ and $L_1\subset D.$ Let $L_0 = L\big|_{Q_1}^{Q_2}.$ See Figure~\ref{second-smoothing-picture}. The curve $L_0$ is $C$-Lavrentiev. The curve $L_1$ is $\pi/2$-Lavrentiev. Let us prove that the curve $L_0\cup L_1$ is Lavrentiev with the constant depending only on $C$.

Let $A\in L_0,$ $B\in L_1$ and $A,B\notin L_0\cap L_1.$ Suppose that $\mathrm{dist}(B, Q_1) \le \mathrm{dist}(B, Q_2).$ The other case is similar. Then

\begin{multline*}
\ell((L_0\cup L_1)\big|_A^B)\le
\ell(L_0\big|^{Q_1}_A) + \ell(L_1\big|^B_{Q_1}) \le C\cdot\mathrm{dist}(A, Q_1) + \pi/2\cdot\mathrm{dist}(Q_1,B) \le \\
\le C\cdot\mathrm{dist}(A,B) + (C + \pi/2)\cdot\mathrm{dist}(Q_1, B).
\end{multline*}

It is clear that there exists a universal constant $k$ such that $\mathrm{dist}(Q_1, B)\le k\cdot \mathrm{dist}(B, \partial D)$ (in fact, $k = \sqrt2$ works). Let us continue the above chain of inequalities:

$$
\ell((L_0\cup L_1)\big|_A^B) \le C\cdot\mathrm{dist}(A,B) + (C+\pi/2)k\cdot \mathrm{dist}(B, \partial D) < (C + (C+\pi/2)k)\cdot\mathrm{dist}(A,B).
$$

So the curve $L_0\cup L_1$ is $(C + (C+\pi/2)k)$-Lavrentiev.

Denote $\ell(L_0)$ by $l$. Note that $1\le\ell(L_0\cup L_1)/l\le 1 + \pi/2.$

If $X$ is a subset of $\mathbb C$ and $z\in\mathbb C$ denote by $z\cdot X$ the set $\{z\cdot w:\ w\in X\}.$

By Corollary~\ref{Lavrentiev-bilipschitz} there exists a $C_0$-bi-Lipschitz map $\chi_0:l\cdot\overline{\mathbb D}\to V$ such that $\chi_0(l\cdot\partial\mathbb D) = L_0\cup L_1$ and $C_0$ depends only on $C$. Let us prove that $\chi_0(l\cdot\mathbb D)\cap (L\cup L_1) = \varnothing.$ First, $\gamma_1\cup\gamma_2\setminus L_0$ is the union of two arcs. Second, each of these arcs emanates from $\partial L_0$ outside  $\chi_0(l\cdot\mathbb D)$ because $\gamma_1$ and $\gamma_2$ are compact arcs and $\partial(\gamma_1\cup\gamma_2)\cap B(r_0)=\varnothing.$ The rest follows from the fact that $\chi_0(l\cdot\mathbb D)\subset B(r_0)$ and $L\cap B(r_0) = (\gamma_1\cup\gamma_2)\cap B(r_0)$.

Now we consider the M\"obius map $\chi_1$ which maps $l\cdot\mathbb D$ to itself and the arc $\chi_0^{-1}(L_0)$ with its middle point to the arc $l\cdot\partial_+\mathbb D$ with its middle point respectively. There exists a real number $C_1$ which depends only on the ratio $\ell(\chi_0^{-1}(L_0))/\ell(l\cdot\partial\mathbb D)$ such that $\chi_1$ is $C_1$-bi-Lipschitz. But $\ell(\chi_0^{-1}(L_0))/\ell(l\cdot\partial\mathbb D)\ge (\ell(L_0)/C_0)/(2\pi \ell(L_0)) = 1/2\pi C_0.$ Hence $C_1$ depends only on $C$.

Let $\chi:\overline{\mathbb D}\to V$ be the map such that $\chi(z) = \chi_0\circ\chi_1^{-1}(lz).$ By construction $\chi$ is $C'$-bi-Lipschitz map with respect to the Euclidean metric on $V$ divided by $l^2$ where $C'$ depends only on $C$, $\chi(\overline{\mathbb D})\cap (L\cup L_1) = L_0\cup L_1$ and $\chi(\partial_+\mathbb D) = L_0.$ To check that $\chi$ is a bypass it remains only to prove that the curve $(L\setminus L_0)\cup L_1$ is Lavrentiev. It is true by Corollary~\ref{Lavrentiev-concatenation2} since the minimal angle of this curve at the points $\partial L_0$ is nonzero.

\smallskip
\noindent{\bf Case 2.} $P\in\partial D$.

Let $(r,\varphi)$ be polar coordinates in $V$ such that $r$ is the distance to the point $P$. We can apply a 2-bi-Lipschitz homeomorphism $\phi$ of $B(r_0)$ such that $\phi:(r,\varphi)\mapsto (r, \varphi'(\varphi))$, where the functions $\varphi'(\varphi)$ and $\varphi(\varphi')$ are smooth, and the inner angle of $\phi(D)$ at $P$ is greater than $\pi.$

The connected component of $\phi(L\cap B(r_0))$ containing $P$ divides $B(r_0)$ into two parts. One part contains $\phi(D).$ In the other part we can construct a disk $D'$ in a similar way as $D$. The disk $D'$ satisfies the assumption of the first case, so we can construct a bypass with desired properties. Then we compose this bypass with a map $\phi^{-1}$ and obtain a sought-for bypass for the curve $L$.
\end{proof}

Now we take a point $P\in\partial L_0$ and apply the last lemma to the point $\Phi_0(P, \delta)$ on the curve $L' = \Phi_0(F_0(L\times\{\delta\}),\delta).$ We get a bypass $\chi_1$ in a neighborhood of $\Phi_0(P, \delta)$ which can be chosen to be arbitrarily small. So we assume this neighborhood to lie inside $V$. Also we can assume the attaching arc of the bypass $\chi_1$ to be arbitrarily small. By construction of $\chi_1$ all arcs $\chi_1(F_{+\mapsto-}(\partial_+\mathbb D, t))$ have comparable lengths such that the constant of comparing depends only on $C$. So these arcs have the arbitrarily small length for all $t\in[0;1].$ Let $F_1$ be the isotopy associated with the bypass $\chi_1$. %The range of $\int\limits_{F_1(L'\times\{t\})}\beta$ can be assumed arbitrarily small.

There exists an integral correction square $R_1\subset V$ for the isotopy $F_1$ such that $R_1\cap \chi_1(\overline{\mathbb D}) = \varnothing.$ Since $L'\cap R_1$ stays fixed during the isotopy $F_1$, by Proposition~\ref{correcting-isotopy-existence} there exists a correcting isotopy $\Phi_1$ which essentially depends only on the function $\int\limits_{F_1(L'\times\{t\})}\beta$. So again by choosing the bypass $\chi_1$ to be sufficiently small we can achieve that the range of the integral of the form $\beta$ on any subarc of $L'$ during the isotopy $(p,t) \mapsto \Phi_1(F_1(p, t), t)$ is arbitrarily small. So condition~(\ref{LScond4}) is satisfied by the concatenation of the isotopies $(p,t) \mapsto \Phi_0(F_0(p, t), t),$ where $p\in L$ and $t\in[0;\delta]$, and $(p,t) \mapsto \Phi_1(F_1(p, t), t),$ where $p\in \Phi_0(F_0(L,\delta), \delta)$ and $t\in[0;1].$

In a similar way we take a point $Q\in\partial L_0\setminus\{P\}$ and by Lemma~\ref{second-smoothing-lemma} we construct a bypass $\chi_2$ in a neighborhood of the point $\Phi_1(\Phi_0(Q,\delta),1)$, correct the associated isotopy $F_2$ by an isotopy $\Phi_2.$ The concatenation of the three constructed isotopies satisfies conditions~(\ref{LScond1})--(\ref{LScond7}). And moreover, if a subarc of the curve $L$ was smooth then its image after the isotopy is a piecewise smooth arc. So the only thing that remains to prove is that we can smooth the corners of a piecewise smooth curve.

Smoothing the corners of a piecewise smooth curve is pretty standard so we only sketch the argument. Let $p$ be a corner of the curve $L$ and $(r,\varphi)$ be polar coordinates in some neighborhood of $p$, where $r = \mathrm{dist}(p,\bullet)$. Since the curve $L$ is Lavrentiev and hence cusp-free, it is possible to construct an isotopy of $S$ trivial outside the neighborhood of $p$ of the form $(r,\varphi) \mapsto (r, \varphi'(r, \varphi, t)),$ where $\varphi'(r, \varphi, t)$ is smooth, $\partial\varphi'/\partial\varphi\neq0$ and $\varphi'(r,\varphi, t) = \varphi'(\varphi, t)$ if $r<r_0$ for some positive $r_0$, from the curve $L$ to a curve which is straight near $p$ such that at any moment of time the curve $L$ is mapped to a piecewise smooth curve. By construction the isotopy is Lipschitz and there exists $C$ such that at any moment of time of this isotopy a $C$-bi-Lipschitz homeomorphism is applied. Again if the neighborhood was chosen sufficiently small this isotopy can be corrected with help of an integral correction square not intersecting the neighborhood so that we obtain an isotopy satisfying condition~(\ref{LScond4}). After eliminating all the corners we obtain an isotopy satisfying all the conditions of the present proposition.
\end{proof}

\begin{rema}
All constructed smooth curves in the last proposition are in fact $C^{\infty}$-smooth.
If $L$ is a $C^1$-smooth curve, we can make it $C^{\infty}$-smooth by an isotopy satisfying all conditions of the last proposition in the alternative way. By the standard argument there is an arbitrarily small $C^1$-smooth isotopy of $L$ making the curve $C^{\infty}$-smooth. Such isotopy trivially satisfies all conditions except~(\ref{LScond2}) and~(\ref{LScond5}). It satisfies condition~(\ref{LScond2}) by Lemma~\ref{Lavrentiev-C1-isotopy}. For condition~(\ref{LScond5}) we again can correct this isotopy preserving condition~(\ref{LScond4}).  
\end{rema}

\subsection{Regular tubular neighborhoods}

One of the consequences of Tukia's theorem is that a Lavrentiev curve on the plane has a bi-Lipschitz bi-collar that is smooth outside the curve. In this subsection we note that this immediately provides a similar bi-collar for a Lavrentiev curve on a surface if the curve is coorientable. We will need this in the following section in the proof of Proposition~\ref{r-invariant-regular-neighbourhood-existence}.

Let $L$ be a closed embedded curve on a surface $S$. Let us remind that a collar of $L$ is a continuous embedding $L\times[0;1)\to S$, where $p\times\{0\}\mapsto p$ for any $p\in L.$ We call two collars equivalent if they can be included in a finite sequence of collars in which for any two adjacent collars one collar is a subset of the other. A {\it coorientation} of the curve $L$ is an equivalence class of collars. If there exists at least one collar, the curve is called {\it coorientable}.

It is known that any closed embedded curve on a surface has a tubular neighborhood which is homeomorphic either to an annulus or to a M\"obius strip (to see this one can, for example, first construct a collar locally on the orientable covering of the surface using Sch\"onflies theorem and then apply the fact from \cite[Theorem V.4.C]{Bing} that a subset of a metric space is collared if it is locally collared). In the first case the curve is coorientable while in the second case it is not.

\begin{defi}
\label{closed-regular-tubular-neighbourhood-definition}
Let $L$ be a coorientable closed Lavrentiev curve on a smooth surface $S$. A subset $N\subset S$ is called {\it closed regular tubular neighborhood} of the curve $L$ if there exists a map $f: \mathbb S^1\times[-1;1] \to N$ such that
\begin{enumerate}
\item $f$ is a bi-Lipschitz homeomorphism.
\item $f(\mathbb S^1\times\{0\}) = L$.
\item $f: \mathbb S^1\times ([-1;0) \cup (0;1]) \to N\setminus L$ is a diffeomorphism.
\end{enumerate}
\end{defi}

Let us prove that such neighborhood exists.

Since the curve is coorientable, it has a tubular neighborhood which is homeomorphic to an annulus. So we can assume that the surface $S$ is homeomorphic to an annulus. And we can also assume that this annulus is embedded in the plane. Tukia's theorem provides us a closed regular tubular neighborhood of the curve $L$ lying in the plane, and if we take a sufficiently small such neighborhood it lies inside the annulus hence in $S$.

\section{Regular neighborhoods of Legendrian curves}
\label{regular-neighborhood-section}
%From the definition of a Legendrian curve (Definition~\ref{def-angle}) it follows that, locally, the projection of a Legendrian curve to a contact plane is injective. The same statement holds simultaneously for curves in some Legendrian isotopy. 
In this section we construct a regular projection in some neighborhood of any Legendrian curve. This allows one to apply results of Subsection~\ref{regular-projection-section} to any contact manifolds.
%  In Subsection~\ref{regular-projection-section} we prove that a Legendrian curve can be recovered from its regular projection. Of course  So locally Legendrian curves "are" curves on the plane. In this subsection we make this assertion global. More precisely, we prove that for any Legendrian curve there exists a surface and a neighborhood of this surface where the regular projection to the surface is defined such that the regular projection of this curve to the surface is injective (Proposition~\ref{regular-neighbourhood-existence}). This allows us to apply the results of Subsection~\ref{regular-projection-section} globally.

\begin{defi}
A vector field on a contact manifold is called {\it contact} if the flow of the vector field preserves the contact structure.

We call an unordered pair of two opposite vectors a {\it line element}.

A smooth field of line elements is called {\it contact} if the restriction of the field to any simply connected open subset $U$  has the form $\{v,-v\}$, where $v$ is a contact vector field.
\end{defi}

For a field of line elements $l$ let $l_p$ denote the line element at the point $p.$
Let $A$ be a subset of the manifold $M$. Let $\overline B_{rl}(A)$ (respectively, $B_{rl}(A)$) denote the set of vectors $\{\lambda v:\ v\in l_p,\ p\in A,\ |\lambda|\le r\}$ (respectively, $\{\lambda v:\ v\in l_p,\ p\in A,\ |\lambda|< r\}$).
If $A$ is compact and $A\cap\partial M = \varnothing$, the flow of the field $l$ is defined on $B_{rl}(A)$ for some $r>0$. We denote this flow by $\exp_l:B_{rl}(A) \to M$.

\begin{defi}\label{regular-neighbourhood-definition}
A neighborhood $U$ of the Legendrian curve $L$ is called {\it regular} if there exists a positive number $r\le+\infty$, a smooth surface $S\subset U,$ the map $\mathrm{pr}:U \to S$ and a contact field of line elements $l$ such that
\begin{enumerate}
\item $l$ is transverse to the contact structure.
\item\label{regular-neighborhood-condition2} $\exp_l: B_{rl}(S) \to U$ is a diffeomorphism.
\item $\mathrm{pr}$ is the projection along the trajectories of $l$.
\item\label{regular-neighborhood-condition4} $\mathrm{pr}\big|_L$ is injective.
\end{enumerate}

The quadruple $(U, S, l, r)$ will be also called regular neighborhood.

The projection $\mathrm{pr}$ is called {\it regular}.
\end{defi}

We note that from condition~(\ref{regular-neighborhood-condition2}) it follows that the field $l$ is transverse to the surface $S$. 

\begin{rema}\label{regular-neighbourhood-contact-form}
Suppose that the contact structure is coorientable along the curve $L$ and $(U,S,l,r)$ is a regular neighborhood of the curve $L$. Since $L$ and $\mathrm{pr}(L)$ are homotopic to each other, the contact structure is also coorientable on the curve $\mathrm{pr}(L).$ Since the field $l$ is transverse to the contact structure, $l$ has the form $\{v, -v\}$, where $v$ is a vector field in some neighborhood of $\mathrm{pr}(L)$ in $U$. Since $l$ is transverse to $S$, $v$ provide a coorientation for some neighborhood $S'\subset S$ of the curve $\mathrm{pr}(L)$. Thus $S'$ is oriented and $B_{rl}(S')$ can be identified with $S'\times(-r;r)$ where $T_q M \ni z\cdot v  \leftrightarrow (q, z).$ Since the field $l$ is contact and it is transverse to the contact structure, there exists a differential 1-form $\beta$ on $S'$ such that the contact structure on $U' = \exp_l(S'\times(-r;r))$ is the kernel of the 1-form
$
dz + \mathrm{pr}^*\beta.
$
\end{rema}

\begin{prop}\label{regular-neighbourhood-existence}
Any compact Legendrian curve lying in the interior of the manifold has a regular neighborhood.
\end{prop}

On any contact manifold there exists a contact field of line elements which is transverse to the contact structure (Lemma~\ref{Reeb-line-element-existence}).

Therefore Proposition~\ref{regular-neighbourhood-existence} is a corollary of the following lemma.

\begin{lemm}\label{regular-neighbourhood-existence-lemma}
Let $L$ be a compact Legendrian curve lying in the interior of the contact manifold $(M, \xi).$ Let $l$ be a contact field of line elements transverse to $\xi$. Then the field $l$ is the contact field of some regular neighborhood $U$ of the curve $L.$
\end{lemm}

\begin{proof}
Since $L$ is Legendrian, by Lemma~\ref{transverse-collar2} there exists $\varepsilon>0$, such that $\exp_l:\overline B_{\varepsilon l}(L) \to M$ is an embedding. 

Thus by Lemma~\ref{surface-of-trajectories2} there exist a smooth embedded surface $S\subset M$ which is transverse to the field $l$ and a neighborhood $U$ of the curve $L$ such that the map $\exp_l:B_{\varepsilon l}(S)\to U$ is a diffeomorphism.

It remains to prove that $\mathrm{pr}\big|_L$ is injective. Suppose the contrary, that for some point $q\in S$ there are two points $p_1, p_2\in L$ such that the points in each pair $q, p_1$ and $q, p_2$ are joined by a trajectory having the length less the $\varepsilon.$ This contradicts the fact that $\exp_l: B_{\varepsilon l}(L) \to M$ is an embedding.
\end{proof}

Now we move to the proof of the lemmas used above.

First, we prove lemmas~\ref{transverse-collar2} and~\ref{surface-of-trajectories2} in a particular case, when the field $l$ is a pair of vector fields (lemmas~\ref{transverse-collar} and~\ref{surface-of-trajectories}). Essentially the same proof works in the general case and we will indicate necessary changes.

Let $Z$ be a vector field on a manifold $M$.
Let $\Phi$ denote the flow of the vector field $Z$. For a subset $A\subset M$ and a pair $a\le b$ of real numbers we write $\Phi:A\times[a;b]\to M$ for the map such that $\Phi:(p,z)\mapsto \exp_{\{Z,-Z\}}(z\cdot Z_p)$ where $Z_p$ is a vector $Z$ at the point $p$.

By the length of a trajectory $\Phi(\{p\}\times [a;b])$ we call the number $|a-b|$.

\begin{lemm}
\label{transverse-collar}
Let $Z$ be a smooth vector field on the manifold $M$ transverse to the contact structure, $L$ be a compact Legendrian curve in the interior of $M$. Then there exists a positive real number $\varepsilon$ such that the map $\Phi:L\times[-\varepsilon;\varepsilon]\to M$ is defined and is injective.
\end{lemm}

\begin{proof}
For any point $p\in L$ we choose a chart $V_p$ with coordinates $(x,y,z)$ such that $\partial /\partial z = Z$, $p = (0,0,0)$ and the coordinate plane $Oxy$ coincides with $\xi_p.$

By the definition of Legendrian curve (Definition~\ref{def-angle}) there exists a neighborhood $W_p\subset V_p$ such that the orthogonal projection to the plane $\xi_p = Oxy$ is injective on $L\cap W_p.$ Let $\varepsilon_p$ be a positive real number and $D_p$ be an open disk in the plane $Oxy$ centered at $p$ such that $\Phi(D_p\times (-\varepsilon_p;\varepsilon_p))\subset W_p.$ We note that for any point $q\in L\cap \Phi(D_p\times (-\varepsilon_p/2;\varepsilon_p/2))$ the trajectory $\Phi(\{q\}\times[-\varepsilon_p/2;\varepsilon_p/2])$ is defined and embedded in $\Phi(D_p\times (-\varepsilon_p;\varepsilon_p))$, and it intersects $L$ at the unique point $q$.

Since $L\subset\bigcup\limits_{p\in L} \Phi(D_p\times (-\varepsilon_p/2;\varepsilon_p/2))$ and the curve $L$ is compact, there exist a finite number of points $\{p_i\}\subset L$ such that $L\subset\bigcup\limits_{i} \Phi(D_{p_i}\times (-\varepsilon_{p_i}/2;\varepsilon_{p_i}/2))$.

Let $\varepsilon = \min\limits_i \varepsilon_{p_i}/4$. Then $\Phi:L\times[-\varepsilon;\varepsilon]\to M$ is an embedding.
\end{proof}

\begin{lemm}
\label{surface-of-trajectories}
Let $Z$ be a smooth vector field on the 3-manifold $M$ which is nonzero everywhere, $\varepsilon > 0$ and $C$ be a compact curve in $M$, such that the map $\Phi: C\times [-\varepsilon;\varepsilon]\to M\setminus\partial M$ is defined and injective. Then there exists a neighborhood $U$ of $C$ and a smooth surface $S\subset U$, such that
$\Phi: S\times (-\varepsilon;\varepsilon)\to U$ is a diffeomorphism.
\end{lemm}

We give a plan of the proof. Note that the space of connected components of trajectories of the vector field $Z$ in a sought-for neighborhood $U$ is homeomorphic to a surface.
For any point it is easy to construct a neighborhood having the space of trajectories homeomorphic to a disk and each trajectory homeomorphic to an interval. We will construct a neighborhood $U$ as a subset of a finite union of such neighborhoods which cover the curve. A difficulty is that even the union of two such neighborhoods may have the space of trajectories being non-Hausdorff. We give an example.

\begin{exam}\label{non-Hausdorff-example}
Let $M = \mathbb R^3$ with coordinates $(x,y,z)$, $Z = \partial/\partial z,$ $U_1 = \{(x,y,z):\ |z-y|<1,\ 0<y<2\},$ $U_2 = \{(x,y,z):\ |z+y|<1,\ 0<y<2\}.$ Then the trajectories $\{(x,y,z):\ x=0,\ y = 1,\ 0<z<2\}$ and $\{(x,y,z):\ x=0,\ y = 1,\ -2<z<0\}$ have no non-intersecting neighborhoods in the space of all trajectories of the vector field $Z$ in $U_1\cup U_2$, see Figure~\ref{non-Hausdorff} on the left.
\end{exam}

\begin{figure}
\center{\includegraphics{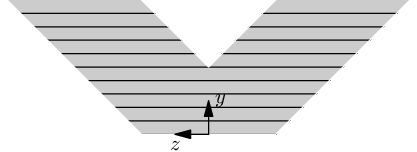}\hspace{1cm} \includegraphics{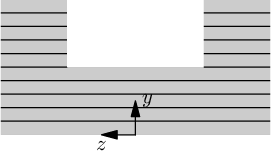}}
\caption{The space of trajectories is not Hausdorff}
\label{non-Hausdorff}
\end{figure}

We can see that in this example both neighborhoods are products $U_i = \Phi(D_i\times [-1;1])$ for some disks $D_i$ where $i = 1,2.$ Somewhere the disks are joined by short trajectories of the vector field $Z$ and somewhere they are joined by long trajectories. While taking a union of such neighborhoods we will forbid this in condition~\ref{Sj-V} in Claim~\ref{Sj-definition}. We will require that trajectories joining the curve with the disks are short. This is Condition~\ref{Sj-II}. This guarantees that trajectories joining the disks which intersect the curve are short. To guarantee that trajectories, which are close to the curve but do not intersect the curve, joining the disks are short, we require all disks to be "almost horizontal". This is the condition~\ref{Sj-IIIb}. That means that in all charts that we fix in Claim~\ref{atlas} the range of the coordinate $z$ on any disk is small. We achieve the "almost horizontality" by requiring that all disks are small.

\begin{exam}\label{non-Hausdorff-example2}
Let $M = \mathbb R^3$ with coordinates $(x,y,z)$, $Z = \partial/\partial z,$ $U_0 = \{(x,y,z): 0 < y < 1,\  -2 < z < 2\},$ $U_1 = \{(x,y,z):\ 0 < y < 2,\ 1<z<2\},$ $U_2 = \{(x,y,z):\ 0 < y < 2,\ -2<z<-1\}.$ Then the trajectories $\{(x,y,z):\ x=0,\ y = 1,\ 1<z<2\}$ and $\{(x,y,z):\ x=0,\ y = 1,\ -2<z<-1\}$ have no non-intersecting neighborhoods in the space of all trajectories of the vector field $Z$ in $U_0\cup U_1\cup U_2$, see Figure~\ref{non-Hausdorff} on the right.
\end{exam}

This is another example of non-Hausdorff space of trajectories in the union of neighborhoods of the form $U_i = \Phi(D_i\times (-r_i, r_i)).$ In this example we see that some (actually any) trajectory in $U_0$ contains a trajectory of $U_1$ and a trajectory of $U_2.$ We forbid this type of situation in condition~\ref{Sj-IV} in Claim~\ref{Sj-definition}. 

In claims~\ref{atlas}--\ref{last-claim} we follow the assumptions and the notations of Lemma~\ref{surface-of-trajectories}. %For a finite set $X = \{x_i\}_i$ of any objects $x_i$ we will ignore the number of objects in the set by specifying only the variable by which the set is indexed. 
We will also assume the manifold $M$ to be compact and that all constructed objects lie in the interior of the manifold. We fix some Riemannian metric on $M$.

Since $\Phi:C\times[-\varepsilon;\varepsilon]\to M\setminus\partial M$ is defined and injective and $C$ is compact, there exists $\delta_1>0$ such that the map $\Phi:C\times[-\varepsilon-\delta_1;\varepsilon+\delta_1]\to M\setminus\partial M$ is defined and injective. We set

\begin{equation}
\label{delta-definition}
\delta = \min(\delta_1, \varepsilon/2).
\end{equation}

\begin{clai}
\label{atlas}
There exists a finite set $\{U_i, D_i, \mathrm{pr}_i, z_i\}_i$ where $U_i$ is an open subset of $M$, $D_i$ is a smooth embedded disk in $U_i$, $\mathrm{pr}_i$ is a smooth map $U_i\to D_i,$ $z_i$ is a smooth uniformly continuous function $U_i\to (-\varepsilon/2; \varepsilon/2)$ such that
\begin{enumerate}[label=\arabic*), ref=\arabic*)]
\item\label{atlas2} $\Phi\circ\left(\mathrm{pr}_i\times z_i\right) = \mathrm{id}_{U_i}$ and $\left(\mathrm{pr}_i\times z_i\right)\circ\Phi\big|_{D_i\times(-\varepsilon/2; \varepsilon/2)} = \mathrm{id}_{D_i\times(-\varepsilon/2; \varepsilon/2)}$ for any $i$;
%\item\label{atlas1} $\mathrm{pr}_i$ is the projection along the trajectories of the vector field $Z$;
\item\label{atlas3} for any $i$ and $j$ the relation $\mathrm{pr}_j\circ\mathrm{pr}_i^{-1}:\mathrm{pr}_i (U_i\cap U_j) \to \mathrm{pr}_j (U_i\cap U_j)$ is a map and, moreover, it is a diffeomorphism;
\item $C \subset \bigcup\limits_i U_i$; \label{atlas4}
\item\label{atlas5} $\mathrm{pr}_i: C\cap U_i \to D_i$ is injective for any $i$;
\item\label{atlas6} if $p\in C\cap U_i$, $p\in\Phi(D_i\times(-\delta/8;\delta/8)).$
%$\Phi(\{p\}\times [-\varepsilon/4;\varepsilon/4])$ is defined and embedded in $U_i$.
\end{enumerate}
\end{clai}

\begin{proof}
For any $p\in C$ we choose a chart $V_p$ containing the trajectory $\Phi(\{p\}\times[-\varepsilon;\varepsilon])$ with coordinates $(x,y,z)$ such that $Z = \partial/\partial z.$ There exists a disk $D_p$ in the plane $z = z(p)$ centered at $p$ such that $\Phi(D_p\times (-\varepsilon;\varepsilon))\subset V_p.$ Since $\Phi: \{p\}\times[-\varepsilon;\varepsilon]\to M$ is an embedding, the disk $D_p$ can be chosen sufficiently small such that $|z(q)-z(p)|<\delta/8$ for any point $q\in C\cap \Phi(D_p \times(-\varepsilon;\varepsilon)).$ Let $U_p = \Phi(D_p \times(-\varepsilon/2;\varepsilon/2)).$ We can also choose $D_p$ sufficiently small so that $\overline U_p\subset V_p$ and thus $z$ is uniformly continuous in $U_p$.

Let $\mathrm{pr}_p$ denote the projection of $U_p$ to $D_p$ along the $z$-axis. The map $\mathrm{pr}_p\big|_{C\cap U_p}$ is injective since in the other case some points $q_1, q_2\in C\cap U_p$ with $\mathrm{pr}_p (q_1) = \mathrm{pr}_p (q_2)$ are joined by a trajectory of the field $z$ having the length less than $\varepsilon$ which contradicts the fact that $\Phi:\left(C\cap U_p\right)\times[-\varepsilon;\varepsilon]\to M$ is injective.

%If $q\in C\cap U_p,$ then $\Phi(\{q\}\times [-\varepsilon/4;\varepsilon/4])\subset U_p$ by the construction.

Let $p,p'\in C$. Let us prove that $\mathrm{pr}_{p'}\circ\mathrm{pr}_p^{-1} : \mathrm{pr}_p (U_p\cap U_{p'}) \to \mathrm{pr}_{p'} (U_p\cap U_{p'})$ is a map. Suppose the contrary. Let $q_0, q_1\in U_p\cap U_{p'}$, $\mathrm{pr}_p (q_0) = \mathrm{pr}_p (q_1)$ but $\mathrm{pr}_{p'} (q_0)\neq \mathrm{pr}_{p'} (q_1).$ The intervals $\Phi(\{q_0\}\times (-\varepsilon/2;\varepsilon/2))$ and $\Phi(\{q_1\}\times (-\varepsilon/2;\varepsilon/2))$ are embedded in $\Phi(D_p\times(-\varepsilon;\varepsilon))$ by construction, the former equality means that they intersect each other but the latter inequality means that they do not intersect each other. A contradiction.

Let us prove that the map $\mathrm{pr}_{p'}\circ\mathrm{pr}_p^{-1} : \mathrm{pr}_p (U_p\cap U_{p'}) \to \mathrm{pr}_{p'} (U_p\cap U_{p'})$ is smooth. Let $(x', y', z')$ be the coordinates in the chart $V_{p'}$. We have shown that the functions $x(x',y'),$ $y(x',y')$ are defined. Since they are the transition functions between two smooth charts $V_p$ and $V_{p'}$, they are smooth.

The composition of maps $\mathrm{pr}_{p'}\circ\mathrm{pr}_p^{-1} : \mathrm{pr}_p (U_p\cap U_{p'}) \to \mathrm{pr}_{p'} (U_p\cap U_{p'})$ and $\mathrm{pr}_{p}\circ\mathrm{pr}_{p'}^{-1} : \mathrm{pr}_{p'} (U_{p}\cap U_{p'}) \to \mathrm{pr}_{p} (U_{p}\cap U_{p'})$ is the identity, and these maps are smooth, thus they are diffeomorphisms.

Since $C$ is compact, it is covered by a finite number of the constructed charts $U_{p_i}.$ We set $U_i = U_{p_i},$ $D_i = D_{p_i}.$ If $(x,y,z)$ are the coordinates in the chart $V_{p_i}$, we set $z_i = z - z(p_i).$
\end{proof}

We fix such $d_0>0$ that for any $i$

\begin{equation}\label{zi-range-bound}
\mathrm{dist}(p,q) < d_0 \implies |z_i(p)-z_i(q)| < \delta/8
\end{equation}
if $z_i$ is defined at the points $p$ and $q.$ Such $d_0$ exists since the functions $z_i$ are uniformly continuous.

Let 
\begin{equation}\label{d1-definition}
d_1 = \inf\limits_{p\in C} \max\limits_i\mathrm{dist}(p, M \setminus U_i).
\end{equation}

Since $C$ is compact and is contained in $\bigcup\limits_i U_i$, $d_1>0.$

\begin{clai}
\label{Sj-definition}
There exists a finite set $\{S_j\}_j$ of smooth embedded surfaces without boundary indexed by $j\in\{0,1,\dots,N-1\}$ if the curve $C$ is homeomorphic to the segment and indexed by $j\in\mathbb Z_N$ if the curve $C$ is homeomorphic to the circle, such that
\begin{enumerate}[label = \Roman*., ref = \Roman*]
\item\label{Sj-I} $S_j$ is transverse to the vector field $Z$ for any $j$.
\item\label{Sj-II} $C\subset\bigcup\limits_j\Phi(S_j\times(-\delta/4;\delta/4))$.
\item\label{Sj-III} If $S_j\subset U_i$
\begin{enumerate}[ref = \Roman{enumi}(\alph{enumii})]
\item\label{Sj-IIIa} the map $\Phi:S_j\times[-\frac34\delta;\frac34\delta]\to U_i$ is defined and injective,
\item\label{Sj-IIIb} the range of $z_i$ on $S_j$ is less than $\delta/8$,
\item\label{Sj-IIIc} the map $\mathrm{pr}_i:S_j\to D_i$ is injective and open (the image of any open subset is open).
\end{enumerate}
\item\label{Sj-IV} If $S_{j-1} \cup S_{j+1}\subset U_i,$ then $\mathrm{pr}_i (S_{j-1})\cap\mathrm{pr}_i (S_{j+1})=\varnothing.$
\item\label{Sj-V} If $S_j\cup S_{j+1}\subset U_i,$ $P_0\in S_j,$ $P_1\in S_{j+1}$ and $\mathrm{pr}_i (P_0) = \mathrm{pr}_i (P_1),$ then $|z_i(P_0) - z_i(P_1)|<\frac34\delta.$
\item\label{curve-b} Any triple of the surfaces $S_{j-1},$ $S_{j}$ and $S_{j+1}$ lie inside some $U_i$.
\end{enumerate}
\end{clai}

\begin{proof}
By condition~\ref{atlas4} in Claim~\ref{atlas} for any $p\in C$ we can choose a chart $U_i$ which contains $p$. Let $W_p$ be an open subset of $D_i$ such that $\mathrm{pr}_i(p)\in W_p$, $\mathrm{diam}(\Phi(W_p\times\{z_i(p\}))< \min(d_0, d_1/3)$ and $\mathrm{diam}(C\cap\Phi(W_p\times(-\varepsilon/2;\varepsilon/2))<d_1/3.$

Let us check conditions~\ref{Sj-I} and~\ref{Sj-III} for the surfaces $\Phi(W_p\times\{z_i(p)\}).$ Clearly any such surface is transverse to $Z$. Suppose that $\Phi(W_p\times\{z_i(p)\})\subset U_{i'}.$ 

We begin with condition~\ref{Sj-IIIb}. Since $\mathrm{diam} (\Phi(W_p\times\{z_i(p\})) < d_0$, the range of $z_{i'}$ on this surface is less than $\delta/8$ by equation~(\ref{zi-range-bound}).

Then we consider condition~\ref{Sj-IIIa}. Since $p$ lies in our surface, $p\in U_{i'}$. Since $p\in C\cap U_{i'}$, by~\ref{atlas6} in Claim~\ref{atlas} $\Phi(\{p\}\times [-\varepsilon/2 + \delta/8; \varepsilon/2-\delta/8])\subset U_{i'}.$ We already proved that the range of $z_{i'}$ on the surface is less than $\delta/8.$ Again since $p$ lies in the surface, $\Phi(\{P\}\times [-\varepsilon/2 + \delta/4; \varepsilon/2-\delta/4])\subset U_{i'}$ for any point $P$ in the surface. We obtained what we needed because $\delta\le\varepsilon/2.$

Condition~\ref{Sj-IIIc} follows from condition~\ref{atlas3} in Claim~\ref{atlas} because $W_p\subset D_i.$

Now we consider an open cover $C = \bigcup\limits_{p\in C} C\cap \Phi(W_p\times(-\varepsilon/2;\varepsilon/2)).$ Suppose $\{\gamma_j\}_j$ is a subdivision of the curve $C$ into compact subarcs $\gamma_j$ indexed by $j\in\{0,1,\dots, N-1\}$ if $C$ is homeomorphic to a segment and indexed by $j\in\mathbb Z_N$ if $C$ is homeomorphic to a circle such that any two subarcs have a common endpoint if and only if their indices differ by one and each subarc is contained in some element of the cover. Such subdivision exists because $C$ is compact. Let $p_j\in C$ denote such point that $\Phi(W_{p_j}\times(-\varepsilon/2;\varepsilon/2))\supset \gamma_j.$ Let $i(j)$ be such number that $W_{p_j}\subset D_{i(j)}.$ Let $S_j$ denote the surface $\Phi(W_{p_j}\times\{z_{i(j)}(p_j)\}).$

Let us check condition~\ref{Sj-II} for the set of the surfaces $\{S_j\}_j.$ By construction $\gamma_j\subset \Phi(W_{p_j}\times(-\varepsilon/2;\varepsilon/2)) \subset U_{i(j)}.$ By condition~\ref{atlas6} in Claim~\ref{atlas} the range of $z_i$ on $C\cap U_{i(j)}$ is less than $\delta/4.$ Therefore, since $p_j\in C\cap U_{i(j)},$ $|z_{i(j)}(p)-z_{i(j)}(p_j)|<\delta/4$ for any $p\in\gamma_j.$ Thus $\gamma_j \subset \Phi(S_j\times(-\delta/4;\delta/4)).$ So we are done because $C = \bigcup\limits_j \gamma_j.$

Let us check condition~\ref{Sj-V}. Let $q$ denote the common endpoint of the arcs $\gamma_j$ and $\gamma_{j+1}.$ The segment $\Phi(\{q\}\times(-\delta/4;\delta/4))$ intersects the surfaces $S_j$ and $S_{j+1}$ in some points $Q_0$ and $Q_1$ respectively by discussion above. Such points are unique by~\ref{Sj-IIIa}. And by~\ref{Sj-IIIa} also, $q\in U_i.$ We have $|z_i(Q_0)-z_i(q)|<\delta/4$ and $|z_i(Q_1)-z_i(q)|<\delta/4.$ Thus $|z_i(Q_0) - z_i(Q_1)| < \delta/2.$ By~\ref{Sj-IIIb}, $|z_i(P_0) - z_i(Q_0)| < \delta/8$ and $|z_i(P_1) - z_i(Q_1)| < \delta/8$. Thus $|z_i(P_0)-z_i(P_1)|<\frac34\delta.$

Let us check condition~\ref{curve-b}. By the definition of $d_1$ (equation~(\ref{d1-definition})) there exists $i$ such that $\mathrm{dist}(p_j, M\setminus U_i)\ge d_1.$ We have $S_j\subset U_i$ because $p_j\in S_j$ and $\mathrm{diam}(S_j) < d_1/3 < d_1$.  Let $q_0$ be the common endpoint of the arcs $\gamma_{j-1}$ and $\gamma_j$. Since $\gamma_j\subset C\cap \Phi(W_{p_j}\times (-\varepsilon/2;\varepsilon/2))$ and $\mathrm{diam}(C\cap \Phi(W_{p_j}\times (-\varepsilon/2;\varepsilon/2))) < d_1/3,$ $\mathrm{dist}(p_j, q_0) < d_1/3.$ Similarly, $\mathrm{dist}(q_0, p_{j-1})< d_1/3.$ So $\mathrm{dist}(p_j, p_{j-1})<\frac23 d_1.$ Since $p_{j-1}\in S_{j-1}$ and $\mathrm{diam}(S_{j-1}) < d_1/3,$ $\max\limits_{P\in S_{j-1}} \mathrm{dist}(p_j, P) < d_1.$ Thus $S_{j-1} \subset U_i.$ For the surface $S_{j+1}$ the proof is similar.

To fulfill condition~\ref{Sj-IV} we need to make the neighborhoods $W_j$ (and the surfaces $S_j$ respectively) smaller. We will not change the points $p_j$ and the subarcs $\gamma_j.$ We will do this in such a way that it will remain to hold that $\gamma_j\subset \Phi(W_{p_j}\times (-\varepsilon/2; \varepsilon/2)).$ It is clear that the other conditions remain to hold in this case.

Suppose $S_{j-1}\cup S_{j+1}\subset U_i.$ Recall that %$\mathrm{pr}_{i(j\pm1)}\gamma_{j\pm1}\subset S_{j\pm1},$
 $\gamma_{j\pm1}\subset\Phi(S_{j\pm1}\times (-\delta/4;\delta/4))$ and by~\ref{Sj-IIIa} we have $\Phi(S_{j\pm1}\times (-\frac34\delta;\frac34\delta))\subset U_i$. Thus $\gamma_{j\pm1}\subset U_i.$ By condition~\ref{atlas5} in Claim~\ref{atlas} the projection $\mathrm{pr}_i:C\cap U_i\to D_i$ is injective. Therefore $\mathrm{pr}_i (\gamma_{j-1})$ and $\mathrm{pr}_i (\gamma_{j+1})$ are embedded disjoint compact arcs in $D_i.$ Let $N_{-1}$ and $N_1$ be their non-intersecting neighborhoods.
Let us prove that $\mathrm{pr}_{i(j\pm1)}\gamma_{j\pm 1}\subset S_{j\pm1}\cap \mathrm{pr}_i^{-1} (N_{\pm1}).$
  By condition~\ref{atlas6} in Claim~\ref{atlas} $\mathrm{pr}_i (\gamma_{j\pm1})\subset\Phi(\gamma_{j\pm1}\times (-\delta/8;\delta/8))$ and $\mathrm{pr}_{i(j\pm1)} (\gamma_{j\pm1})\subset\Phi(\gamma_{j\pm1}\times (-\delta/8;\delta/8))$. Thus 
%  $\mathrm{pr}_i (\gamma_{j\pm1})\subset\Phi(S_{j\pm1}\times (-\frac38\delta;\frac38\delta))\subset U_i.$ In other words,
   each point of the curve $\mathrm{pr}_i(\gamma_{j\pm1})$ is joined with a point on the curve $\mathrm{pr}_{i(j\pm1)}(\gamma_{j\pm1})$ by a trajectory of the field $Z$ having the length less than $\delta/4.$ Since $\mathrm{pr}_i(\gamma_{j\pm1})\subset D_i$, $\delta/4 < \varepsilon/2$ and $U_i = \Phi(D_i\times(-\varepsilon/2;\varepsilon/2)),$ this trajectory lies inside $U_i$. Thus $\mathrm{pr}_{i(j\pm1)}\gamma_{j\pm 1}\subset \mathrm{pr}_i^{-1} (N_{\pm1}).$ 
%   Since $\Phi(\mathrm{pr}_i(\gamma_{j\pm1})\times (-\varepsilon/2;\varepsilon/2))\subset \mathrm{pr}_i^{-1}(N_{\pm 1})$ and $\varepsilon/2 > \frac38\delta$, we have $\mathrm{pr}_{i(j\pm1)}\gamma_{j\pm 1}\subset\mathrm{pr}_i^{-1} (N_{\pm1}).$ 
   It is clear that $S_{j\pm1}\cap \mathrm{pr}_i^{-1}(N_{\pm1})$ is open in $S_{j\pm 1}$ and contains $\mathrm{pr}_{i(j\pm1)}(\gamma_{j\pm1}).$ So we substitute $S_{j\pm1}$ by $S_{j\pm1}\cap \mathrm{pr}_i^{-1}(N_{\pm1})$. We obtain $\mathrm{pr}_i(S_{j-1})\cap\mathrm{pr}_i(S_{j+1})=\varnothing.$ If we perform this procedure for each such pair $i$ and $j$, we obtain a sought-for set $\{S_j\}_j.$
\end{proof}

Let $X$ be a topological space $\bigsqcup\limits_j S_j /\sim$ where $S_k\ni P\sim Q\in S_l$ if $k=l$ and $P=Q$ or $|k-l|=1$ and the points $P$ and $Q$ are joined by a trajectory of $Z$ having the length less than $\frac34\delta.$

Let us prove that the relation $\sim$ is an equivalence. It is obviously reflexive and symmetric. 

Suppose $P_0\sim P_1\sim P_2$ and $P_0\in S_j,$ $P_1\in S_{j+1}$ and $P_2\in S_{l}$ where $l = j$ or $l = j+2.$ By~\ref{curve-b} and~\ref{Sj-IIIa} there exists $U_i$ such that $S_j\cup S_{j+1}\cup S_{l}\subset U_i$ and the intervals $\Phi(\{P_k\}\times (-\frac34\delta;\frac34\delta))$ lie in $U_i$ for $k = 0, 1, 2.$ Since $P_0\sim P_1\sim P_2$, the union of these three intervals is an interval lying in $U_i.$ Hence $\mathrm{pr}_i(P_0)=\mathrm{pr}_i(P_1)=\mathrm{pr}_i(P_2)$ which contradicts~\ref{Sj-IV} if $l = j+2$. If $l = j,$ it follows that $P_0 = P_2$ by~\ref{Sj-IIIc}. So the relation $\sim$ is trivially transitive. We also proved the following.

\begin{clai}\label{sim-one-two}
Any equivalence class of the relation $\sim$ consists of one or two elements.
\end{clai}

\begin{clai}
\label{sim-equivalence}
Let $P\in S_k,$ $Q\in S_l,$ $S_k\subset U_i$, $S_l\subset U_i.$ Then $P\sim Q$ if and only if $k=l$ and $P=Q$ or $|k-l|=1$ and $\mathrm{pr}_i (P) = \mathrm{pr}_i (Q).$
\end{clai}

\begin{proof}
Let $k = j$ and $l = j + 1.$

By~\ref{Sj-IIIa} the intervals $\Phi(\{P\}\times (-\frac34\delta;\frac34\delta))$ and $\Phi(\{Q\}\times (-\frac34\delta;\frac34\delta))$ lie in $U_i$.

If $P\sim Q$, these intervals intersect each other thus $\mathrm{pr}_i(P) = \mathrm{pr}_i(Q).$

If $\mathrm{pr}_i(P) = \mathrm{pr}_i(Q),$ then by condition ~\ref{Sj-V} we have $P\sim Q.$
\end{proof}

Let $\pi:  \bigsqcup\limits_j S_j \to X$ be the quotient map.

\begin{clai}
\label{Sj-is-open}
The map $\pi: S_j\to X$ is injective and open for any $j$.
\end{clai}

\begin{proof}
The injectivity follows from Claim~\ref{sim-one-two}.

Let $W$ be an open subset of $S_j.$ By~\ref{curve-b} there exists $U_i$ such that $S_{j-1}\cup S_{j}\cup S_{j+1}\subset U_i.$ By Claim~\ref{sim-one-two} and Claim~\ref{sim-equivalence}

$$\pi^{-1}(\pi (W)) = \mathrm{pr}_i^{-1}(\mathrm{pr}_i(W))\cap S_{j-1} \sqcup W \sqcup \mathrm{pr}_i^{-1}(\mathrm{pr}_i(W))\cap S_{j+1}.$$

Since $\mathrm{pr}_i (W)$ is open in $D_i$ by condition~\ref{Sj-IIIc} in Claim~\ref{Sj-definition}, the sets $\mathrm{pr}_i^{-1}(\mathrm{pr}_i(W))\cap S_{j\pm1}$ are open in $S_{j\pm1}.$ Therefore $\pi^{-1}(\pi(W))$ is open in $\bigsqcup\limits_j S_j$ and $\pi:S_j\to X$ is an open embedding.
\end{proof}

\begin{clai}
\label{Sj-cup-Sj+1}
Let $S_j\cup S_{j+1}\subset U_i.$ Then
the composition of relations $\pi(S_j \sqcup S_{j+1}) \stackrel{\pi^{-1}}\longrightarrow S_j \sqcup S_{j+1} \stackrel{\mathrm{pr}_i}\longrightarrow D_i$ is an injective open map.
\end{clai}
\begin{proof}
From Claims~\ref{sim-equivalence} and~\ref{Sj-is-open} it follows that this composition is a well defined injective map. The map $\pi\big|_{S_j\sqcup S_{j+1}}$ is continuous by the definition and open by Claim~\ref{Sj-is-open}. By~\ref{Sj-IIIc} the map $S_j\sqcup S_{j+1} \stackrel{\mathrm{pr}_i}\longrightarrow D_i$ is continuous and open. The claim follows.
\end{proof}

\begin{clai}
The space $X$ is locally Euclidean, Hausdorff and second-countable.
\end{clai}
\begin{proof}
The space $X$ is locally Euclidean, since it is the union of spaces $\pi(S_j)$ which are homeomorphic to an open subset of the plane by Claim~\ref{Sj-is-open}.

The space $X$ is second-countable, since it has a finite atlas $\{\pi(S_j)\}_j.$

Any two points of $X$ lie either in $\pi(S_j)$ for some $j$, or in $\pi(S_{j})\cup \pi(S_{j+1})$ for some $j$, or in $\pi(S_k)\sqcup\pi( S_l)$ for some $k,l.$ All these subsets are open in $X$ by Claim~\ref{Sj-is-open}. It is clear that the subsets $\pi(S_j)$ and $\pi(S_k)\sqcup \pi(S_l)$ are Hausdorff.
The space $\pi(S_{j})\cup \pi(S_{j+1})$ is Hausdorff because it is homeomorphic to an open subset of a disk by Claim~\ref{Sj-cup-Sj+1}. Therefore $X$ is Hausdorff.
\end{proof}

In other words, $X$ is a topological 2-dimensional manifold.

A smooth structure on $X$ is induced by the smooth structures on the surfaces $S_j$ by condition~\ref{atlas3} in Claim~\ref{atlas}.

Let $\{f_j\}_j$ be a partition of unity subordinate to the open cover $X = \bigcup\limits_j \pi(S_j).$ We define a map $\chi:X\to M$ as follows.
If $|\pi^{-1}(P)|=1$, we set $\chi(P) = P.$
Let $S_j\ni P\sim Q\in S_{j+1}.$ There exists the unique $h$ such that $|h|<\frac34\delta$ and $Q = \Phi(P,h).$ We set $\chi(\pi(P))=\chi(\pi(Q))=\Phi(P, f_{j+1}h).$

\begin{clai}\label{chi-immersion}
The map $\chi:X \to M$ is an immersion transverse to the vector field $Z$.
\end{clai}

\begin{proof}
Let $S_{j-1}\cup S_j\cup S_{j+1}\subset U_i.$ Let $(x_i, y_i, z_i)$ be the coordinates in $U_i$ constructed in the proof of Claim~\ref{atlas}.

Let $k = j-1,$ $j$ or $j+1.$ Since the surface $S_k$ and the disk $D_i$ are transverse to the vector field $Z$ and the map $\mathrm{pr}_i: S_k \to D_i$ is injective, the surface $S_k$ is a graph of a smooth function $g_{k} (x_i, y_i)$, which is defined on $\mathrm{pr}_i (S_k)$.

By construction, the surface $\chi(\pi(S_j))$ is a graph of the smooth function $\sum\limits_{j-1\le k\le j+1} f_k g_k$ defined on $\mathrm{pr}_i (S_j).$ Therefore, $\chi:\pi(S_j) \to U_i$ is an embedding transverse to the vector field $Z$. Since $X = \bigcup\limits_j \pi(S_j),$ $\chi:X \to M$ is an immersion transverse to $Z$.
\end{proof}

Recall that for any point $p\in \gamma_j$ there exists a real number $h_j(p)\in (-\delta/4;\delta/4)$ such that $\Phi(p, h_j(p))\in S_j$. This number is unique by~\ref{Sj-IIIa} because any surface $S_j$ lies inside some $U_i.$ If $p\in \gamma_j\cap\gamma_{j+1}$, the points $\Phi(p, h_j(p))\in S_j$ and $\Phi(p, h_{j+1}(p))\in S_{j+1}$ coincide with each other in $X$ by the definition of $X$. Therefore, the set of correspondences $\{p\mapsto \Phi(p, h_j(p))\}_j$, determine a well defined map $\sigma: C\to X.$ It is continuous, since $h_j(p)$ is continuous.

\begin{clai}\label{last-claim}
For any point $p\in C$ the interval $\Phi(\{p\}\times (-\delta;\delta))$ contains the point $\left(\chi\circ\sigma\right) (p).$
\end{clai}
\begin{proof}
The point $p$ lies in some $\gamma_j.$ Then there exists a point $P\in S_{j}$ which is joined with the point $p$ by a trajectory of the vector field $Z$ having the length less than $\delta/4.$

By condition~\ref{curve-b} there exists a number $i$ such that the chart $U_i$ contains the surfaces $S_{j-1}, S_j, S_{j+1}.$ Then $|z_i(P)-z_i(p)|<\delta/4.$

\smallskip
{\noindent\bf Case 1.} There exists a point $P_1\in S_{j+1}$ such that $\mathrm{pr}_i(p) = \mathrm{pr}_i(P_1).$

By condition~\ref{Sj-V} we have $|z_i(P_1)-z_i(P)|<\frac34\delta.$ Thus $|z_i(P_1) - z_i(p)|< \delta.$

Since the point $(\chi\circ\sigma)(p)$ lies on the segment of a trajectory of the vector field $Z$ with endpoints $P$ and $P_1$ and this segment is contained in the chart $U_i$, $|z_i ((\chi\circ\sigma)(p)) - z_i(p)| < \delta.$

\smallskip
{\noindent\bf Case 2.} There exists a point $P_1\in S_{j-1}$ such that $\mathrm{pr}_i(p) = \mathrm{pr}_i(P_1).$

This case is similar to the previous one.

\smallskip
{\noindent\bf Case 3.} $\mathrm{pr}_i (p) \notin \mathrm{pr}_i(S_{j-1}\cup S_{j+1}).$

In this case $(\chi\circ\sigma) (p)=P$ and $|z_i ((\chi\circ\sigma)(p)) - z_i(p)| < \delta/4.$

\end{proof}

\begin{proof}[Proof of Lemma~\ref{surface-of-trajectories}]
Recall that by condition~(\ref{delta-definition}) and the definition of $\delta_1$ for any point $p\in C$ the trajectory $\Phi(\{p\}\times[-\varepsilon-\delta;\varepsilon+\delta])$ is well defined and embedded.
By Claim~\ref{last-claim} the point $(\chi\circ\sigma)(p)$ lies on the interval $\Phi(\{p\}\times(-\delta;\delta))$ for any point $p\in C$. So the map $\Psi:(q,t)\mapsto \Phi(\chi(q),t)$ from $\sigma(C)\times [-\varepsilon;\varepsilon]$ to $M\setminus\partial M$ is well defined and injective. 

Thus the map $\Psi$ is well defined on $X_1\times [-\varepsilon;\varepsilon]$ where $X_1$ is some neighborhood of the curve $\sigma(C).$
% such that the map $\Psi:X_1\times [-\varepsilon;\varepsilon]\to M$ is well defined where $\Psi(p, t) = \Phi(\chi(p),t)$.

Since the map $\chi$ is an immersion transverse to the vector field $Z$ (Claim~\ref{chi-immersion}), the map $\Psi$ is a local diffeomorphism. Since on the compact subset $\sigma(C)\times[-\varepsilon;\varepsilon]$ the map $\Psi$ is injective and the map $\Psi$ is a local homeomorphism, there exists a neighborhood of the subset $\sigma(C)\times[-\varepsilon;\varepsilon]$ such that the restriction of the map $\Psi$ on this neighborhood is an open embedding. Therefore, there exists an open subset $X_2\subset X_1$ which contains the curve $\sigma(C)$ such that $\Psi: X_2\times (-\varepsilon;\varepsilon)\to M$ is an open embedding.

Hence $S = \chi(X_2)$ is a sought-for surface and $U = \Phi(S\times(-\varepsilon;\varepsilon))$ is a regular neighborhood.
\end{proof}

\begin{lemm}
\label{transverse-collar2}
Let $l$ be a smooth field of line elements on the manifold $M$ transverse to the contact structure and $L$ be a compact Legendrian curve in $M\setminus\partial M$.Then there exists $\varepsilon>0$ such that the map $\exp_{l}:\overline B_{\varepsilon l}(L)\to M$ is well defined and injective.
\end{lemm}

\begin{lemm}
\label{surface-of-trajectories2}
Let $l$ be a smooth everywhere nonzero field of line elements on the manifold $M$, $\varepsilon$ be a positive real number and $C$ be a compact curve in $M$ such that the map $\exp_{l}: \overline B_{\varepsilon l}(C)\to M\setminus\partial M$ is well defined and injective. Then there exists a neighborhood $U$ of $C$ and a smooth surface $S\subset U$ without boundary such that $\exp_{l}: B_{\varepsilon l}(S)\to U$ is a diffeomorphism.
\end{lemm}

The proof of these two lemmas is similar to the proof of Lemmas~\ref{transverse-collar} and~\ref{surface-of-trajectories} because all arguments were local except the fact in the end of the proof of Lemma~\ref{surface-of-trajectories} that if a local homeomorphism is injective on a compact subset then it is a homeomorphism of some neighborhood of this subset. We only indicate which changes should be made.

All sets of the form $A\times[-\delta;\delta]$ and $A\times(-\delta;\delta)$ we substitute by $\overline B_{\delta l}(A)$ and $B_{\delta l}(A)$ respectively.

In the proof of Claim~\ref{atlas} we should choose the neighborhoods $V_p$ to be simply connected. Then in any such neighborhood the field $l$ has the form $\{Z, -Z\}$ where $Z$ is a vector field and we can use $Z$ to construct the coordinates $(x,y,z).$

A subset of the chart $U_i$ of the form $\Phi(A\times\{z'_i\})$ we substitute by the set $\{(x_i,y_i,z_i)\in U_i:\ (x_i,y_i,z_i-z'_i)\in A\}.$

The rest remains unchanged.

\begin{lemm}\label{Reeb-line-element-existence}
On any contact manifold there exists a contact field of line elements transverse to the contact structure.
\end{lemm}

\begin{proof}
Let $(M,\xi)$ denote the contact manifold, $w_1(\xi)$ denote the first Stiefel--Whitney class of the contact structure $\xi.$

If $w_1(\xi)=0,$ the contact structure is coorientable thus there exists a Reeb vector field $v$ (see \cite[Sections 1.1 and 2.3]{Gei}), and $\{v, -v\}$ is a sought-for line element.

Suppose $w_1(\xi)\neq0$. Let $p:\widetilde M\to M$ denote the 2-fold covering corresponding to the subgroup $\ker(w_1(\xi):\pi_1(M)\to\mathbb Z_2)$ and $h:\widetilde M \to \widetilde M$ be the diffeomorphism which acts as a transposition on any fiber of the covering.

Then the contact structure $p_*^{-1}(\xi)$ is coorientable. Thus it is the kernel of a differential 1-form $\alpha$ on $\widetilde M.$ We note that $h$ changes the coorientation of the contact structure. Therefore $\ker(\alpha-h^*\alpha) = p_*^{-1}(\xi).$ We note that $h^*(\alpha-h^*(\alpha)) = -(\alpha-h^*(\alpha)).$ Thus if $v$ is the Reeb vector field of the contact form $\alpha-h^*(\alpha),$ $p_* v$ is a sought-for contact field of line elements.
\end{proof}

\subsection{Case $r=+\infty$}

In this subsection we prove for closed Legendrian curves the existence of a regular neighborhood with $r=+\infty$ (Proposition~\ref{r-invariant-regular-neighbourhood-existence}).

\begin{defi}\label{regular-neighbourhood-core}
Let $U$ be a regular neighborhood of the closed Legendrian curve $L$. Let $l$ denote the field of line elements in $U.$ A compact set $K\subset U$ is called a {\it core of the regular neighborhood} $U$ if $L\subset K\setminus\partial K$ and there exists a smooth compact surface $S_0\subset U$ such that
\begin{enumerate}
\item\label{core-P1} $S_0$ is transverse to $l$ and $\mathrm{pr}\big|_{S_0}$ is injective.
\item $K = \exp_l (\overline B_{r_0 l}(S_0))$ for some $r_0>0$.
\item\label{core-P3} $S_0$ is homeomorphic to an annulus if the contact structure is coorientable on $L$, and $S_0$ is homeomorphic to a M\"obius band otherwise.
\item $\partial S_0$ is a smooth link transverse to the contact structure.
\end{enumerate}

The triple $(K, S_0, r_0)$ will also be called the core of the regular neighborhood.
\end{defi}

We note that the interior of a core of a regular neighborhood is itself a regular neighborhood.

\begin{lemm}\label{transverse-push-off}
Any regular neighborhood of a closed Legendrian curve has a core.
\end{lemm}

\begin{proof} Let $(U,S,l,r)$ denote the regular neighborhood, and $L$ denote the Legendrian curve. We consider two cases.

\noindent{\bf Case 1.} The contact structure is coorientable on the curve $L.$

By Remark~\ref{regular-neighbourhood-contact-form} we can take a smaller surface $S'\subset S$ such that $(U' = \exp_l(B_{rl}(S')), S', l, r)$ is a regular neighborhood of $L$ and the contact structure on $U'$ is the kernel of a 1-form $dz+\mathrm{pr}^*\beta$, where $z:U'\to\mathbb R$ is a function, $z\big|_{S'}\equiv0$, $\partial z/\partial l = \pm1$ and $\beta$ is a differential 1-form on $S'.$ We orient the surface $S'$ by the form $d\beta.$ It is sufficient to construct a core for the regular neighborhood $U'$ because $U'\subset U.$ So we can assume $S' = S$ and $U' = U.$

The curve $\mathrm{pr}(L)$ is Lavrentiev by Lemma~\ref{Lavrentiev-projection-lemma}. Let $f:\mathbb S^1\times[-1;1]\to S$ be a closed regular tubular neighborhood of the curve $\mathrm{pr}(L)$ (see Definition~\ref{closed-regular-tubular-neighbourhood-definition}).

We orient the boundary of $f(\mathbb S^1\times[0;1])$ in concert with the orientation of the interior. We orient the circle $\mathbb S^1$ in such a way that the map $f(\bullet, 0)$ preserves the orientation. We choose a positive coordinate $s\in[0;2\pi]$ on the circle $\mathbb S^1.$
We set $$z(s, t) = z_0 - \int\limits_{f([0;s]\times\{t\})}\beta + \frac{s}{2\pi}\int\limits_{f([0;2\pi]\times\{t\})}\beta,$$
where $z_0$ is the $z$-coordinate of the point on the curve $L$ whose projection is $f(0,0).$

Since $L$ is Legendrian, the integral of $\beta$ on any subarc of $L$ is zero (Proposition~\ref{leg-integral}). Therefore $z(s, 0)$ is the $z$-coordinate of the point on the curve $L$ whose projection is $f(s,0).$

Since $f$ is bi-Lipschitz, $z(s,t)$ is a continuous function by Corollary~\ref{integral-continuity-corollary}. So there exists a positive real number $t_0$ such that $|z(s, t)| < r$ for any $s\in[0;2\pi]$ and $t\in[0;t_0].$

Let $L_+ = \{p\in U:\ \exists s\in[0;2\pi]\ \mathrm{pr}(p) = f(s,t_0),\ z(p) = z(s,t_0)\}.$ The set $L_+$ is a closed curve because $z(0,t_0) = z(2\pi, t_0) = z_0$ by definition. The curve $f(\mathbb S^1\times\{ t_0\}) = \mathrm{pr}(L_+)$ is smooth because $f$ is a regular neighborhood of a Legendrian curve. Therefore $L_+$ is smooth. We take a derivative $\partial/\partial s$ of each part of the equation in the definition of $z(s,t)$ and obtain the following:

\begin{equation}\label{L+-is-transverse}
\frac{\partial z(s,t_0)}{\partial s} + \beta\left(\frac{\partial f(s, t_0)}{\partial s}\right) = \frac1{2\pi}\int\limits_{f(\mathbb S^1\times\{t_0\})}\beta.
\end{equation}

By Stokes theorem (see \cite[Chapter IX, Theorems 5A and 7A]{Whi})

\begin{equation}\label{Stokes-for-L+}
0<\int\limits_{f(\mathbb S^1\times [0;t_0])}d\beta = \int\limits_{\mathrm{pr}(L)}\beta - \int\limits_{f(\mathbb S^1\times\{t_0\})}  \beta = - \int\limits_{f(\mathbb S^1\times\{t_0\})}\beta.
\end{equation}

From equations~(\ref{L+-is-transverse}) and~(\ref{Stokes-for-L+}) it follows that $L_+$ is negatively transverse, i.e. the contact form $dz+\mathrm{pr}^*\beta$ is negative on the velocity vector of $L_+.$

Similarly we can construct a positively transverse curve $L_-$ with $\mathrm{pr}(L_-) = f(\mathbb S^1\times\{t_1\})$ for some $t_1<0$. A sought-for surface $S_0$ can be constructed as a graph of a smooth function $\widetilde z(p)$ on the annulus cobounded by $\mathrm{pr}(L_+)$ and $\mathrm{pr}(L_-)$. We must ensure that $\partial S_0 = L_+\cup L_-$ and $L\subset \exp_l(\overline B_{r_0l}(S_0))\subset U$ for some $r_0$. So we have to ensure $\widetilde z(f(s, t_0)) = z(s, t_0),$ $\widetilde z(f(s, t_1)) = z(s, t_1)$ and $|\widetilde z(f(s, 0)) - z(s, 0)|<r_0$ where $r_0 = r - |\max \widetilde z|.$ It is easy to construct such smooth function $\widetilde z.$

\smallskip

\noindent{\bf Case 2.} The contact structure is not coorientable on $L.$

In this case we can pass to the double covering $\pi:\widetilde S \to S$ corresponding to the kernel of the homomorphism $w_1:\pi_1(S)\to\mathbb Z_2$ where $w_1$ is the restriction of the first Stiefel-Whitney class of the contact structure to $S$. By Lemma~\ref{Lavrentiev-covering} the inverse image $\pi^{-1}(\mathrm{pr}(L))$ under the covering of the Lavrentiev curve $\mathrm{pr}(L)$ is also a Lavrentiev curve. Then there exists a closed regular tubular neighborhood $f:\mathbb S^1\times [-1;1] \to \widetilde S$ of $\pi^{-1}(\mathrm{pr}(L)).$ Similarly to Case 1 we can use the map $\pi\circ f\big|_{\mathbb S^1\times[0;1]}$ to construct a sought-for M\"obius band.
\end{proof}

\begin{lemm}\label{core-orientations}
Let $(K, S_0, r_0)$ be a core of the regular neighborhood $U$ with a contact field of line elements $l$ of the Legendrian curve $L$. If $p\in\partial S_0$, $v_1\in l$ and $v_1\in T_p U$, a vector $v_2\in T_p \partial S_0$ and the vector $v_1$ direct to the same side of the contact plane, $v_3\in T_p S_0$ is a vector directed inward the surface, then the basis $v_1,v_2,v_3$ is positive.
\end{lemm}

If the contact structure is coorientable, the field $l$ is a pair of contact vector fields in $U$. Suppose some coorientation is fixed. So one vector field in $l$ is positive and the other is negative. In this case this lemma says that the orientation of $S_0$ induced by the positive contact vector field agrees with the orientation of $\partial S_0$ induced by the coorientation of the contact structure.

\begin{proof}
The condition that we have to check is local. So we can assume that the contact structure is coorientable by passing to the double covering if it is needed. Similarly to Case 1 in Lemma~\ref{transverse-push-off} by Remark~\ref{regular-neighbourhood-contact-form} we can assume that the contact structure in $K$ is the kernel of the differential 1-form $dz+\mathrm{pr}^*\beta$ where $\beta$ is a differential 1-form on $S_0$ and $\mathrm{pr}$ is the projection to $S_0$ along the trajectories of the contact field.

Let $\partial_+ S_0$ be some component of $\partial S_0$. Let $A$ be an annulus in $S_0$ cobounded by $\partial_+ S_0$ and $\mathrm{pr}(L)$. By Stokes theorem for $A$
$$
\int\limits_A d\beta = \int\limits_{\partial_+ S_0}\beta + \int\limits_{\mathrm{pr}(L)}\beta = \int\limits_{\partial_+ S_0}\beta
$$
because $L$ is Legendrian. Since $\partial_+ S_0$ is transverse to the contact structure, the expression under the integral on the right side has a sign independent of a point on the curve. So the orientation of $\partial_+ S_0$ as a part of the boundary of $S_0$ agrees with the coorientation of the contact structure.
\end{proof}

\begin{lemm}\label{core-curve-lemma}
Let $(K, S_0, r_0)$ be a core of some regular neighborhood with contact field $l$ of the closed Legendrian curve $L$. Let $\mathrm{pr}$ be the projection from $K$ to $S_0$ along the trajectories of $l.$ Then the curve $\mathrm{pr}(L)$ is isotopic in $S_0$ to the core curve of $S_0.$
\end{lemm}

\begin{proof}
First we consider the case when $S_0$ is an annulus. The curve $\mathrm{pr}(L)$ is embedded. If it is not isotopic to the core curve, it bounds a disk which contradicts the fact that $L$ is closed.

Now suppose that $S_0$ is a M\"obius band. Since the curve $\mathrm{pr}(L)$ is embedded, it is isotopic to the core curve or to twice the core curve or it bounds a disk. The latter case is impossible since $L$ is closed. If it is isotopic to twice the core curve, the contact structure is coorientable on $\mathrm{pr}(L).$ Since $\mathrm{pr}(L)$ and $L$ are isotopic, the contact structure is coorientable on $L$ which contradicts property~(\ref{core-P3}) in Definition~\ref{regular-neighbourhood-core}.
\end{proof}

\begin{prop}\label{r-invariant-regular-neighbourhood-existence}
Let $(K, S_0, r_0)$ be a core of a regular neighborhood with contact field $l_0$ of the Legendrian curve $L$. Then the curve $L$ has a regular neighborhood $(U, S, l, +\infty)$ such that $U\supset K,$ $S\supset S_0$ and the field $l$ coincides with the field $l_0$ on $K$.
\end{prop}

\begin{proof}

We consider two cases.

\smallskip
\noindent{\bf Case 1.} $S_0$ is an annulus.

There exists a neighborhood $V$ of $K$ with positive coordinates $(x,y,z)$ such that
$\partial z/\partial l_0 = \pm1$ and $S_0 = \{(x,y,0):\ x\in\mathbb S^1,\ y\in[-1;1]\}.$ We can assume that the closure of $V$ is compact and is contained in the interior of the manifold.

Then the contact structure in $V$ is the kernel of the differential 1-form $dz + a(x,y)dx + b(x,y)dy$ where $\dfrac{\partial b}{\partial x} - \dfrac{\partial a}{\partial y} > 0.$

By condition~(4) in the definition of a core $a(x,\pm1)\neq0$ for any $x\in\mathbb S^1.$ Moreover, by Lemma~\ref{core-orientations} $a(x,1)<0$ and $a(x,-1)>0.$

Let $r_1, \delta$ be such real numbers that $r_1>r_0$, $\delta > 0$, $\{(x,y,z):\ |y|\le 1+\delta,\ |z|\le r_1 \}\subset V$ and $a(x,1+\delta)<0,$ $a(x,-1-\delta)>0$ for any $x\in\mathbb S^1.$ Let $K_1 = \{(x,y,z):\ |y|\le 1+\delta,\ |z|\le r_1\}.$ It is clear that $K$ is contained in the interior of $K_1.$

Let $H$ be a smooth function on the manifold such that (see the left side of Figure~\ref{function-H-and-contact-vector-field})
\begin{enumerate}
\item $H\equiv0$ outside $V$.
\item $H\big|_{K_1}$ is independent of $x.$
\item $H\big|_{\partial K_1} = z.$
\item $H(y, z) = r_1$ if $|y|\le 1+\delta/2$ and $r_0\le z\le r_1.$
\item $H(y, z) = -r_1$ if $|y|\le 1+\delta/2$ and $r_0\le -z\le r_1.$
\end{enumerate}

\begin{figure}[ht]
\center{\includegraphics{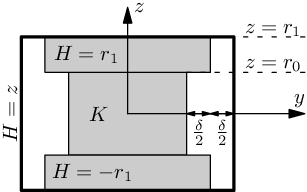} \quad \quad \includegraphics{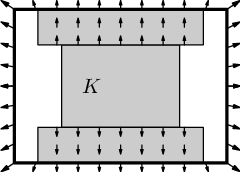}}
\caption{The contact Hamiltonian $H$ and the contact vector field $X_H$}
\label{function-H-and-contact-vector-field}
\end{figure}

Let $X_H$ be the contact vector field corresponding to the Hamiltonian $H$ (see~\cite[\S 2.3]{Gei}).

The direct calculation shows that in $K_1$

$$
X_H = \frac1{\dfrac{\partial b}{\partial x} - \dfrac{\partial a}{\partial y}}\left(\dfrac{\partial H}{\partial z}b - \dfrac{\partial H}{\partial y},  -\dfrac{\partial H}{\partial z}a, \dfrac{\partial H}{\partial y}a\right)   + (0,0,H).
$$

Let us check that $X_H$ is transverse to $\partial K_1$ (see the right side of Figure~\ref{function-H-and-contact-vector-field}, all arrows are made proportionally smaller).

If $z = r_1,$ $dz (X_H) = r_1 > 0.$

If $z = -r_1,$ $dz (X_H) = -r_1<0.$

If $y = \pm(1+\delta),$ $dy(X_H)=-\dfrac{a}{\dfrac{\partial b}{\partial x} - \dfrac{\partial a}{\partial y}}$ which is positive if $y=1+\delta$, and which is negative if $y=-1-\delta.$

Since $X_H$ is zero outside $V$ and the closure of $V$ is compact and is contained in the interior of the manifold, any trajectory of $X_H$ is defined for all moments of time. Let $\Phi_H:\overline V\times\mathbb R \to \overline V$ be the flow of of $X_H.$ Since the vector $X_H$ at any point of $\partial K_1$ is directed outside $K_1$, all trajectories of $X_H$, which intersect the boundary $\partial K_1$, are distinct and non-closed. 

Let
$$
U = \{\Phi_H((x,y,\pm r_1), t):\ x\in\mathbb S^1,\ |y|<1+\delta/2,\ t>0\}\sqcup \{(x,y,z):\ x\in\mathbb S^1,\ |y|<1+\delta/2, \ z\in(-r_1;r_1)\}.
$$

The first part of $U$ is the union of all trajectories of the vector field $X_H$ emanating from the points of the form $(x, y, \pm r_1),$ where $x\in\mathbb S^1$ and $|y|<1+\delta/2$. It is obvious that $U$ is open and contains $K$.

Let $l$ be the field of line elements in $U$ such that 
$$l = \{\partial/\partial z, -\partial/\partial z\}\text{ in } U\cap K_1\text{ and } l = \{X_H/r_1, -X_H/r_1\}\text{ in } U\setminus K_1.$$ It is obvious that $l$ is contact. On every trajectory of $l$ there are points at which the field is transverse to the contact structure. Thus, since $l$ is contact, $l$ is transverse to the contact structure everywhere.

We set $S = \{(x,y,0):\ x\in\mathbb S^1,\ |y|<1+\delta/2 \}.$ By construction, $(U, S, l, +\infty)$ is a regular neighborhood of $L$, $S\supset S_0$ and the field $l$ coincides with the field $l_0 = \{\partial/\partial z,-\partial/\partial z\}$ on $K$. 

\smallskip
\noindent{\bf Case 2.} $S_0$ is a M\"obius band.

We pass to the double covering corresponding to the first Stiefel--Whitney class of the contact structure. The inverse images under the covering of the Legendrian curve, the regular neighborhood and the core are respectively a Legendrian curve, its regular neighborhood and a core of the obtained regular neighborhood. Then it is sufficient to make construction in Case 1 invariant under the deck transformation $h$.

We can choose $V$ invariant by setting $V = W\cap h(W)$ where $W$ is a neighborhood of $K$. We choose coordinates $(x,y,z)$ in $V$ such that $h:(x,y,z)\mapsto (-x, -y, -z)$, where $x\in\mathbb S^1 = \partial\mathbb D\subset\mathbb C, y\in[-1;1], z\in[-r_0;r_0].$ Thus $K_1$ is invariant. We choose a contact form $\alpha = dz + a(x,y)dx + b(x,y)dy$ such that $h^*\alpha = -\alpha.$
We can additionally make $H = -H\circ h$ thus $h_*X_H = X_H$.
\end{proof}

\section{Smoothing Legendrian links}
\label{Legendrian-smoothing-section}

\begin{defi}\label{legendrian-link}
By a {\it link} we call a 1-dimensional compact topogical submanifold without boundary. The link is called {\it Legendrian (respectively, Lavrentiev)} if every its component is a Legendrian (respectively, Lavrentiev) curve. An isotopy of links is called {\it Legendrian} if the restriction of this isotopy on any component of the link is a Legendrian isotopy of Legendrian curves.
\end{defi}

\begin{prop}\label{smoothing-legendrian-link}
Any Legendrian link is Legendrian isotopic to a smooth link.
\end{prop}

Since any compact Legendrian curve has a regular neighborhood (Proposition~\ref{regular-neighbourhood-existence}), Proposition~\ref{smoothing-legendrian-link} follows from Proposition~\ref{smoothing-Legendrian-arc}.

\begin{prop}\label{smoothing-Legendrian-arc}
Let $L$ be a compact Legendrian curve, $(U, S, l, r)$ be its regular neighborhood, $\gamma_0\subset L$ be its compact subarc, $\partial \gamma_0\cap \partial L = \varnothing,$ $V\subset U$ is open and $\gamma_0\subset V.$ Then there exist a number $C$ and an isotopy $F:L\times[0;1]\to U$ such that
\begin{enumerate}
\item $F(\bullet, 0) = \mathrm{id}_{L}.$
\item $F(\bullet, t)$ is $C$-bi-Lipschitz for any $t\in[0;1].$
\item $F$ is Lipschitz.
\item\label{leg-smoothing-P4} $(U,S,l,r)$ is a regular neighborhood of $F(L\times\{t\})$ for any $t\in[0;1].$
\item $F$ is Legendrian.
\item $F(p, t) = p$ for any $p\in L\setminus V$ and any $t\in[0;1].$
\item $F((L\cap V)\times[0;1])\subset V.$
\item\label{leg-smoothing-P8} The arc $F(\gamma_0\times\{1\})$ is contained in the interior of a smooth subarc of $F(L\times\{1\}).$
\item If the subarc $\gamma$ of $L$ is smooth then the arc $F(\gamma\times\{1\})$ is smooth.
\end{enumerate}
\end{prop}

\begin{proof}
There exist a smooth surface $S_0\subset U$ without boundary and a real number $r_0$ such that
\begin{enumerate}
\item $S_0$ is transverse to $l$.
\item $\mathrm{pr}:S_0\to S$ is injective.
\item $\gamma_0\subset U_0\subset V$ where $U_0=\exp_l (B_{r_0l}(S_0)).$
\item $S_0$ is orientable.
\end{enumerate}

Let $\mathrm{pr}_0:U_0 \to S_0$ be the projection along the trajectories of $l$. Since $S_0$ is orientable, by Remark~\ref{regular-neighbourhood-contact-form} the contact structure on $U_0$ is the kernel of the differential 1-form $dz+\mathrm{pr}_0^*\beta$ such that $\partial z/\partial l = \pm1,$ $z\big|_{S_0} \equiv 0$ and $\beta\in\Omega^1(S_0).$ Let $L'$ be a compact subarc of $L$ such that $\gamma_0\subset L' \subset U_0$ and $\partial\gamma_0\cap\partial L'=\varnothing.$ Then there exists an open subset $V_0\subset S_0$ containing $\mathrm{pr}_0 (\gamma_0)$ such that $\mathrm{pr}(\overline V_0)\cap\mathrm{pr} (L) \subset \mathrm{pr}(L'\setminus\partial L').$  We apply Proposition~\ref{Lavrentiev-smoothing} to the subarc $\mathrm{pr}_0 (\gamma_0)$ of the curve $\mathrm{pr}_0(L')$ lying on the surface $S_0$, to the open set $V_0$ and the real number $\varepsilon = r_0 - \max\{|z(p)|:\ p\in L'\}.$

Among the assumptions of Proposition~\ref{Lavrentiev-smoothing} we need to check only that the curve $\mathrm{pr}_0 (L')$ is Lavrentiev. Since $L'$ is Legendrian, its regular projection $\mathrm{pr}_0 (L')$ is locally Lavrentiev by Lemma~\ref{Lavrentiev-projection-lemma}. Since $\mathrm{pr}_0 (L')$ is compact and embedded, it is Lavrentiev by Lemma~\ref{compact-Lavrentiev}.

By Proposition~\ref{Lavrentiev-smoothing} we have an isotopy $F_0:\mathrm{pr}_0 (L')\times[0;1]\to S_0.$ Let $\partial L' = \{P, Q\}.$ We note that $\mathrm{pr}_0(P)\notin V_0$ by the definition of $V_0$. Thus by condition~(\ref{LScond6}) in Proposition~\ref{Lavrentiev-smoothing} the point $\mathrm{pr}_0(P)$ stays fixed under the isotopy $F_0.$ Let us prove that there exists the unique map $F':L'\times[0;1]\to U_0$ such that $F'(P,t) = P$, $F'(L'\times\{t\})$ is Legendrian and $\mathrm{pr}_0 \circ F'(\bullet, t) = F_0(\bullet, t)\circ\mathrm{pr}_0$ for any $t\in[0;1]$.

The last condition on $F'$ means that $F'$ is a lift of the isotopy $F_0$ under the projection $\mathrm{pr}_0.$ Since we want the curves $F'(L'\times\{t\})$ to be Legendrian, the $z$-coordinate is uniquely determined on these curves by Lemma~\ref{z-uniqueness} and the fact that $F'(P,t) = P.$ By condition~(\ref{LScond4}) in Proposition~\ref{Lavrentiev-smoothing} and by the definition of $\varepsilon$ the restriction of $z$-coordinate on $F'(L'\times\{t\})$ is bounded above and below by $\pm r_0$ thus the curves $F'(L'\times\{t\})$ are well defined for any $t$ and lie in $U_0.$

Then we note that $F'(Q,t) = Q$ for any $t\in[0;1].$ Indeed, this follows from condition~(\ref{LScond5}) in Proposition~\ref{Lavrentiev-smoothing}. That means that we can extend the map $F'$ to the map $F:L\times[0;1]\to V$ such that $F(p, t) = p$ for any $p\in L\setminus L'$ and $t\in[0;1].$ Then we check one by one that all conditions are satisfied by $F$% and that $F$ is continuous
.

Condition (1) is trivial.

For conditions (6) and (7) we remind that $U_0\subset V$ and all our constructions were done in $U_0.$

Condition (8). By~(\ref{LScond8}) in  Proposition~\ref{Lavrentiev-smoothing} the arc $F_0(\mathrm{pr}_0(\gamma_0)\times\{1\})$ is contained in a smooth subarc of $F_0(\mathrm{pr}_0(L')\times\{1\})$. Since the contact form is smooth the Legendrian lifting of a smooth arc is also smooth.

To check the remaining conditions we introduce a suitable open cover of the curve $L$. There exists an open set $W'\subset S$ such that $W'\cap \mathrm{pr}(L) = \mathrm{pr}(L'\setminus\partial L').$ We set $W_0 = \mathrm{pr}(V_0)\cup W'.$ Since $V_0\subset S_0$ is open, $\mathrm{pr}(V_0)\subset S$ is also open and $W_0$ is open. By the definition of $V_0$, $W_0\cap \mathrm{pr} (L) = \mathrm{pr}(L'\setminus\partial L').$ Since the isotopy $F$ is supported in $\mathrm{pr}_0^{-1}(V_0)$ and by the definition of $V_0$ $\mathrm{pr}(\overline V_0)\cap\mathrm{pr} (L) \subset \mathrm{pr}(L'\setminus\partial L'),$ there exists an open subset $W_1\subset S$ such that $W_1\supset \mathrm{pr}((L \setminus L')\cup \partial L')$ and if $\mathrm{pr}(F(p, t)) \in W_1$ for some $t\in[0;1]$ then $F(p, t) = p$ for any $t\in[0;1].$ By construction $\mathrm{pr}(L)\subset W_0\cup W_1.$ What we need later is that in $\mathrm{pr}^{-1}(W_1)$ the isotopy $F$ is trivial, while $\mathrm{pr}^{-1}(W_0)\cap F(L\times \{t\}) = F'((L'\setminus\partial L')\times\{t\})$ for any $t\in[0;1]$. 

Condition (9). It is sufficient to consider subarcs which are contained either in $\mathrm{pr}^{-1}(W_0)$ or in $\mathrm{pr}^{-1}(W_1).$ The latter case is trivial. The former case follows from condition~(\ref{LScond9}) in Proposition~\ref{Lavrentiev-smoothing} and the fact that if the regular projection of a Legendrian curve is smooth and the contact form is smooth then the curve is smooth.

Condition (5). By condition~(\ref{LScond2}) in Proposition~\ref{Lavrentiev-smoothing} all maps $F_0(\bullet, t):\mathrm{pr}_0(L')\to S_0$ are $C$-bi-Lipschitz for some common $C$. All curves $F'(L'\times\{t\})$ are Legendrian. So by Proposition~\ref{bi-Lipschitz-isotopy-lifting} $F'$ is a continuous Legendrian isotopy. Since the definition of a Legendrian isotopy is local, $F\big|_{\mathrm{pr}^{-1}(W_1)\cap L\times[0;1]}$ is trivial, $F\big|_{\mathrm{pr}^{-1}(W_0)\cap L\times[0;1]} = F'\big|_{\mathrm{pr}^{-1}(W_0)\cap L\times[0;1]}$ and $L\subset \mathrm{pr}^{-1}(W_0)\cup \mathrm{pr}^{-1}(W_1)$, $F$ is also a Legendrian isotopy.

Condition (2). Similarly to the checking of condition (5) by Proposition~\ref{bi-Lipschitz-isotopy-lifting} the maps $F'(\bullet, t)$ are bi-Lipschitz with a common constant. By Lemma~\ref{bi-Lipschitz-family} the maps $F(\bullet, t)$ are bi-Lipschitz with a common constant.

Condition (3). By condition (3) in Proposition~\ref{Lavrentiev-smoothing} $F_0$ is Lipschitz. Then by Lemma~\ref{Lipschitz-isotopy-lifting} $F'$ is Lipschitz. Thus $F$ is locally Lipschitz. Since $L\times[0;1]$ is compact, $F$ is Lipschitz.

Condition (4). We only need to check that $\mathrm{pr}\big|_{F(L\times\{t\})}$ is injective for any $t\in[0;1].$ This is satisfied in $\mathrm{pr}^{-1}(W_1)$ because $U$ is a regular neighborhood of $L$. This is satisfied in $\mathrm{pr}^{-1}(W_0)$ because $\mathrm{pr}:S_0\to S$ is injective and $\mathrm{pr}_0: F'(L'\times\{t\})\to S_0$ is injective by construction.
\end{proof}

Let $L$ be a smooth Legendrian link. Let us prove that any $C^1$-continuous isotopy $F:L\times[0;1]\to (M,\xi)$ of smooth Legendrian links is a Legendrian isotopy in the sense of Definition~\ref{legendrian-link}. Indeed by Lemma~\ref{Lavrentiev-C1-isotopy} the curves $F(\gamma\times\{t\})$ are Lavrentiev with a common constant where $\gamma$ is any component of the link $L$. Since the curves $F(\gamma\times\{t\})$ are tangent to the contact structure, the integral of the contact form on any their subarc is zero, hence they are Legendrian and by Proposition~\ref{leg-C-isotopy} the isotopy is Legendrian.

\begin{prop}\label{smooth-Legendrian-isotopy}
Let two smooth Legendrian links be Legendrian isotopic as Legendrian Lavrentiev links. Then they are smoothly Legendrian isotopic.
\end{prop}

We say that two Legendrian curves {\it have the common core} $K$ if there exists their common regular neighborhood such that $K$ is its core (see Definition~\ref{regular-neighbourhood-core}).

\begin{lemm}\label{smooth-isotopy-lemma6}
Let $K$ be a core of the regular neighborhood $U$ of the closed Legendrian curve $L$. Let $S$ denote the surface of the core $K$ and $\mathrm{pr}:K\to S$ be the projection along the trajectories of the contact field. Suppose that a Legendrian curve $L'$ is isotopic to $L$, $L'\subset K\setminus \partial K$ and $\mathrm{pr}\big|_{L'}$ is injective. Then $U$ is a regular neighborhood of $L'$ and $K$ is a core of $U$ as a regular neighborhood of $L'$.
\end{lemm}
\begin{proof}
Let us show that $U$ is a regular neighborhood of $L'$. We need to check only that the projection of $L'$ along the trajectories of the contact field is injective. Since it is injective in $K\setminus\partial K$, by~(\ref{core-P1}) in the definition of the core it is injective in $U$.

Then we show that $K$ is the core of $U$ as a regular neighborhood of $L'.$ We need to check only that the contact structure is coorientable on $L'$ if and only if it is coorientable on $L.$ Since $L$ and $L'$ are isotopic, this is true.
\end{proof}

\begin{lemm}\label{smooth-isotopy-lemma5}
Let $K$ be a core of a regular neighborhood $U$ of the closed Legendrian curve $L$. Then there exists a smooth Legendrian curve $L'$ which has the common core $K$ with $L$.
\end{lemm}
\begin{proof}
We represent the curve $L$ as the union of two compact subarcs $\gamma_1$ and $\gamma_2$. We apply Proposition~\ref{smoothing-Legendrian-arc} for the regular neighborhood $K\setminus\partial K$ of the curve $L$ two times. First, we put $\gamma_0 = \gamma_1$, and then we put $\gamma_0 = \gamma_2.$ By~(\ref{leg-smoothing-P8}) of that proposition we get a smooth Legendrian curve $L'$ in $K\setminus\partial K$. By~(\ref{leg-smoothing-P4}) $K\setminus\partial K$ is a regular neighborhood of $L'$. By~Lemma~\ref{smooth-isotopy-lemma6} $U$ is a regular neighborhood of $L'$ and $K$ is the core of $U$ as a regular neighborhood of $L'.$
\end{proof}

\begin{lemm}\label{smooth-isotopy-lemma1}
Let two smooth Legendrian closed curves $L_0$ and $L_1$ have a common core. Then $L_0$ and $L_1$ are smoothly Legendrian isotopic.
\end{lemm}

\begin{proof}
Let $S$ be the surface of the core. By Lemma~\ref{core-curve-lemma} the curves $\mathrm{pr}(L_0)$ and $\mathrm{pr}(L_1)$ are isotopic inside $S$ to the core curve of $S$. Let $\{\gamma_t\}_{t\in[0;1]}$ be a smooth isotopy between $\gamma_0=\mathrm{pr}(L_0)$ and $\gamma_1=\mathrm{pr}(L_1)$.

By Proposition~\ref{r-invariant-regular-neighbourhood-existence} there exists a common regular neighborhood $U$ of the curves $L_0$ and $L_1$ with $r=+\infty$ which contains the core. 

Suppose that $S$ is a M\"obius band. By the definition of the core, 
% The contact field of line elements is transverse to $S$. So this field does not split into a pair of vector fields, otherwise $S$ will be coorientable, thus orientable since the contact manifold is oriented. The contact field is also transverse to the contact structure. This means that 
 the contact structure is not coorientable on $L_0$. Since $\gamma_0 = \mathrm{pr}(L_0)$ is isotopic to $L_0$, the contact structure is not coorientable on $\gamma_0.$ The same holds for all curves $\gamma_t.$ This means that for any $t$ the surface $\mathrm{pr}^{-1}(\gamma_t)$ is homeomorphic to a M\"obius band and the foliation on it given by the intersection with the contact planes consists of closed leaves only, and there is the unique leaf that projects to $S$ injectively (the projection of any other leaf is a double covering). This leaf is the unique closed Legendrian curve $L_t\subset U$ such that $\mathrm{pr}(L_t)=\gamma_t.$ Then $\{L_t\}_{t\in[0;1]}$ is a sought-for isotopy.

If $S$ is an annulus, by Remark~\ref{regular-neighbourhood-contact-form} the contact structure on $\mathrm{pr}^{-1}(S)$ is the kernel of the 1-form $dz+\mathrm{pr}^*\beta$ (see Subsection~\ref{regular-projection-section}). Then we can continuously deform the isotopy $\{\gamma_t\}_{t\in[0;1]}$ in such a way that the area with respect to the area form $d\beta$ of the connected components of $S\setminus \gamma_t$ is independent of $t$. This can be done as in Subsection~\ref{correct-integral-section} but instead of using a correction square we use the whole annulus $S$ to correct the integral. Here we do not need the results of that section in full generality because all is smooth. After the integral is corrected, we can lift the isotopy to a Legendrian isotopy of closed Legendrian curves.
\end{proof}

\begin{lemm}\label{smooth-isotopy-lemma4}
Let $l$ be a contact field of line elements transverse to the contact structure.
Let for each $j = 0$ and $j = 1$ the closed Legendrian curves $L_j$ and $L$ have a common core with the field $l$ and $L_j$ be smooth. Then the curves $L_0$ and $L_1$ are smoothly Legendrian isotopic.
\end{lemm}

\begin{proof}
For each $j=0,1$ let $K_j$ and $S_j$ denote the core for $L_j$ and the surface of the core respectively. Let $(U, S, l, r)$ be a regular neighborhood of $L$ such that $U\subset K_{0}\cap K_{1}.$ The projections $\mathrm{pr}_{j}:S\to S_j$ are local homeomorphisms and they are injective on the compact subset $\mathrm{pr}(L).$  Therefore there exists an open neighborhood $\widetilde S\subset S$ of $\mathrm{pr}(L),$ such that the restrictions of the projections $\mathrm{pr}_{0}$ and $\mathrm{pr}_{1}$ on this neighborhood are open embeddings. We set $\widetilde U = \exp_{l}(B_{r l}(\widetilde S)).$ It is clear that $(\widetilde U,\widetilde S, l, r)$ is a regular neighborhood of $L$. Let $K'$ be its core and $S'$ be the surface of the core. By~(\ref{core-P1}) in the definition of the core the maps $\mathrm{pr}_j:S' \to S_j$ are embeddings.

By Lemma~\ref{smooth-isotopy-lemma5} there exists a smooth closed Legendrian curve $L'$ having the common core $K'$ with $L$. Let us show that $K_j$ is a common core of $L'$ and $L_j$. The map $\mathrm{pr}_j:S' \to S_j$ is an embedding, $\mathrm{pr}_j\big|_{L'} = \mathrm{pr}_j\big|_{S'}\circ\mathrm{pr}'\big|_{L'}$  and the map $\mathrm{pr}'\big|_{L'}$ is injective, hence $\mathrm{pr}_j\big|_{L'}$ is injective.
%The curves $\mathrm{pr}'(L)$ and $\mathrm{pr}'(L')$ are isotopic in $S'$ (to the core curve of $S'$). Hence the curves $\mathrm{pr}_j (L)$ and $\mathrm{pr}_j (L')$ are isotopic in $S_j$. Since $K_j$ is a core for $L$, the curve $\mathrm{pr}_j(L)$ is isotopic to the core curve of $S_j$ (in $S_j$). Hence the curve $\mathrm{pr}_j(L')$ is isotopic to the core curve of $S_j.$ 
By Lemma~\ref{smooth-isotopy-lemma6} $K_j$ is a common core of $L'$ and $L_j.$

Since $K_0$ is a common core of $L_0$ and $L'$, and $K_1$ is a common core of $L_1$ and $L'$, by Lemma~\ref{smooth-isotopy-lemma1} the curves $L_0, L_1$ are smoothly Legendrian isotopic.
\end{proof}

\begin{lemm}\label{smooth-isotopy-lemma2}
Let $\{l_s\}_{s\in[0;1]}$ be a smooth family of contact fields of line elements on a contact manifold. Let $(K_s, S, r)$ be a core of some regular neighborhood of the Legendrian curve $L$ with the contact field $l_s$ for any $s\in[0;1].$ Let for each $s=0$ and $s=1$ the Legendrian curves $L_s$ and $L$ have the common core $K_s.$ Let the curves $L_0$ and $L_1$ be smooth. Then the curves $L_0$ and $L_1$ are smoothly Legendrian isotopic.
\end{lemm}

\begin{proof}
Let us fix $s_0\in[0;1].$ By Lemma~\ref{smooth-isotopy-lemma5} there exists a smooth Legendrian curve $L_{s_0}$ having the common core $K_{s_0}$ with $L$. Let us prove that there exists an interval $J\subset[0;1]$ containing $s_0$ such that $K_s$ is a common core of $L_{s_0}$ and $L$ for any $s\in J.$

The surface $S$ and the number $r$ are fixed, the field $l_s$ continuously depends on $s$, $L_{s_0}\subset K_{s_0}\setminus\partial K_{s_0}$ and $K_s = \exp_{ l_s}\left(\overline B_{r l_s}(S)\right)$. This means that $L_{s_0}\subset K_s\setminus\partial K_s$ for $s\in J$ where $J$ is some interval containing $s_0.$ Since $L_{s_0}$ is a smooth Legendrian curve and the field $l_{s_0}$ is transverse to the contact planes, the projection $\mathrm{pr}_{s_0}\big|_{L_{s_0}}$ is regular. Since the map $\mathrm{pr}_{s_0}\big|_{L_{s_0}}$ is injective and regular, any its $C^1$-small deformation is also injective. Hence we can shorten the interval $J$ to make the map $\mathrm{pr}_{s}\big|_{L_{s_0}}$ injective for any $s\in J.$ By Lemma~\ref{smooth-isotopy-lemma6} $K_s$ is a common core.

Since $[0;1]$ is compact, there exist a sequence $0 = s_0 < s_1 < \dots < s_M = 1,$ an open cover $[0;1] = \bigcup\limits_{i = 1}^M J_i$ and a sequence of smooth Legendrian curves $\{L_{s_i}\}_{i = 1}^{M-1}$ such that for any $i = 1,\dots, M$:  $\{s_{i-1}, s_i\}\subset J_i$ and for any $s\in J_i$ the curves $L_{s_i}$ and $L$ have the common core $K_s.$

Therefore for any $i = 0,\dots, M-1$ the curves $L_{s_{i}}$ and $L_{s_{i+1}}$ have the common core $K_{s_i}.$ So by Lemma~\ref{smooth-isotopy-lemma1} the curves $L_{s_{i}}$ and $L_{s_{i+1}}$ are smoothly Legendrian isotopic. Concatenating all constructed isotopies we obtain that $L_0$ and $L_1$ are Legendrian isotopic as smooth Legendrian curves.
\end{proof}

\begin{defi}
Let $l_0$ and $l_1$ be contact fields of line elements transverse to the contact structure. By the {\it affine combination of the fields $l_0$ and $l_1$ with the coefficient} $s\in[0;1]$ we call the field defined in any simply connected open set by the equation

$$
l(l_0, l_1, s) := \{v, -v\}, \text{ where } v= (1-s)v_0 + s v_1,\ l_0 = \{v_0, -v_0\},\ l_1 = \{v_1, -v_1\}
$$
and
$v_0$ and $v_1$ are directed to the same side of the contact plane. The field $l(l_0, l_1, s)$ is contact and is transverse to the contact structure for any $s\in[0;1].$ 
\end{defi}

\begin{lemm}\label{smooth-isotopy-lemma3}
Let for each $j = 0$ and $j = 1$ the closed Legendrian curves $L_j$ and $L$ have a common core and $L_j$ be smooth. Then the curves $L_0$ and $L_1$ are smoothly Legendrian isotopic.
\end{lemm}
\begin{proof}
Denote the contact fields by $l_0$ and $l_1.$  For any $s\in[0;1]$ we define a field $l_s = l(l_0, l_1, s)$ in the intersection of the interiors of these two cores of $L$.

Let us fix $s_0\in[0;1].$ By Lemma~\ref{regular-neighbourhood-existence-lemma} the curve $L$ has a regular neighborhood with the field $l_{s_0}.$ Let $(K, S, r)$ denote a core of this regular neighborhood. We claim that there exists a neighborhood $J$ of $s_0$ such that for any $s\in J$ the compact set $K_{s} = \exp_{l_s}(\overline B_{rl_{s}}(S))$ is a core of some regular neighborhood of $L$ with the field $l_s.$ It is sufficient to construct such interval for each condition in Definition~\ref{regular-neighbourhood-core} and then take the intersection of these intervals.

The map $\exp_{l_s}\big|_{\overline B_{rl_{s}}(S)}$ is a smooth embedding for any $s$ in some neighborhood of $s_0$ because it is a smooth embedding for $s = s_0$, $\overline B_{rl_{s}}(S)$ is compact and $l_s$ depends smoothly on $s$.

The curve $L$ is contained in $K_s\setminus\partial K_s$ for any $s$ in some neighborhood of $s_0$ by continuity of $l_s$ on $s$. We prove that $\mathrm{pr}_s\big|_L$ is injective for $s$ sufficiently close to $s_0$. Suppose the contrary. Then there exist sequences $P_i\in L,$ $Q_i\in L,$ $s'_i\in[0;1]$ such that the points $P_i$ and $Q_i$ are connected in $K_{s'_i}$ by a trajectory of the field $l_{s'_i}$ and $s'_i\to s_0$ as $i\to+\infty.$ Since $L$ is compact, we can assume that $P_i\to P,$ $Q_i\to Q.$ Since $\mathrm{pr}_{s_0}\big|_L$ is injective, $P = Q.$ By the definition of a Legendrian curve for any local coordinate system at the point $P = Q$ there exists a neighborhood of this point such that the angle between the segment joining any two points of the curve in this neighborhood and the contact plane is small. This contradicts the fact that $l_{s_0}$ is transverse to the contact plane.

So by Lemma~\ref{smooth-isotopy-lemma6} $K_s$ is a core of some regular neighborhood of $L$ for any $s\in J$ where $J$ is some interval containing $s_0.$

Then by compactness of $[0;1]$ there exist a sequence $0 = s_0 < s_1 < \dots < s_M = 1$, an open cover $[0;1] =  \bigcup\limits_{i=1}^M J_i$, a sequence of surfaces $\{S_i\}_{i = 1}^M$ and a sequence of real numbers $\{r_i\}_{i = 1}^M$ such that for any $i = 1,\dots, M:$ $\{s_{i-1}, s_i\}\subset J_i$ and for any $s\in J_i$ the compact set $K_{i,s} = \exp_{l_s}(\overline B_{r_il_s}(S_i))$ is a core of a regular neighborhood of $L$ with the field $l_s$.

Therefore for any $i = 1,\dots,M-1$ the compacts $K_{i, s_{i}}$ and $K_{i+1, s_{i}}$ are cores of regular neighborhoods of $L$ with the same field $l_{s_i}.$

By Lemma~\ref{smooth-isotopy-lemma5} for any $i = 1,\dots, M$ there exist smooth Legendrian curves $L_i'$ and $L_{i}''$ such that $K_{i, s_{i-1}}$ is a common core of $L_i'$ and $L$, and $K_{i, s_{i}}$ is a common core of $L_i''$ and $L$. By Lemma~\ref{smooth-isotopy-lemma4} the curves $L_i''$ and $L_{i+1}'$ are smoothly Legendrian isotopic, and by Lemma~\ref{smooth-isotopy-lemma2} the curves $L_i'$ and $L_{i}''$ are smoothly Legendrian isotopic. So the curves $L_1'$ and $L_M''$ are smoothly Legendrian isotopic.

By Lemma~\ref{smooth-isotopy-lemma4} $L_0$ is smoothly Legendrian isotopic to $L_1'$, and $L_M''$ is smoothly Legendrian isotopic to $L_1.$ So $L_0$ is smoothly Legendrian isotopic to $L_1.$
\end{proof}

\begin{proof}[Proof of Proposition~\ref{smooth-Legendrian-isotopy}]
Let $L_0$ and $L_1$ be smooth Legendrian links and let $\{L_t\}_{t\in[0;1]}$ be a Legendrian isotopy.

For a component $c$ of the link $L_0$ let $c_t$ denote the corresponding component of the link $L_t.$  

Let $t_0\in[0;1]$ and $K$ be a core of some regular neighborhood of some connected component $c_{t_0}$ of the link $L_{t_0}.$ We claim that $K$ is a common core of the curves $c_{t_0}$ and $c_t$ for $t\in I$ where $I$ is some neighborhood of $t_0.$ It is clear that $c_{t}$ lies in the interior of $K$ for $t$ sufficiently close to $t_0$. Let $S$ be the surface of the core $K$. By Lemma~\ref{smooth-isotopy-lemma6} it only remains to prove that for $t$ sufficiently close to $t_0$ the curve $c_t$ is projected to $S$ along the trajectories of the contact field injectively. To show this we take some point $P\in c_{t_0}$ and choose such Euclidean coordinates in the neighborhood of $P$ that the orthogonal projection to the contact plane at $P$ coincides with the projection along trajectories. This is possible because the contact field is transverse to contact planes. We choose a neighborhood $V$ of $P$ such that the intersection of $V$ with each trajectory is either an interval or the empty set. We can assume that the coordinates are defined on the whole $V$ because we can take smaller $V$. By the definition of Legendrian isotopy (Definition~\ref{def-angle}) there exists a smaller neighborhood $U$ of $P$ and a neighborhood $I$ of $t_0$ such that for all moments $t\in I$ the projection of the set $c_t\cap U$ to the contact plane is injective. Since the intersection of $V$ with each trajectory is connected, the projection of $c_t\cap U$ to $S$ along trajectories is injective. Since $c_{t_0}$ is compact, we can cover it by a finite number of such neighborhoods and then intersect all corresponding intervals of time.

Since $[0;1]$ is compact, by the discussion above we have the following. There exist a sequence $0 = t_0 < t_1 < \dots < t_N = 1$ and a cover $[0;1] = \bigcup\limits_{j=1}^N I_j$ by connected open subsets such that $\{t_{j-1}, t_j\}\subset I_{j}$, for any component $c$ of the link $L_0$ and for any $j = 1,\dots, N$ there exist a sequence of compacts $\{K_j^c\}_{j=1}^N$ such that $K_j^c$ is a core of some regular neighborhood of the curve $c_t$ for $t\in I_j$.

By choosing the compacts $K_j^c$ sufficiently small, we can also assume that for any $j = 1, \dots, N-1$ the sets $K_j^c\cup K_{j+1}^c$ do not intersect each other for distinct $c$.

By Lemma~\ref{smooth-isotopy-lemma5} for any component $c$ and any $j \in \{1,\dots,N\}$ there exists a smooth Legendrian curve $L_j^c$ which has a regular neighborhood with the core $K_j^c.$

We see that for $j \in \{1,\dots,N\}$ the curves $c_{t_{j}}$ and $L^c_j$ have a common core $K^c_j$ and for $j\in \{0,\dots,N-1\}$ the curves $c_{t_{j}}$ and $L^c_{j+1}$ have a common core $K^c_{j+1}.$ By Lemma~\ref{smooth-isotopy-lemma3} the curves $L^c_j$ and $L^c_{j+1}$ are smoothly Legendrian isotopic for $j\in\{1,\dots,N-1\}.$ This isotopy can be chosen in a sufficiently small neighborhood of $K^c_j\cup K^c_{j+1}$ so that such neighborhoods do not intersect each other for distinct $c$. So this isotopy provides a smooth Legendrian isotopy between Legendrian links $\bigcup\limits_{c\in\pi_0(L_0)} L_j^c$ and $\bigcup\limits_{c\in\pi_0(L_0)} L_{j+1}^c.$

Since the curves $c_0$ and $L^c_1$ have a common core $K^c_{1},$ they are smoothly Legendrian isotopic by Lemma~\ref{smooth-isotopy-lemma1}. For each $c$ the isotopy can be chosen to be supported in a sufficiently small neighborhood of $K^c_1$ such that the union of these isotopies provides an isotopy between $L_0$ and $\bigcup\limits_{c\in\pi_0(L_0)} L_{1}^c.$ The same holds for the curves $c_1$ and $L^c_N.$

Concatenating all constructed smooth isotopies we obtain a smooth Legendrian isotopy between $L_0$ and $L_1.$
\end{proof}

\begin{rema}
Propositions~\ref{smoothing-legendrian-link} and~\ref{smooth-Legendrian-isotopy} constitutes Theorem~\ref{main-theorem}. Using this theorem we can associate canonically an equivalence class of smooth Legendrian links with a given Legendrian Lavrentiev link. Let us emphasize that this is not just a correspondence. Legendrian isotopies relate the result of smoothing with the initial link in a geometrical manner not depending on a coordinate system.
\end{rema}

\begin{coro}\label{Lipschitz-C-bi-Lipschitz-isotopy}
Let $L_0$ and $L_1$ be Legendrian links in the contact manifold $M$. If they are Legendrian isotopic, there exist a real number $C$ and a Legendrian isotopy $F:L_0\times[0;1]\to M$ from $L_0$ to $L_1$ such that $F$ is Lipschitz and the map $F(\bullet, t): L_0 \to M$ is $C$-bi-Lipschitz for any $t\in[0;1].$
\end{coro}

\begin{proof}
By Proposition~\ref{smoothing-Legendrian-arc} each component of the link $L_0$ can be smoothed to a link $L_0'$. This smoothing can be done in pairwise disjoint neighborhoods of the components so it gives a Legendrian isotopy of the link. This isotopy is Lipschitz and all curves in the isotopy are $C$-bi-Lipschitz for a common $C$. Let $L_1'$ be a similar smoothing for the link $L_1.$ By Proposition~\ref{smooth-Legendrian-isotopy} the links $L_0'$ and $L_1'$ are smoothly Legendrian isotopic. So there exists a smooth Legendrian isotopy between the links $L_0'$ and $L_1'$. It is automatically Lipschitz and all curves in this isotopy are $C$-bi-Lipschitz for some $C$ by Lemma~\ref{Lavrentiev-C1-isotopy}. The concatenation of the considered isotopies is a sought-for Legendrian isotopy.
\end{proof}

\section{Continuous Legendrian curves}
\label{Continuous-Legendrian-curves-section}
It is natural to try to define Legendrian continuous curves in the same way as in Definition~\ref{def-angle} but replacing the phrase "locally Lavrentiev" by "continuous". Actually it was the starting point of the present work. In this section we discuss a problem which occurs on this way.

We start with the following equivalent definition that was introduced to me by Ivan Dynnikov. Formally, we do not use this definition in this paper. But we find this definition illuminating.

\begin{defi}
\label{def-projection}
A continuous curve $L$ is called {\it Legendrian} if for any point $p\in L$ and any Euclidean coordinates of class $C^1$ in a neighborhood of $p$ there exists a smaller neighborhood $U\ni p$ such that the orthogonal projection to the contact plane at $p$ is injective on $L\cap U.$

A continuous isotopy of Legendrian curves $\{L_t\}_{t\in[0;1]}$ is called {\it Legendrian} if for any moment $t_0\in[0;1]$, for any point $p\in L_{t_0}$ and any Euclidean coordinates of class $C^1$ in a neighborhood of $p$ there exist a smaller neighborhood $U\ni p$ and an interval $I\ni t_0$ such that for any moment $t\in I$ the orthogonal projection to the contact plane at $p$ is injective on $L_t\cap U.$
\end{defi}

\begin{prop}
\label{leg-criteria1}
Definition~\ref{def-projection} and Definition~\ref{def-angle} with "locally Lavrentiev" replaced by "continuous" are equivalent.
\end{prop}

\begin{proof}
First we prove that if the curve (isotopy) is Legendrian by Definition~\ref{def-angle} then it is Legendrian by Definition~\ref{def-projection}.

Suppose that Euclidean coordinates in a neighborhood of $p\in L$ ($L_{t_0}$) are set. Then by Definition~\ref{def-angle} there exists a neighborhood $U$ of $p$ (and a time interval $I\ni t_0$) such that the angle between the contact plane $\xi_p$ and the vector $p''-p'$ is less than $\pi/2$ for any two points $p', p'' \in L\cap U$ ($\in L_t\cap U$ for $t\in I$). Therefore in $U$ the orthogonal projection to the contact plane $\xi_p$ is injective on the curve $L$ ($L_t$ for $t\in I$).

Now let $L$ (the isotopy $L_t$) be Legendrian by Definition~\ref{def-projection}. We prove by the contrary that it is Legendrian by Definition~\ref{def-angle}.

Suppose that for some $\varepsilon > 0$ for some local Euclidean coordinates at the point $p$ on the curve $L$ ($L_{t_0}$) there exist sequences $p_n', p_n''$ of points on the curve $L$ ($L_{t_n}$) such that
\begin{enumerate}
\item[\it 1.] $p_n' \to p$ and $p_n''\to p$ if $n\to\infty$.
\item[\it 2.] $\angle(p_n'-p_n'', \xi_p) > \varepsilon.$
\item[({\it 3.})] $t_n\to t_0$ if $n\to\infty$.
\end{enumerate}

By Lemma~\ref{curve-through-infinite-number-of-points} there exists a curve $\gamma_1 (t)$ of class $C^1$ which passes through the points $p$ and $p'_{n_i}$ where $n_i$ is some increasing sequence. If $\gamma_1(t)$ is transverse to $\xi_p$, there exist local Euclidean coordinates at $p$ such that the contact plane $\xi_p$ is the $xy$-plane and $\gamma_1(t)$ is $z$-axis. So the orthogonal projection to $\xi_p$ is not injective in any neighborhood of $p$ (and any neighborhood of $t_0$), a contradiction. Therefore $\gamma_1(t)$ is tangent to $\xi_p.$

By Lemma~\ref{curve-through-infinite-number-of-points} there exists a curve $\gamma_2 (s)$ of class $C^1$ which passes through the points $p''_{n_i} - p'_{n_i}$ for an infinite number of indices $i$. For any $t$ the curve $\gamma_1(t) + \gamma_2(s)$ of parameter $s$ is transverse to contact planes in some neighborhoods of $p$. Since the curve $\gamma_1(t)$ is tangent to $\xi_p$, for distinct $t$ the curves $\gamma_1(t) + \gamma_2(s)$ of parameter $s$ are disjoint in some neighborhood of $p$. Moreover, there exist Euclidean coordinates such that these curves are straight lines parallel to $z$-axis and $\xi_p$ is $xy$-plane. So the orthogonal projection to $\xi_p$ is not injective in any neighborhood of $p$ (and any neighborhood of $t_0$), a contradiction.
\end{proof}

\begin{lemm}
\label{curve-through-infinite-number-of-points}
Let $\{p_n\}_{n\in\mathbb N}$ be a sequence of points which converges to the point $p.$ Then there exists a $C^1$-smooth curve which passes through $p$ and through $p_n$ for an infinite number of indices.
\end{lemm}

\begin{proof}
We can assume that no point $p_n$ coincide with $p$. Also we can assume that all points lie in Euclidean space $\mathbb R^3.$

There exists a sequence $n_i$ such that $(p_{n_i}-p)/|p_{n_i}-p|$ converges to some vector $v$ because the sphere is compact. We can assume that $n_i = i.$

There exists a non-compact embedded piecewise linear curve such that any its breaking point belongs to the sequence $\{p_n\}_{n\in\mathbb N}$ and the direction of its edge tends to $v$. We can construct such curve by adding inductively to the curve a straight line segment such that its added endpoint is sufficiently close to $p$. Angles at breaking points can be smoothed in such a way that in a neighborhood of any breaking point the unit vector tangent to the obtained smooth curve moves along the geodesic on the sphere from the direction of the previous edge to the direction of the next edge. The compactification of this curve is $C^1$-smooth because the  unit tangent vector tends to $v$.
\end{proof}

\subsection{A counterexample}

First, we construct a family of Legendrian continuous curves having a common point and having the same regular projection.

The regular projection of the curve that we will construct is a curve in the unit square. This curve was introduced by Lance and Thomas in~\cite{LT}. Let $\{s_n\}_{n\in\mathbb N}$ be a sequence of real numbers such that $0<s_n<1$ for any $n\in\mathbb N.$ Let $S_1$ be a set containing the unique element which is a square whose side has the length $a_1.$ Let $C_1$ denote a set containing the unique element which is the cross lying symmetrically in the square from $S_1$ which touches each side of the square along the segment of length $s_1\cdot a_1.$ With the cross from $C_1$ we associate three segments inside this cross as in Figure~\ref{cross-with-three-segments}.

\begin{figure}[h]
\centering
\includegraphics[scale=1]{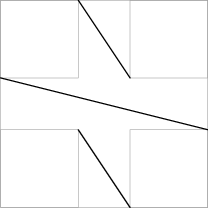}
\caption{A cross with three segments}
\label{cross-with-three-segments}
\end{figure}

Let $S_2$ be four squares which constitutes the complement to the cross from $C_1$ inside the square from $S_1.$ Then we iterate this procedure for each square in $S_2.$ For each square from $S_2$ we construct a similar cross which touches the side of the square along the segment of length $s_2\cdot a_2$ where $a_2 = (1 - 2s_1)\cdot a_1$ and $a_2$ is the length of the side of each square from $S_2$. In this way for each $n\in\mathbb N$ we obtain a set $S_n$ of squares whose sides have length $a_n$ and a set $C_n$ of crosses. The cross from $C_n$ contained in the square $X\in S_n$ we denote by $C(X).$ For each $n\in\mathbb N$ with each cross from $C_n$ we associate three segments as in Figure~\ref{cross-with-three-segments}.

The closure of the union of all segments associated with constructed crosses is an embedded curve. This curve is also the union of all associated segments and the Cantor set $\bigcap\limits_{n=1}^{\infty}\bigcup\limits_{X\in S_n} X.$ The measure of this curve is positive if and only if $\sum\limits_{n=1}^{\infty}s_n < \infty.$ We denote this curve by $\gamma.$ Denote by $\gamma_n$ the curve $\left(\gamma\setminus \bigcup\limits_{X\in S_n}X\right)\cup \bigcup\limits_{X\in S_n} X_{\diagup}$ where $X_{\diagup}$ denotes the diagonal of the square $X$ joining the left bottom corner with the right top corner. The curves $\gamma_n$ are shown in Figure~\ref{curve-of-positive-measure} for $n = 1,2,3,4.$

\begin{figure}[h]
\centering
\includegraphics[scale=1]{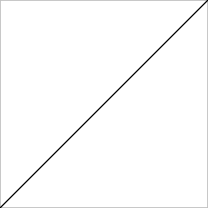}
\includegraphics[scale=1]{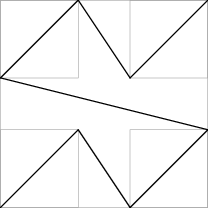}
\includegraphics[scale=1]{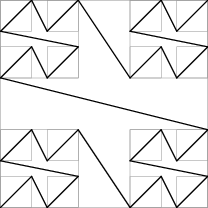}
\includegraphics[scale=1]{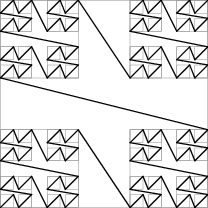}
\caption{Approximating the curve of positive measure}
\label{curve-of-positive-measure}
\end{figure}

Suppose that $\sum\limits_{n=1}^{\infty}s_n < \infty$ and $a_n/s_n\to0$ if $n\to\infty.$ These conditions are satisfied, for example, if $s_n = 1/(n+1)^2.$ Consider the contact manifold $(\mathbb R^3, \ker(dz-ydx))$ and suppose that the square from $S_1$ is the square $[0;a_1]\times[0;a_1]\times\{0\}.$ For any non-negative real number $K$ we construct a Legendrian curve $\widetilde\gamma (K)$ in $(\mathbb R^3, \ker(dz-ydx))$ whose orthogonal projection to $x,y$-plane coincides with $\gamma$ and which passes through the points $(0, 0, 0)$ and $(a_1, a_1, (K+1/2) a_1^2).$

First we define $\widetilde\gamma(0).$ We will define it as a graph of a function $z$ on $\gamma.$ The curve $\gamma_n$ is piecewise linear, so there is the unique Legendrian curve $\widetilde \gamma_n$ passing through the origin whose projection is $\gamma_n.$  Consider the $z$-coordinate of the point $p\in\widetilde\gamma_n$ as a function on $\gamma_n$ and denote it by $z_n(\mathrm{pr}(p)).$ It is a direct check that $z_1(a_1,a_1) = z_2(a_1,a_1).$ A similar argument shows that $z_n(p) = z_{n+m}(p)$ for any $m\ge 0$ where $p$ is the left bottom corner of any square from $S_n.$ So we define $z(p) = z_n(p).$ These corners constitute a dense subset of $\gamma.$ So we can define $z$ by continuity if we prove that for any square $X\in S_n$ the range of $z_{n+m}$ on $X$ is bounded above by something that is independent of $m$ and tends to zero if $n$ tends to infinity. 

Suppose that $p$ is the left bottom corner of $X\in S_n,$ $q$ is the left bottom corner of $Y\in S_{n+m}$ and $X\supset Y.$ Then we take a sequence $p = p_0, p_1, \dots, p_m = q$ of the left bottom corners of the squares $X = X_0, X_1, \dots, X_m = Y$ respectively such that $X_k\in S_{n+k},$ $X_k\supset X_{k+1}.$ Then
$$
z(q) - z(p) = \sum\limits_{k = 1}^m (z(p_k) - z(p_{k-1})) = \sum\limits_{k = 1}^m (z_{n+k}(p_k) - z_{n+k}(p_{k-1})).
$$

On any Legendrian curve we have $z_1 - z_0 = \int\limits_{t_0}^{t_1} y(t)dx(t).$ Also $|y|\le a_1$ and the total variation of $x$ on $\gamma_{n+k}\cap X_{k-1}$ is $3a_{n+k-1}$, so

$$
|z(q) - z(p)| \le \sum\limits_{k = 1}^m (3a_1a_{n+k-1}) < \sum\limits_{k = 1}^m (3a_1 a_n 2^{1-k}) < 6 a_1 a_n.
$$

So the function $z$ is well defined on $\gamma$.

Let $\{k_n\}_{n\in\mathbb N}$ be a sequence such that $k_1 = K$ and $k_{n+1} = k_n / (1-s_n)^2$ for each $n\in\mathbb N.$ It is clear that the sequence $k_n$ converges. We define the curve $\widetilde\gamma(K)$ as a graph of the function $z + \Delta z$ on $\gamma$ where $\Delta z$ is the unique monotonic continuous function such that
\begin{itemize}
\item $\Delta z$ is constant on each associated segment.
\item $\Delta z(0,0) = 0.$
\item For any $n\in\mathbb N$ for each square from $S_n$ we have $\Delta z(q) - \Delta z(p) = k_n a_n^2$ where $p$ and $q$ are the left bottom and the right top corners of the square respectively.
\end{itemize} 

%The last condition is compatible with itself for distinct $n$ because $k_{n+1} = k_n/(1-s_n)^2.$ So we can define $\Delta z$ on the union of all associated segments. Since it is a dense subset of $\gamma$, we only need to prove that $\Delta z$ is uniformly continuous on this subset. This is true because by construction the range of $\Delta z$ on each square from $S_n$ equals $k_n a_n^2$ which tends to 0 if $n\to\infty.$ We note that $\Delta z$ is a monotonic function.

If $(x,y,z)\in \widetilde\gamma(K)$, $z = z(x,y) + \Delta z(x,y).$

Let $X\in S_n,$ $(x,y)\in\gamma\cap X,$ $(x_0,y_0)$ be the left bottom corner of $X$. Let us prove that
\begin{equation}
\label{z-bound}
|z(x,y) - z(x_0,y_0) - (x-x_0)y_0 | \le a_n^2/2.
\end{equation}

Let $f(x,y) = z(x,y) - z(x_0,y_0) - (x-x_0)y_0$ be a function on $\gamma\cap X.$ We note that $f(x_1, y_1) = \int\limits_{t_0}^{t_1} (y(t)-y_0)dx(t)$ where $x_i = x(t_i)$ and $y_i = y(t_i)$ for $i = 0,1.$ Let us prove that the right top corner $(x_1,y_1)$ of the square $X$ is the point where the function $f$ achieves its maximum. Consider the right top corners of four squares from $S_{n+1}$ which are contained in $X.$ Since $z = z_{n+1}$ at these points, it is the direct check that the maximum value of $f$ among these points is being achieved at $(x_1, y_1).$ Let $X = X_1 \supset X_2 \supset \dots$ be a sequence of squares such that $\bigcap\limits_{k=1}^{\infty}X_k = \{(x,y)\}$ and $X_k\in S_{n+k}$ for any $k = 1,2,\dots.$ Let $(x_k, y_k)$ be the right top corner of $X_k.$ A similar argument shows that $f(x_1,y_1) \ge f(x_2,y_2) \ge \ldots.$ Therefore $f(x_1,y_1)\ge f(x,y).$ Since $f(x_1,y_1) = a_n^2/2,$ we get
$$
z(x,y) - z(x_0,y_0) - (x-x_0)y_0 \le a_n^2/2.
$$

In a similar way one can prove that $f(x_0, y_1)$ is the minimum of $f$ and that $f(x_0, y_1)> -a_n^2/2.$ So the inequality~(\ref{z-bound}) is proved.

Then we prove that $\widetilde\gamma(K)$ is Legendrian. Let $p(x,y,z)$ be a point on $\widetilde\gamma(K).$ Let us find a neighborhood of the point $p$ where condition~(\ref{angle-condition}) of Definition~\ref{def-angle} is satisfied.

\vspace{3mm}

\noindent{\bf Case 1.}
If $(x,y)$ lies in the interior of some associated segment, the curve $\widetilde\gamma(K)$ is a smooth Legendrian curve in some neighborhood of $p$. So this case is obvious.

\vspace{3mm}
\noindent{\bf Case 2.} Suppose that $(x,y)$ is neither the left bottom corner nor the right top corner of any square from $S_n$ for any $n$ and that $(x,y)$ does not lie on any associated segment. Let $p'(x',y',z')$ and $p''(x'',y'',z'')$ be two points on $\widetilde\gamma(K).$ Let $n$ be the largest number such that the points $(x',y')$ and $(x'', y'')$ lie in some square $X\in S_{n}.$ Let $m$ be the largest number such that the points $(x,y),$ $(x',y')$ and $(x'', y'')$ lie in some square from $S_m.$  Let $(x_0, y_0)$ be the left bottom corner of $X$ and $p_0(x_0,y_0,z_0)$ be the corresponding point on $\widetilde\gamma(K).$ We need to prove that $\angle(p''-p',\xi_p)\to 0$ if $m\to\infty.$
$$
\angle(p''-p',\xi_p) \le \angle(p''-p',\xi_{p_0}) + \angle(\xi_{p_0},\xi_p).
$$

Since $\xi$ is continuous, $\angle(\xi_{p_0},\xi_p)\to0$ if $m\to\infty.$ So it is sufficient to prove that $\angle(p''-p',\xi_{p_0})\to 0$ if $m\to\infty.$

\begin{equation}
\label{angle-bound}
\sin\angle(p''-p',\xi_{p_0}) = \left|-(x''-x')y_0 + z''-z'\right|/|p''-p'|\sqrt{y_0^2+1} < \left|-(x''-x')y_0 + z''-z'\right|/|p''-p'|.
\end{equation}

\noindent{\bf Case 2a.} If $(x',y')$ and $(x'',y'')$ lie in the union of squares from $S_{n+1}$, $|p''-p'|\ge s_na_n.$ Recall that $z' = z(x',y') + \Delta z(x', y')$ and $z'' = z(x'',y'') + \Delta z(x'', y'')$. Hence continuing inequality~(\ref{angle-bound})

\begin{multline*}
\sin\angle(p''-p',\xi_{p_0}) < \left|\left(z''-z(x_0,y_0)-(x''-x_0)y_0\right) - \left(z'-z(x_0,y_0)-(x'-x_0)y_0\right)\right|/|p''-p'| \le \\ \le \left|\Delta z(x'',y'') - \Delta z(x',y')\right|/|p''-p'| + \\ + \left|\left(z(x'',y'')-z(x_0,y_0)-(x''-x_0)y_0\right) - \left(z(x',y')-z(x_0,y_0)-(x'-x_0)y_0\right)\right|/|p''-p'|\le\\ \le (k_na_n^2 + a_n^2)/(s_na_n) = (k_n+1)a_n/s_n \to0, 
\end{multline*}
where we used estimate~(\ref{z-bound}).

\noindent{\bf Case 2b.} If $(x',y')$ and $(x'',y'')$ lie on the same segment associated with the cross $C(X)$  
$$z''-z' = \int\limits_{x'}^{x''} ydx = (x''-x')(y'+y'')/2.$$

Therefore continuing inequality~(\ref{angle-bound})
$$
\sin\angle(p''-p',\xi_{p_0}) < |x''-x'|\left((y'+y'')/2-y_0\right)/|p''-p'| < (y'+y'')/2-y_0 < a_n\to0.
$$

\noindent{\bf Case 2c.} If $(x',y')$ and $(x'',y'')$ lie on distinct segments associated with the cross $C(X),$ the estimate is similar to Case 2a.

\noindent{\bf Case 2d.} If $(x',y')$ lies inside some square $Y\in S_{n+1}$, $(x'', y'')$ lies on a segment $I$ associated with the cross $C(X)$ and $Y\cap I = \varnothing,$ the estimate is similar to Case 2a.

\noindent{\bf Case 2e.} Suppose that $(x',y')$ lies inside some square $Y\in S_{n+1}$, $(x'',y'')$ lies on a segment $I$ associated with the cross $C(X)$ and $Y\cap I = \{(x_1,y_1)\}.$ Let $l$ be the largest number such that the points $(x', y')$ and $(x_1, y_1)$ lie inside some square $Z\in S_l.$ It is clear that $X\supset Y\supset Z.$ Let $(x_0',y_0')$ be the left bottom corner of $Z.$

Let $p''' (x''', y''', z''')$ be the point on $\widetilde \gamma(K)$ such that its projection to $xy$-plane coincides with the orthogonal projection of $(x', y')$ to $I$. Let us continue bound~(\ref{angle-bound}):

\begin{multline*}
\left|-(x''-x')y_0 + z''-z'\right|/|p''-p'| = \left|-(x''-x''')y_0 + z''-z'''  -(x'''-x')y_0 + z'''-z'\right|/|p''-p'| \le \\ \le \left|-(x''-x''')y_0 + z''-z'''\right|/\sqrt{(x''-x''')^2 + (y''-y''')^2}  + \\ + \left|-(x'''-x')y_0 + z'''-z'\right|/\sqrt{(x'-x''')^2 + (y'-y''')^2}.
\end{multline*}

The first summand can be bounded in the same way as in Case 2b. Now we consider the second. We begin with the denominator. Since $\mathrm{dist}((x',y'), (x_1,y_1))\ge a_{l+1}$ and the angle between $I$ and the side of the square $Y$ is not less than $\arctan s_n$,

$$\sqrt{(x'-x''')^2 + (y'-y''')^2} > a_{l+1} \sin \arctan s_n = a_{l+1} s_n / \sqrt{1 + s_n^2 } > a_{l+1} s_n / \sqrt2.$$

Now we bound the nominator. Let $(x_1,y_1,z_1)\in\widetilde \gamma(K).$ 

\begin{multline*}
\left|-(x'-x''')y_0 + z'-z'''\right| = \left|\left(-(x'-x_1)y_0 + z'-z_1\right) + \left(-(x_1-x''')y_0 + z_1-z'''\right)\right| \le \\ \le \left|\left(-(x'-x_1)y_0 + z(x',y')-z(x_1,y_1)\right) + \Delta z(x',y') - \Delta z(x_1,y_1)\right| +\\ + \left|-(x_1-x''')y_0 + z(x_1,y_1)-z(x''',y''')\right| \le \left|-(x'-x_1)y_0 + z(x',y')-z(x_1,y_1)\right| +\\ + \left|\Delta z(x',y') - \Delta z(x_1,y_1)\right|   + \left|-(x_1-x''')y_0 + (x_1 - x''')(y_1+y''')/2\right| \le \\ \le \left|-(x'-x_1)y_0 + z(x',y')-z(x_1,y_1)\right| + k_l a_l^2 + a_l a_n,
\end{multline*}
because $|x_1 - x'''| < a_l.$ Recall that $(x_0',y_0')$ is the left bottom corner of the square $Z.$ Then we bound the first summand
\begin{multline*}
\left|-(x'-x_1)y_0 + z(x',y')-z(x_1,y_1)\right| = \left|\left(-(x'-x_1)y'_0 + z(x',y')-z(x_1,y_1)\right) + (x'-x_1)(y'_0-y_0)\right| \le\\ \le \left|\left(-(x'-x_0')y'_0 + z(x',y')-z(x_0',y_0')\right) - \left(-(x_1-x_0')y'_0 + z(x_1,y_1)-z(x_0',y_0')\right)\right| + \\ + \left|(x'-x_1)(y'_0-y_0)\right| \le a_l^2/2 + a_l^2/2 + a_l a_n
\end{multline*}
by (\ref{z-bound}) for the square $Z$. So the nominator is bounded above by $a_l^2 + 2a_la_n + k_la_l^2.$ Now we combine the estimates for the nominator and the denominator.

\begin{multline*}
\left|-(x'-x''')y_0 + z'-z'''\right|/\sqrt{(x'-x''')^2 + (y'-y''')^2} < \sqrt2\left(a_l^2 + 2a_l a_n + k_l a_l^2\right) / \left(a_{l+1} s_n\right) = \\ = \sqrt2\left(a_l^2 + 2a_l a_n + k_l a_l^2\right) / \left(a_l (1-s_l) s_n/2\right) = 2\sqrt2\left((1+k_l)a_l + 2a_n \right) / \left((1-s_l) s_n\right) < \\ < \frac{2\sqrt2 (3 + k_l)}{1-s_l} \frac{a_n}{s_n} \to 0.
\end{multline*}

\vspace{3mm}
\noindent{\bf Case 3.} Suppose that $(x,y)$ is the left bottom or the right top corner of some square from $S_n$ for some $n\in\mathbb N.$ This case is similar to Case 2.

\vspace{3mm}

So we proved that $\widetilde \gamma(K)$ is a Legendrian curve.

We can connect the ends of the curve $\gamma$ by a piecewise linear path to obtain a closed embedded curve on the plane. Then we can lift the obtained closed curve to a Legendrian curve which contains the curve $\widetilde\gamma(K).$ We can find $K$ such that the obtained Legendrian curve is closed. So we get a Legendrian unknot whose regular projection to the plane is an embedded curve. 

One can define the Thurston--Bennequin number for continuous closed Legendrian curves in the usual way. By Lemma~\ref{transverse-collar} any Legendrian curve has a collar transverse to the contact structure. By the Thurston--Bennequin number of the Legendrian curve (in $\mathbb R^3$) we call the linking number of the boundary components of the collar. This number keeps constant during Legendrian isotopies. To see this we note that Proposition~\ref{regular-neighbourhood-existence} also holds for continuous Legendrian curves, and that any regular neighborhood of the Legendrian curve serves as a regular neighborhood for all Legendrian curves obtained by a small Legendrian isotopy, see the end of the proof of Proposition~\ref{smooth-Legendrian-isotopy} for details.

The Thurston--Bennequin number of the constructed Legendrian unknot is zero. Hence by \cite{ben} it is not Legendrian isotopic to a smooth Legendrian knot in $(\mathbb R^3, dz - ydx)$. So Theorem~\ref{main-theorem} is wrong for continuous Legendrian curves.

\section{Contactomorphisms}
\label{contactomorphisms-section}
We continue to use Convention~\ref{Lavrentiev-metric-equivalence} assuming that the ambient manifold is equipped with some Riemannian metric.

%We say that a map from $M$ to itself is {\it supported} on a subset $K$ if its restriction to $M\setminus K$ is the identity. 

We equip the set of all locally bi-Lipschitz self-homeomorphisms of the manifold with the compact-open topology. If the manifold is compact, this topology is induced by the standard metric

$$
\mathrm{dist}(f,g) = \sup\limits_{p\in M} \mathrm{dist}(f(p),g(p))
.$$

\begin{lemm}\label{space-of-bi-Lipschitz-homeomorphisms-is-complete}
The metric space of all $C$-bi-Lipschitz self-homeomorphisms of a compact manifold $M$ is complete for any $C\ge1.$
\end{lemm}
\begin{proof}
Pick a Cauchy sequence $f_n.$ It converges to a surjective continuous map $f.$ For any distinct $p,q\in M$

$$\frac{\mathrm{dist}(f(p),f(q))}{\mathrm{dist}(p,q)} = \lim\limits_{n\to\infty}\frac{\mathrm{dist}(f_n(p),f_n(q))}{\mathrm{dist}(p,q)}.$$

Since $f_n$ are $C$-bi-Lipschitz, $f$ is a $C$-bi-Lipschitz homeomorphism.
\end{proof}

\begin{rema}
Bi-Lipschitz homeomorphisms are given in local charts by Lipschitz functions. The set of all bounded Lipschitz functions on the open unit ball $B\subset\mathbb R^n$ coincides with the Sobolev space $W^{1,\infty}(B)$ (see~\cite{Hein} for details). It is natural to introduce a right-invariant metric on the space of bi-Lipschitz homeomorphisms which is topologically equivalent to the one induced by the Sobolev metric:
$$\mathrm{dist}_{\mathrm{Lip}}(f,g) = \sup\limits_{p\in M}\mathrm{dist}(f(p), g(p)) + \sup\limits_{p\neq q}\left|\ln\frac{\mathrm{dist}(f(p), f(q))}{\mathrm{dist}(g(p), g(q))}\right|.
$$

However this metric is left-discontinuous. So a continuous ambient isotopy can become discontinuous after multiplication from the left which is not convenient.
\end{rema}

\begin{defi}\label{contactomorphism-def}
A locally bi-Lipschitz self-homeomorphism of a smooth contact 3-manifold is called {\it contact} if it preserves the class of compact Lavrentiev Legendrian curves. We call such a homeomorphism a {\it contactomorphism}.
\end{defi}

Recall that if $\beta$ is a 1-form on the smooth surface $S$ such that $d\beta$ is an area form, $\mathrm{pr}$ is the projection from $M = S \times \mathbb R$ to $S$ and $z$ is the projection $S\times\mathbb R\to\mathbb R,$ then the kernel of the 1-form $dz + \mathrm{pr}^*\beta$ determines a contact structure on $M$. We equip $M$ with a product of a Riemannian metric on $S$ and the Euclidean metric on $\mathbb R$.

\begin{lemm}\label{bi-Lipschitz-homeomorphism-lift}
Let $h$ be a locally bi-Lipschitz homeomorphism of the smooth connected surface $S.$ Suppose that $\int\limits_{\gamma}\beta = \int\limits_{h(\gamma)}\beta$ for any closed rectifiable curve $\gamma.$  Let $p\in S.$ Then there is the unique contactomorphism $\widetilde h$ of $M = S\times\mathbb R$ such that $\widetilde h(p, z) = (h(p), z)$ for any $z\in\mathbb R$ and $\mathrm{pr}\circ\widetilde h = h\circ\mathrm{pr}.$
\end{lemm}

\begin{proof}
Let $q\in S$ and $q\neq p.$ We connect the points $p$ and $q$ by a Lavrentiev arc $\gamma$ directed from $p$ to $q.$ By Lemmas~\ref{z-uniqueness} and~\ref{Lavrentiev-projection-lemma3} for any $z\in\mathbb R$ there exists the unique Legendrian curve having the endpoint $(q,z)$ and the projection $\gamma.$ The other endpoint of this curve is $(p, z + \int\limits_{\gamma}\beta).$ Since a contactomorphism must map Legendrian curves to Legendrian curves, we must have
$$
\widetilde h: (q, z) \mapsto \left(h(q), z + \int\limits_{\gamma}\beta - \int\limits_{h(\gamma)}\beta\right).
$$

The right side of this formula does not depend on $\gamma$ by our assumptions on $h$. So we can take this formula as the definition of $\widetilde h.$

$\widetilde h$ is locally bi-Lipschitz because the function $q \mapsto  \int\limits_{\gamma}\beta - \int\limits_{h(\gamma)}\beta$ is locally Lipschitz. $\widetilde h$ maps Legendrian curves to Legendrian curves by construction. Therefore $\widetilde h$ is a contactomorphism.
\end{proof}

\begin{exam}\label{bi-Lipschitz-contactomorphism-example}
Let $\rho, \varphi, z$ be the cylindrical coordinates on $\mathbb R^3.$ Consider the contact form $dz + \rho^2d\varphi.$ A locally bi-Lipschitz homeomorphism of $xy$-plane preserves the integral of $\rho^2d\varphi$ on any closed rectifiable curve if and only if it preserves the area form $d(\rho^2d\varphi) = 2dx\wedge dy.$ So the map 
$$(\rho, \varphi, z) \mapsto \left(\rho, \varphi + \ln\rho , z - \frac{\rho^2}2\right)$$
is a contactomorphism by the previous lemma. We note that the rays $\varphi = \mathrm{const}, z=0$ are mapped to Legendrian curves whose projections to the plane $z=0$ are logarithmic spirals.
\end{exam}

\begin{prop}\label{subspace-of-contactomorphisms-is-closed}
The space of all $C$-bi-Lipschitz contactomorphisms is closed in the space of all $C$-bi-Lipschitz homeomorphisms for any $C\ge 1.$
\end{prop}

\begin{proof}
Suppose that a sequence of $C$-bi-Lipschitz contactomorphisms $f_n$ converges to a $C$-bi-Lipschitz homeomorphism $f$. Let $L$ be any compact Lavrentiev Legendrian curve. We need to prove that $f(L)$ is Legendrian. It is clear that $f(L)$ is a Lavrentiev curve. 

The inequality~(\ref{integral-continuity-bound2}) holds actually in dimension 3 also (the same is true for Corollary~\ref{integral-continuity-corollary}). This means that the integral of the contact form on any subarc of $f(L)$ is the limit of the integral on the corresponding subarc of $f_n(L)$ which is zero because $f_n(L)$ is Legendrian. So the integral of the contact form on any subarc of $f(L)$ is zero and hence $f(L)$ is Legendrian by Proposition~\ref{leg-integral}.
\end{proof}

\begin{coro}\label{space-of-contactomorphisms-is-complete}
The metric space of all $C$-bi-Lipschitz self-contactomorphisms of a compact contact manifold $M$ is complete for any $C\ge1.$
\end{coro}

\begin{proof}
This is an immediate consequence of Lemma~\ref{space-of-bi-Lipschitz-homeomorphisms-is-complete} and Proposition~\ref{subspace-of-contactomorphisms-is-closed}.
\end{proof}

Corollary~\ref{space-of-contactomorphisms-is-complete} allows one to construct contactomorphisms as limits in $C^0$-topology.

\begin{prop}
A smooth contactomorphism in the usual sense, i.e. a diffeomorphism preserving the distribution of contact planes, is a  contactomorphism in the sense of Definition~\ref{contactomorphism-def}.
\end{prop}
\begin{proof}
Let us prove that a smooth contactomorphism preserves the class of compact Legendrian Lavrentiev curves. Recall that Legendrian curves are characterized by the fact that the integral of the contact form on any their subarc is zero. Since the inverse image under the contactomorphism of the contact form is the contact form multiplied by some smooth nonzero function, the integral along any Legendrian arc remains to be zero after applying the contactomorphism.
\end{proof}

\begin{defi}
We say that a family of maps from $X$ to $Y$ is {\it locally uniformly (bi-)Lipschitz} if for any compact $K\subset X$ there exists a number $C$ such that the restriction to $K$ of any map is $C$-(bi-)Lipschitz.
\end{defi}

\begin{defi}\label{contact-isotopy}
A locally uniformly bi-Lipschitz ambient isotopy of a contact manifold is called {\it contact} if every homeomorphism in this isotopy is a contactomorphism.
\end{defi}

\begin{lemm}\label{bi-Lipschitz-isotopy-lift}
Let $h_t$ be a locally uniformly bi-Lipschitz continuous family of homeomorphisms of a connected surface as in Lemma~\ref{bi-Lipschitz-homeomorphism-lift}. Then the lift $\widetilde h_t$ constructed in that Lemma of this isotopy is a contact isotopy.
\end{lemm}

\begin{proof}
For the proof it is only needed to check that the isotopy $\widetilde h_t$ is locally uniformly bi-Lipschitz. This is equivalent to the fact that the functions
$$
q \mapsto \int\limits_{\gamma}\beta - \int\limits_{h_t(\gamma)}\beta
$$
are locally uniformly Lipschitz. This is true because $h_t$ are uniformly Lipschitz.
\end{proof}

\begin{figure}
\center{
\includegraphics{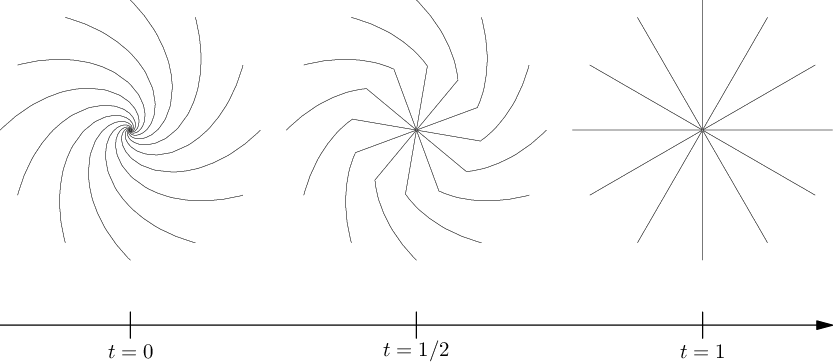}
}
\caption{The bi-Lipschitz homeomorphisms from Example~\ref{contact-isotopy-example}}
\label{contact-isotopy-example-Fig}
\end{figure}

\begin{exam}\label{contact-isotopy-example}
By Lemma~\ref{bi-Lipschitz-isotopy-lift} the following is a contact isotopy of the cylinder $\rho \le 1$ in $\left(\mathbb R^3, dz + \rho^2d\varphi\right)$:
$$(\rho, \varphi, z) \mapsto 
\left\{
\begin{aligned}
\left(\rho, \varphi + \ln\rho , z - \frac{\rho^2}{2} + \frac{t^2}{2}\right),& \quad \rho\ge t,\\
\left(\rho, \varphi + \ln t, z\right),& \quad \rho < t,
\end{aligned}
\right.
$$
where $t\in[0;1].$ See Figure~\ref{contact-isotopy-example-Fig} for the images of polar rays under the $xy$-part of this isotopy.
\end{exam}

\begin{prop}
Let $L$ be a Legendrian curve, $h_t$ be a contact isotopy. Then $h_t(L)$ is a Legendrian isotopy.
\end{prop}

\begin{proof}
By Definition~\ref{contact-isotopy} the contactomorphisms $h_t$ are bi-Lipschitz with a common constant. Therefore the curves $h_t(L)$ are Lavrentiev with a common constant. So $\{h_t(L)\}$ is a Legendrian isotopy by Proposition~\ref{leg-C-isotopy}.
\end{proof}

\end{document}